\documentclass[a4paper,11pt]{amsart}
\usepackage[vmargin=3cm,hmargin=3cm]{geometry}
\usepackage{graphicx}

\usepackage{mathpazo} 

\usepackage{xcolor}
\usepackage[
    colorlinks=true,
    linkcolor=darkblue,
    urlcolor=seablue,
    citecolor=darkgreen
    ]{hyperref}
\usepackage{orcidlink}

\usepackage[
    style=ieee,
    sorting=nyt,
    citestyle=numeric-comp,
    url=false,
    giveninits=true,
    ]{biblatex}
\AtBeginBibliography{\small}
\AtEveryBibitem{\clearname{translator}\clearlist{publisher}\clearfield{pagetotal}}

\usepackage{amsmath,amsthm,amssymb}
\allowdisplaybreaks[1]
\usepackage{mathrsfs}
\usepackage{thmtools, thm-restate}
\usepackage{quiver}
\tikzset{between/.style n args={2}{/tikz/execute at end to={
      \tikzset{spath/split at keep middle={current}{#1}{#2}}
    }}}
\usetikzlibrary{decorations.markings, graphs, shapes.geometric}
\usepackage[bbgreekl]{mathbbol}
\usepackage{mathtools}
\usepackage{stmaryrd}
\usepackage{enumitem}
\usepackage{braket}
\usepackage{bussproofs}

\usepackage{comment}

\usepackage{cleveref} 

\makeatletter
\AddToHook{cmd/appendices/before}{\def\cref@section@alias{appendix}}
\makeatother

\theoremstyle{plain}
\newtheorem{theorem}{Theorem}[section]
\newtheorem{subtheorem}{Theorem}[theorem]

\newtheorem{proposition}[theorem]{Proposition}
\newtheorem{corollary}[theorem]{Corollary}
\newtheorem{lemma}[theorem]{Lemma}
\newtheorem*{caveat}{Caveat}

\theoremstyle{definition}
\newtheorem{definition}[theorem]{Definition}
\newtheorem{example}[theorem]{Example}
\newtheorem{remark}[theorem]{Remark}
\newtheorem{scholium}[theorem]{Scholium}

\newtheorem{question}[theorem]{Question}

\Crefname{subtheorem}{Theorem}{Theorems}
\crefname{subtheorem}{theorem}{theorems}
\Crefname{assumption}{Assumption}{Assumptions}
\crefname{assumption}{assumption}{assumptions}
\Crefname{construction}{Construction}{Constructions}
\crefname{construction}{construction}{constructions}
\Crefname{fact}{Fact}{Facts}
\crefname{fact}{fact}{facts}

\usepackage{wrapfig}
 \title{
    Profunctorial algebras
}

\thanks{
    The authors express their gratitude to Nathanael Arkor, Ivan Di Liberti, Sam van Gool, Jérémie Marquès, and Daniela Petrişan for their advice during the preparation of this work.  The second-named author acknowledges financial support from the Agence Nationale de la Recherche (ANR), project ANR-23-CE48-0012-01.
}

\author{Quentin Aristote}
\author{Umberto Tarantino}
\date{}

\address[Quentin Aristote]{Université Paris-Cité, CNRS, Inria, IRIF, F-75013, Paris, France.}
\email{aristote@irif.fr}
\urladdr{https://quentin.aristote.fr/}

\address[Umberto Tarantino]{Université Paris Cité, CNRS, IRIF, F-75013, Paris, France.}
\email{tarantino@irif.fr}
\urladdr{https://utarantino.github.io/}


\newcommand{\bb}[1]{\mathbb{#1}}

\renewcommand{\u}[1]{\underline{#1}}
\renewcommand{\frak}[1]{\mathfrak{#1}}

\DeclareFontFamily{U}{mathx}{}
\DeclareFontShape{U}{mathx}{m}{n}{<-> mathx10}{}
\DeclareSymbolFont{mathx}{U}{mathx}{m}{n}
\DeclareMathAccent{\widehat}{0}{mathx}{"70}
\DeclareMathAccent{\widecheck}{0}{mathx}{"71}

\newcommand{\tah}[1]{\widecheck{#1}} \renewcommand{\hat}[1]{\widehat{#1}}

\DeclareMathSymbol{:}{\mathpunct}{operators}{"3A}

\newcommand{\id}{\mathrm{id}}
\newcommand{\op}[1]{{#1}^{\operatorname{op}}}
\newcommand{\co}[1]{{#1}^{\operatorname{co}}}
\newcommand{\coop}[1]{{#1}^{\operatorname{coop}}}

\DeclareMathOperator{\Mod}{Mod}

\def\slashedrightarrow{\relbar\mathrel{\mkern-4mu}\joinrel\mapstochar\mathrel{\mkern-4mu}\joinrel\rightarrow}
\newcommand{\pro}{\slashedrightarrow}

\tikzset{kleisli/.style={ postaction={decorate, decoration={markings, mark= at position 0.5 with {
        \draw circle[radius=1.2pt] ; }}
}}}
\newcommand{\klarrow}{\mathrel{\tikz [line width=.11ex, double distance=.33ex]
    \draw[->, kleisli]
    (0,0) -- (0.4,0);}
}

\newcommand{\ult}{\mathrel{\tikz [line width=.11ex, double distance=.33ex]
    \draw[>-]
    (0.3,0) -- (0,0);}
}

\newcommand{\central}[1]{{#1}^\dagger}

\newcommand{\namedCat}[1]{\mathbf{#1}}
\newcommand{\bicat}[1]{\mathsf{#1}}

\newcommand{\K}{\bicat K} \newcommand{\ps}[1]{\bicat{#1}} 

\let\Set\relax \newcommand{\Set}{\namedCat{Set}}
\newcommand{\Pos}{\namedCat{Pos}}
\newcommand{\Rel}{\namedCat{Rel}}
\newcommand{\Cat}{\namedCat{Cat}}
\newcommand{\CAT}{\namedCat{CAT}}
\newcommand{\Prof}{\namedCat{Prof}}
\newcommand{\PROF}{\namedCat{PROF}}

\newcommand{\Span}[1]{\namedCat{Span}(#1)}

\newcommand{\DFib}[1]{\namedCat{TSDF}(#1)}
\newcommand{\Adj}[1]{\namedCat{Adj}(#1)}

\newcommand{\UltCat}{\namedCat{UltCat}}
\newcommand{\UltSp}{\namedCat{UltSp}}

\newcommand{\ConvSp}{\namedCat{ConvSp}}

\newcommand{\Top}{\namedCat{Top}}
\newcommand{\CompHaus}{\namedCat{KHaus}}
\newcommand{\Clos}{\namedCat{Clos}}

\newcommand{\two}{\namedCat{2}}

\newcommand{\Alg}[1]{\namedCat{Alg}\mathopen{}\left(#1\right)\mathclose{}}

\newcommand{\nLaxAlgrco}[1]{\namedCat{LaxAlg}_{\operatorname{r,\, co}}^{\operatorname{norm}}\mathopen{}\left(#1\right)\mathclose{}}

\newcommand{\PsAlg}[1]{\namedCat{PsAlg}_{\operatorname{ps}}\mathopen{} (#1 )\mathclose{}}

\newcommand{\PsAlgco}[1]{\namedCat{PsAlg}_{\operatorname{co}}\mathopen{} (#1 )\mathclose{}}

\newcommand{\Kl}[1]{\namedCat{Kl}\mathopen{}\left(#1\right)\mathclose{}}

\newcommand{\namedFunc}[1]{\namedCat{#1}}

\newcommand{\unit}[1][]{\eta^{#1}}
\newcommand{\mult}[1][]{\mu^{#1}}
\newcommand{\counit}[1][]{\varepsilon^{#1}}
\newcommand{\monad}[1]{\langle \kern0.05em {#1},\unit[#1],\mult[#1]\rangle}

\newcommand{\bicomma}[2]{{{#2} \triangleleft {#1}}}

\newcommand{\dom}[1][]{d_0}
\newcommand{\cod}[1][]{d_1}

\newcommand{\E}{\mathcal{E}}
\newcommand{\M}{\mathcal{M}}

\newcommand{\supp}{\operatorname{supp}}
\newcommand{\pt}{\operatorname{pt}}
\newcommand{\psh}{\namedFunc{psh}\kern0.1em}

\definecolor{darkgreen}{RGB}{0,102,0}
\definecolor{lightblue}{RGB}{111, 151, 242}
\definecolor{darkred}{RGB}{178,0,0}
\definecolor{darkblue}{HTML}{0000B2}
\definecolor{magenta}{HTML}{B20059}
\definecolor{brightyellow}{RGB}{237, 217, 83}
\definecolor{orange}{RGB}{255,123,0}
\definecolor{lightpink}{RGB}{255, 103, 129}
\definecolor{brightpink}{RGB}{232, 88, 232}
\definecolor{grassgreen}{RGB}{0,154,23}
\definecolor{seablue}{RGB}{52,111,111}
\definecolor{skyblue}{RGB}{135, 206, 235}

\newcommand{\stringdiagram}[2][]{\raisebox{-.5\height}{\includegraphics[scale=1.2, #1]{#2.pdf}}}
\newcommand{\centerstringdiagram}[2][]{\begin{center} \stringdiagram[#1]{#2} \end{center}}
 
\begin{document}

\maketitle

\begin{abstract}
  We provide a bicategorical generalization of Barr's landmark 1970 paper, in
  which he describes how to extend $\Set$-monads to relations and uses this to
  characterize topological spaces as the \emph{relational algebras} of the
  ultrafilter monad. With two-sided discrete fibrations playing the role of
  relations in a bicategory, we first describe how to extend pseudomonads on a bicategory to skew monads on its bicategory of two-sided discrete fibrations, and we characterize in terms of exact squares when these extensions are themselves pseudomonads. As a wide class of examples, we show that every
  $\Set$-monad induces a pseudomonad on the 2-category of categories admitting a skew extension to profunctors, and in a few relevant cases we introduce suitable quotients also extending to profunctors. Among
  the latter, we then focus on the ultracompletion pseudomonad, whose pseudoalgebras are ultracategories: we characterize the normalized lax algebras of its profunctorial extension as ultraconvergence spaces, a recently-introduced categorification of topological spaces.
\end{abstract}

\tableofcontents
\section{Introduction}

\emph{Bicategories}, as $2$-dimensional categorical structures, have recently
found renewed interest in computer science and logic by providing
$\lambda$-calculus and type theory with semantics that are proof-relevant and
concurrency-aware~\cite{savilleCartesianClosedBicategories2019,ahrensSemanticsTwodimensionalType2022,ahrensBicategoricalTypeTheory2023}.
As recognized already since \cite{seelyModellingComputations2categorical1987},
bicategorical models are indeed flexible enough to give a semantics not only to
types and terms, but also to reduction steps between terms. Such models
naturally arise, for instance, in game-based
semantics~\cite{melliesAsynchronousTemplateGames2021,fioreCartesianClosedBicategory2008,clairambaultCartesianClosedBicategory2023},
or typically where composition of morphisms is defined by means of universal
properties.

This has prompted several tools and techniques from ordinary category theory to
be adapted to the bicategorical setting,
e.g.~\cite{fioreCoherenceNormalisationbyevaluationBicategorical2020,paquetEffectfulSemanticsBicategories2024,azevedodeamorimLogicalRelationsCallbypushvalue2025}.
In the same spirit, the purpose of this work is to adapt a result by
Barr~\cite{barrRelationalAlgebras1970} which is fundamental in computer science
and logic. Barr's contribution is twofold: he describes how to extend monads to
act on relations, and uses this to give an algebraic characterization of
topological spaces. Analogously, we characterize how to extend $2$-dimensional
monads (\emph{pseudomonads}) to $2$-dimensional relations (\emph{two-sided
  discrete fibrations}), and use this to recover
$2$-dimensional topological spaces (\emph{ultraconvergence spaces}) algebraically.

\subsection*{Monads and relations}

Since the work of Moggi
\cite{moggiComputationalLambdacalculusMonads1989,moggiNotionsComputationMonads1991},
monads have been extensively used in theoretical computer science to model
computational effects,
e.g.~\cite{wadlerEssenceFunctionalProgramming1992,bentonMonadsEffects2002,
  plotkinNotionsComputationDetermine2002,plotkinAlgebraicOperationsGeneric2003,kammarAlgebraicFoundationsEffectdependent2012,dagninoMonadicTypeAndEffectSoundness2025}:
given a type constructor $T$ representing a certain kind of effect, a
$T$-effectful function $X \klarrow Y$ is a function of type $X \to TY$, and the
monad structure on $T$ provides a well-behaved theory of composition for
$T$-effectful functions. When interpreting types as sets and terms as functions,
non-determinism is, for instance, modeled by having $T$ be the powerset
construction $\mathcal P$; non-deterministic functions $X \to \mathcal P Y$ can
then be identified with relations, i.e., subsets $R \subseteq X \times Y$.

\looseness=-1 Barr first noticed that, viewing a relation as the span of its
projections $X \leftarrow R \rightarrow Y$, there is a natural way to define an
action of a $\Set$-functor $F$ on relations --- namely, by having $F$ act on the
two projections. Such an extension preserves the natural order of
relations and also, under appropriate conditions, their compositions. For
instance, the extension of the powerset construction transforms a relation
between two sets $A$ and $B$ into the one relating subsets of $A$ and $B$ via
the Egli-Milner formula\footnote{That is, where $X \subseteq A$ is related to $Y \subseteq B$ if and only if $\forall x \in X \ \exists
  y \in Y \, :\,  xRy$ and $\forall y \in Y \ \exists x \in X \, : \, xRy$.}. Similarly, a monad structure on $F$ extends to a \emph{skew monad} on relations, which is a genuine monad under appropriate conditions.

Already, this first part of Barr's result has far-reaching implications in
computer science. By providing a way to extend $\Set$-functors to relations, it
allows coalgebraic methods to speak of (bi)simulations and modal logic,
e.g.~\cite{mossCoalgebraicLogic1999,venemaAutomataFixedPoint2006,
  kupkeCompletenessCoalgebraicCover2012,klinCoalgebraicApproachProcess2004}. More recently, Barr's technique has been employed to combine non-determinism with other effects~\cite{garnerVietorisMonadWeak2019,goyPowersetLikeMonadsWeakly2021,aristoteMonotoneWeakDistributive2025} by constructing \emph{weak
  distributive laws}~\cite{bohmWeakTheoryMonads2010,garnerVietorisMonadWeak2019}
over the powerset monad. For instance, the problem of combining probabilistic
choice with non-determinism had gathered extensive interest in the previous literature~\cite{varaccaDistributingProbabilityNondeterminism2006,keimelMixedPowerdomainsProbability2017}, and a simple solution was given
following this approach~\cite{goyCombiningProbabilisticNondeterministic2020,goubault-larrecqWeakDistributiveLaws2024}.

\looseness=-1
As surveyed in~\cite{kurzRelationLiftingSurvey2016}, Barr's result and its applications have already been generalized in several
settings where one can see relations as
spans~\cite{burroniTcategories1971,carboni2categoricalApproachChange1991,bilkovaRelationLiftingsPreorders2011,bilkovaRelationLiftingApplication2013}.
In the ordered setting, the focus is not on mere relations but on the
\emph{strengthening-closed} ones, i.e., relations $R$ between two posets $(A,
{\le_A})$ and $(B,{\le_B})$ such that if $a ' \ge_A a$, $aRb$ and $b \ge_B b'$,
then $a'Rb'$ as well. These relations, widespread in computer science and
logic (see, e.g., \cite{hoareAxiomaticBasisComputer1969,smythStableCompactification1992,jungStablyCompactSpaces2001}),
correspond to monotone functions $(B, {\ge_B}) \times (A,
{\le_A}) \to \two$ where $\two$ is the two-element lattice.
  Categorifying this
picture,
relations between two categories $A$ and $B$ can be identified with
functors $\op B \times A \to \Set$, that is, \emph{profunctors} (see, e.g., \cite{benabouDistributorsWork2000}). Intuitively, such a profunctor $R$ is a \emph{proof-relevant relation} between
$A$ and $B$: while in the posetal case we only know whether $a R b$ or not, elements of $R(b,a)$ are now witnesses to the relation between
$a$ and $b$.

\looseness=-1 Following Barr, we thus construct extensions to the bicategory of (locally small)
categories and (small) profunctors. Our result crucially relies on seeing the latter as
certain spans called \emph{two-sided discrete fibrations}
\cite{streetFibrationsBicategories1980,
  carboniModulatedBicategories1994,loregianCategoricalNotionsFibration2020}, which provides a \emph{fibered} point of view to the above \emph{indexed}
definition of profunctors. As recognized in
\cite{paquetEffectfulSemantics2Dimensional2023,
  paquetEffectfulSemanticsBicategories2024}, effects in a bicategorical setting
can be modeled by
\emph{pseudomonads}~\cite{blackwellTwodimensionalMonadTheory1989}, a
$2$-dimensional generalization of ordinary monads: our first main contribution
is thus to describe a way to extend pseudomonads on a bicategory $\K$ to \emph{skew monads} on its bicategory of two-sided discrete fibrations in $\K$, and characterize when these skew extensions are themselves pseudomonads 
(\Cref{cor:extension-pseudomonads}).

\subsection*{Topological spaces as algebras}

\looseness=-1 In the second half of \cite{barrRelationalAlgebras1970}, Barr
proceeds to use his relational extensions to characterize topological
spaces as the namesake \emph{relational algebras} of the \emph{ultrafilter monad}
$\beta$ on sets --- that is, certain algebras for the skew extension of $\beta$
to relations. To achieve this, he builds on the work of Manes
\cite{manesTripleTheoreticConstruction1969}, where (ordinary) $\beta$-algebras are
characterized as {compact Hausdorff spaces}. Concretely, this means that a
compact Hausdorff topology on a set $X$ can be specified by a function $\beta X
\to X$ assigning, to each ultrafilter on $X$, its \emph{limit} --- one such limit exists because the space is compact, while it is unique because the space is Hausdorff.
Barr similarly recovered arbitrary topologies on $X$, where each ultrafilter may now have multiple limits or none,
in terms of relations between $\beta X$ and $X$ satisfying appropriate axioms. The study of
spaces in terms of such \emph{convergence relations} led to the development
of the field of \emph{monoidal topology}
\cite{clementinoMonoidalTopologyCategorical2014}.

\looseness=-1 Manes and Barr's
characterizations yield the bottom square of~\Cref{fig:cube}. Our second main contribution is to fill up the top square therein:
we apply our extension theorem to construct a skew monad $\u{\bbbeta}$ on profunctors whose algebras recover \emph{ultraconvergence spaces}, a categorification of topological spaces recently introduced in \cite{goolToposesEnoughPoints2025,saadiaExtendingConceptualCompleteness2025}.

\begin{figure}[h]
  \centering
  \[\begin{tikzcd}[column sep = 8pt, row sep = 15pt]
      &[26pt] \UltSp &[-8pt] &[-6pt] {\nLaxAlgrco{\u\bbbeta}} \\
      \UltCat && {\PsAlgco{\bbbeta}} \\
      & \Top && {\mathbf{LaxAlg}\big(\u\beta\big)} \\
      \CompHaus && {\Alg\beta}
      \arrow["{\textbf{\color{darkred}{Thm.~\ref{thm:ucspaces-algebraically-in-full}}}}"{description},
      color=darkred, tail reversed, from=1-2, to=1-4]
      \arrow["{\textbf{\color{darkred}{Cor.~\ref{cor:ucats-as-uspaces}}}}"{description}, hook, from=2-1, to=1-2, color=darkred]
      \arrow[hook, from=3-2, to=1-2]
      \arrow["{\textnormal{\cite{rosoliniUltracompletions2024}}}"{description}, tail
      reversed, crossing over, from=2-1, to=2-3]
      \arrow[hook, from=2-3, to=1-4]
      \arrow[hook, from=3-4, to=1-4]
      \arrow[hook, from=4-1, to=2-1]
      \arrow[hook, from=4-1, to=3-2]
      \arrow["\textnormal{\cite{manesTripleTheoreticConstruction1969}}"{description},
      tail reversed, from=4-1, to=4-3]
      \arrow[hook, from=4-3, to=3-4]
      \arrow["\textnormal{\cite{barrRelationalAlgebras1970}}"{description, pos=0.62},
      tail reversed, from=3-2, to=3-4]
      \arrow[hook, from=4-3, to=2-3, crossing over]
    \end{tikzcd}\]
  \caption{Algebraic presentations of topological spaces}
  \label{fig:cube}
\end{figure}
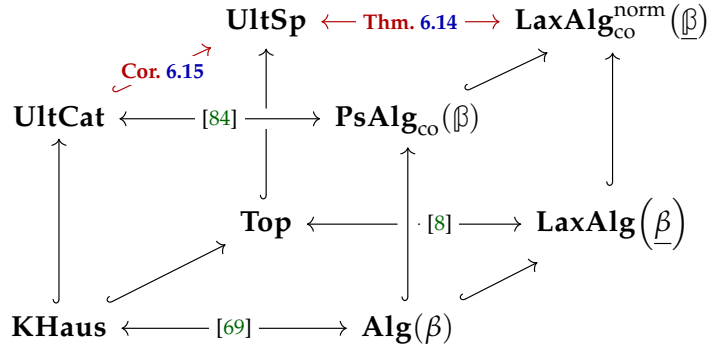

\looseness=-1
The motivation behind ultraconvergence spaces stems from duality theory and categorical logic. In the context of propositional logic,
\emph{Stone duality}
\cite{stoneTheoryRepresentationsBoolean1936,stoneTopologicalRepresentationsDistributive1937}
bridges syntax and semantics, allowing for a topological study of logical
properties which has proved widely influential in computer science
(see, e.g., \cite{gehrkeTopologicalDualityDistributive2024}).
A key feature of Stone duality is Stone's representation theorem for Boolean
algebras, which we can describe in logical terms as follows.
A (classical) propositional theory $\bb T$ can be identified with its
\emph{syntactic algebra} $B_{\bb T}$, i.e.,
the Boolean algebra of propositional formulas
modulo $\bb T$-provable equivalence. Intuitively, two theories with isomorphic syntactic algebras
describe, with possibly different syntax, the same programs and behaviours. Stone's theorem can then be expressed as saying that, endowing the set of
models of a theory $\bb T$ with the \emph{Stone topology}, the
algebra of continuous maps $\Mod(\bb T) \to \two$ is isomorphic to $B_{\bb T}$. In other
words, syntax can be reconstructed from semantics if enough topological
information on the space of models is provided; such reconstructions find
applications, for instance, to domain theory~\cite{abramskyDomainTheory1995} and decidability questions for formal languages~\cite{gehrkeTopologicalApproachRecognition2010}.

\looseness=-1 The quest for an analogue of this result in the setting of
 first-order logic led Makkai to introduce
\emph{ultracategories} \cite{makkaiStoneDualityFirst1987}. Models of a
first-order theory $\bb T$ are arranged in a category, rather than merely a set,
and the possibility of forming \emph{ultraproducts} of models determines what
Makkai calls an \emph{ultrastructure} on the category $\Mod (\bb T)$, playing
the role of the Stone topology in the propositional case. The role of the
syntactic algebra is played, in this context, by the \emph{classifying topos}
$\mathcal E_{\bb T}$ of $\bb T$, a category built out of the syntax of the
theory (see, e.g., \cite[\S D]{johnstoneSketchesElephantTopos2002}). Ultracategories allow for a reconstruction theorem \emph{à la} Stone
for first-order logic: the category of functors $\Mod(\bb T)\to \Set$ (suitably) preserving
ultraproducts is equivalent to $\mathcal E_{\bb T}$. Ultraconvergence spaces \cite{goolToposesEnoughPoints2025}, introduced as \emph{virtual ultracategories} in \cite{saadiaExtendingConceptualCompleteness2025}, are `relational' generalizations of ultracategories extending this theorem to the setting of
\emph{geometric logic}, a fragment of infinitary first-order logic.

\looseness=-1 In this work, we focus on retrieving
{ultraconvergence spaces} as certain algebras for a
skew monad on the bicategory of profunctors. Encoding a categorified Manes' theorem, ultracategories can be defined as algebras for the \emph{ultracompletion} pseudomonad $\bbbeta$ acting on
categories \cite{garnerUltrafiltersFiniteCoproducts2020, rosoliniUltracompletions2024}. This pseudomonad is the one we extend to profunctors, which allows us to characterize ultraconvergence spaces as its \emph{profunctorial algebras} (\Cref{thm:ucspaces-algebraically-in-full}). This way, we provide a categorical justification of ultraconvergence spaces, now framed within a $2$-dimensional
Barr's theorem, and we also allow for
algebraic tools in their study: for instance, we characterize ultracategories as the \emph{representable} ultraconvergence spaces, answering a question raised in \cite{saadiaExtendingConceptualCompleteness2025}.

\subsection*{Content and contributions}

\noindent This paper is organized as follows.
\begin{itemize}
\item In \Cref{sec:02}, we recall the theory of two-sided discrete fibrations in
a bicategory.
\item In \Cref{sec:03}, as our first main contribution, we describe how to extend pseudofunctors, pseudonatural
transformations and modifications on a bicategory to oplax functors, oplax transformations and modifications on its bicategory of two-sided discrete
fibrations, and we characterize when these extensions are actually pseudofunctorial or pseudonatural. In particular, we deduce an extension criterion for pseudomonads (\Cref{cor:extension-pseudomonads}).
\item In \Cref{sec:04}, we focus on the $2$-category $\CAT$ of (locally
small) categories: there, as another contribution, we show that a wide class of pseudomonads, built from
$\Set$-monads, extend to skew monads on
the bicategory $\PROF$ of categories and \emph{small} profunctors
(\Cref{thm:loke-extend}) by means of our criterion. Moreover, we review our characterization through the theory of distributive laws, reobtaining the known result that pseudomonads on $\CAT$ admit at most one extension to pseudomonads on $\PROF$ (\Cref{cor:extension-on-cat}). \item \looseness=-1 In \Cref{sec:05}, we apply our extension theorem to the pseudomonad $\bbbeta$: as our second main contribution, we characterize ultraconvergence spaces as the \emph{normalized lax algebras} of its profunctorial (skew) extension $\u\bbbeta$ (\Cref{thm:ucspaces-algebraically-in-full}). Towards this goal, we discuss the problem of defining lax algebras and their colax morphisms for skew monads, and we introduce sufficient conditions allowing to do so. 
\end{itemize}

\subsection*{Future work}
\looseness=-1 A natural direction for further work is to adapt
the ramifications of Barr's $1$-dimensional result
(coalgebraic techniques, distributive laws, monoidal topology) to the
$2$-dimensional setting.
We also believe the discussions of \Cref{sec:04,sec:05} open two lines of work on the semantics of $\lambda$-calculus.

\looseness=-1 First, $\CAT$ and $\PROF$ are well-known to respectively be
cartesian closed and compact closed bicategories, and thus $2$-dimensional
models of $\lambda$-calculus and of linear logic. The pseudomonads we introduce
in \Cref{sec:04} and their skew extensions are a natural choice to model effects in $\CAT$ and $\PROF$, as
they naturally extend the monads on $\Set$. One direction for future work is
therefore to study these monads and their properties: for instance,
whether they are strong, concurrent or commutative
\cite{paquetEffectfulSemanticsBicategories2024}, and whether they combine with
each other via pseudodistributive laws
\cite{marmolejoDistributiveLawsPseudomonads1999,chengPseudodistributiveLaws2003}. To this latter aim, a theory of distributive extensions between skew monads could be worth developing.

The machinery we present, motivated by the study of ultraconvergence
spaces, leaves us with a modular way to produce natural categorification of
the algebras of many $\Set$-monads. Natural examples are categories of domains,
for instance that of continuous lattices, which can be altered to give
domain-theoretic models of the $\lambda$-calculus (cf.\ \Cref{rem:algebras-filter}), such as Scott domains. We
conjecture that some analogous variations can be performed on their $2$-dimensional
versions, thus providing $2$-dimensional domain-theoretic models of the
$\lambda$-calculus.

\begin{caveat}
  In a first version of this draft, we claimed that the pseudomonads considered in \Cref{sec:04} extend to \emph{pseudomonads} on $\PROF$, but the proof of this result contained a mistake; we are indebted to Jérémie Marquès for pointing it out to us. As a consequence, we do not claim anymore that ultracategories are closed under free small-cocompletion, question that we now leave open. The new version of \Cref{sec:05} works around this issue by showing that the skew extensions of these pseudomonads still admit a well-behaved theory of algebras encompassing ultraconvergence spaces.
\end{caveat} \section{Relations in a bicategory}\label{sec:02}

\looseness=-1
In this section we recall basic definitions for constructing the bicategory of
discrete fibrations in a bicategory. This generalizes the construction of
relations in a regular category --- e.g., in $\Set$ --- as jointly monic
spans~\cite[Thm.\ 2.8.4]{borceuxRegularCategories1994}, the construction of
monotone relations in $\Pos$ as strengthening-closed jointly-monic
spans~\cite[Prop.\ 2.13]{bilkovaRelationLiftingsPreorders2011}, and the
construction of profunctors in categories of enriched categories as
collages~\cite[Prop.\ 6.21]{streetFibrationsBicategories1980}. Although
discrete fibrations in a bicategory go back to
Street~\cite{streetFibrationsBicategories1980}, we adopt the modern approach
accounted for in~\cite{loregianCategoricalNotionsFibration2020} and generalizing
the one of~\cite{carboniModulatedBicategories1994}.

We assume the reader to be familiar with ordinary (i.e.\ $1$-dimensional)
category theory~\cite{borceuxHandbookCategoricalAlgebra1994a}. Bicategories, introduced in~\cite{benabouIntroductionBicategories1967}, are a
$2$-dimensional generalization of categories where objects have
\emph{categories} (in this paper, always assumed to be \emph{locally small}) of morphisms between them, instead of mere sets thereof.

\begin{definition}[{\cite[Def. 1.1]{benabouIntroductionBicategories1967}}] \label{def:bicat}
  A \emph{bicategory} $\K$ consists of:
  \begin{itemize}
  \item a class of objects;
  \item for any two objects $A, B$, a \emph{hom}-category
    $\K(A,B)$ whose objects are called \emph{arrows} and denoted with
    ${\to}$, and whose morphisms are called \emph{$2$-cells} and denoted with
    ${\Rightarrow}$;
  \item for any object $A$, an \emph{identity} arrow $1_A\colon A \to A$;
  \item for any three objects $A,B,C$, a \emph{composition} functor
  $\circ \colon \K(B,C) \times \K(A,B) \to \K(A,C)$;

  \item \looseness=-1 for any arrow $f\colon A \to B$, invertible $2$-cells $\rho_f\colon f \circ 1_A \cong f$ and
    $\lambda_f\colon 1_A \circ f \cong f$
    witnessing unitality of identity arrows;
  \item for any three arrows $f\colon A \to B$, $g\colon B
    \to C$ and $h\colon C \to D$,
   an invertible $2$-cell $\alpha_{h,g,f}\colon h \circ (g \circ f) \cong
    (h \circ g) \circ f$ witnessing associativity of composition.
  \end{itemize}
  This data is required to satisfy some \emph{coherence axioms}, spelled out in Appendix
  \ref{app:A2}. In particular, $\K$ is a \emph{$2$-category} if each structural
  2-cell is the identity.
\end{definition}

The coherence conditions make it so that any two
diagrams between the same source and target arrows, involving only structural $2$-cells, coincide
\cite{maclaneCoherenceBicategoriesIndexed1985}. This allows us to draw diagrams
in a bicategory leaving the structural 2-cells implicit.

\looseness=-1
When not ambiguous, we omit
subscripts. We also generally omit the composition functor entirely,
writing $g \circ f$ as $gf$. Moreover, we write $\alpha *f$ or simply $\alpha f$ for the \emph{whiskering} $\alpha \circ \id_f$ of a 2-cell $\alpha$ with an arrow $f$, and similarly $f*\alpha$ or $f \alpha$ in place of $\id_f \circ \alpha$. Finally, we write $\cdot$ for the composition
operation inside the hom-categories.

\looseness=-1
For a bicategory $\K$, we denote by $\op\K$ and $\co\K$ the bicategories obtained by reversing arrows and $2$-cells respectively.

\looseness=-1 Many concepts of ordinary category theory can be internalized in a
bicategory $\K$: we mention here two such concepts that will appear later. An
arrow $f \colon A \to B$ is an \emph{equivalence} if there exists another arrow
$g \colon B \to A$ together with two invertible 2-cells $gf \cong 1_A$ and $fg
\cong 1_B$. An arrow $f \colon A \to B$ is \emph{left adjoint} to $g\colon B \to
A$ if there are \emph{unit} and \emph{counit} 2-cells $1_A \Rightarrow g f$ and
$fg \Rightarrow 1_B$ satisfying the usual triangle identities.

\begin{example} \label{ex:bicats} Throughout this section we will build
  intuition with the following running examples:
  \begin{itemize}
  \item the category $\Set$ of sets and functions, seen as a bicategory with no
    $2$-cells other than the
    identities; \item the locally\footnotemark{} posetal $2$-category $\Pos$ with objects
    partially ordered sets and hom-category $\Pos(A,B)$ the poset of monotone
    functions $A \to
    B$; \footnotetext{\looseness=-1 A property holds \emph{locally} if it holds for
      all hom-categories.}
  \item \looseness=-1 the $2$-category $\Cat$ whose objects are small
    categories, arrows are functors, and $2$-cells are natural transformations.
  \end{itemize}
  \looseness=-1 Examples of bicategories which are not $2$-categories will be given shortly by
  the bicategories of two-sided discrete fibrations.
\end{example}

\looseness=-1 As bicategories are categories `up to isomorphism', one can
similarly generalize functors and natural transformations: however, the
flexibility provided by the additional dimension also allows us to weaken the
structural isomorphisms to $2$-cells in a fixed direction. We sketch here the
relevant definitions, deferring coherence axioms to Appendix~\ref{app:A2}.

A \emph{pseudofunctor} $\ps F\colon \K \to \K'$ between two bicategories is an
assignment on objects $A \mapsto \ps F A$ together with a family of functors
$\K(A,B) \to \K'(\ps F A,\ps F B)$ and (coherent) invertible $2$-cells $\ps F
1_A \Rightarrow \ps 1_{\ps F A}$ and $\ps F(gf) \Rightarrow (\ps F g)(\ps F f)$
witnessing preservation of identities and composition. More generally, relaxing
the hypotheses of invertibility defines \emph{oplax functors}, while taking
$2$-cells in the other direction defines \emph{lax functors}.

A \emph{pseudonatural transformation} $\gamma\colon \ps F \Rightarrow \ps F'$
between two pseudofunctors $\K \to \K'$ is a family of arrows $\gamma_A\colon
\ps F A \to \ps F' A$ along with (coherent) invertible $2$-cells $\gamma_f\colon
\gamma_B (\ps F f)\Rightarrow (\ps F' f) \gamma_A$ for $f \colon A \to B$
witnessing commutativity. More generally, relaxing the hypotheses of
invertibility defines \emph{oplax natural transformations}, while taking
$2$-cells in the other direction defines \emph{lax natural transformations}. All
three notions also make sense between lax and oplax functors.

Finally, a further notion of higher-order morphism arises compared to the
$1$-dimensional setting: a \emph{modification} $\frak m\colon \gamma
\Rrightarrow \gamma'$ between transformations $\ps F \Rightarrow \ps F'$ is a
(coherent) family of $2$-cells $\frak m_A\colon \gamma_A \Rightarrow \gamma_A'$.

\subsection{Two-sided discrete fibrations}

The entities that will play the role of relations in the $2$-dimensional setting
are a certain kind of spans called \emph{two-sided discrete fibrations}. To
introduce them, we first recall the definition of \emph{spans}. Let $\K$ be a
bicategory.

\begin{definition}
  A \emph{span} from $A$ to $B$ in $\K$ is a pair of arrows $A \xleftarrow{q} E
\xrightarrow{p} B$. A \emph{morphism of spans}
consists of an arrow and two invertible $2$-cells as in the diagram:
\[\begin{tikzcd}[column sep=huge,row sep=.1em]
	& E \\
	A && B \\
	& {E'}
	\arrow["q"', curve={height=5pt}, from=1-2, to=2-1]
	\arrow["p", curve={height=-5pt}, from=1-2, to=2-3]
	\arrow[""{name=0, anchor=center, inner sep=0}, from=1-2, to=3-2]
	\arrow["{q'}", curve={height=-5pt}, from=3-2, to=2-1]
	\arrow["{p'}"', curve={height=5pt}, from=3-2, to=2-3]
	\arrow["\cong"{description}, draw=none, from=0, to=2-1]
	\arrow["\cong"{description}, draw=none, from=0, to=2-3]
\end{tikzcd}\] Two such morphisms $(f\colon E \to E'$, $\nu_0\colon q' f \cong
q$, $\nu_1\colon p'f \cong p)$ and $(g\colon E \to E'$, $\gamma_0\colon q'g
\cong q$, $\gamma_1\colon p'g \cong p)$ are \emph{isomorphic} if there is an
invertible $2$-cell $\delta\colon f \cong g$ such that $\gamma_0 \cdot q'\delta
= \nu_0$ and $\gamma_1 \cdot p' \delta = \nu_1$. We denote by $\Span{\K}(A,B)$
the (ordinary) category whose objects are spans $A \to B$ and whose
arrows are isomorphism classes of morphisms of spans.
\end{definition}

A span of functions $A \xleftarrow{q} E \xrightarrow{p} B$ in $\Set$ naturally
yields the relation $\{ (q(e),p(e)) \in A \times B \mid e \in E \}$, and
similarly in $\Pos$. However, in $\Pos$, we want to capture not just mere
relations, but the \emph{strengthening-closed} ones. For a span $A \to B$ in
$\Pos$ to give rise to a strengthening-closed relation, it needs to be a
so-called \emph{two-sided fibration}. A second obstacle to modeling relations as
spans is that, while every span gives rise to a relation, a given relation may
be determined by different spans. In $\Set$, for instance, spans $A \to B$
correspond precisely to functions $R \to A \times B$, whereas relations
correspond to \emph{injective} functions $R \hookrightarrow A \times B$. In a
bicategory, the analogue of this restriction is to ask for the two-sided
fibration to be \emph{discrete}.

The exact definition of two-sided fibrations (TSFs) is quite technical, and we
thus defer it to \Cref{app:A2}. We instead directly give an alternative
characterization of two-sided discrete fibrations.
To understand the rest of this paper, it is instead enough to have in mind:
\begin{enumerate}
\item the characterizations given in \Cref{ex:two-sided-discrete-fibrations},
  of TS(D)Fs in our running example;
\item the fact that TSDFs are spans, which the constructions of \Cref{sec:03}
  crucially rely on;
\item a generic way of constructing TSDFs that we describe in the rest of this
  section.
\end{enumerate}

\begin{definition}[{\cite[Prop.\ 4.7]{carboniModulatedBicategories1994}}]
  \label{def:two-sided-discrete-fibration}
  A span $A \xleftarrow{q} E \xrightarrow{p} B$ in $\K$ is a
  \emph{two-sided discrete fibration} if it enjoys the following properties.
  \begin{enumerate}
  \item\emph{unique left path lifting}: for any $e\colon D \to E$ and
    $\gamma\colon b \Rightarrow pe$, the category, whose objects are pairs
    $(\zeta\colon e' \Rightarrow e, \nu\colon b \cong pe')$ with $\gamma=
    p\zeta\cdot \nu$ and $q\zeta$ invertible, is essentially discrete and
    non-empty;
  \item \emph{unique right path lifting}: for any $e\colon D \to E$ and
    $\gamma\colon qe \Rightarrow a$, the category, whose objects are pairs
    $(\chi\colon e \Rightarrow e', \nu\colon qe' \cong a)$ with $\gamma = \nu
    \cdot q\chi$ and $p\chi$ invertible, is essentially discrete and non-empty;
  \item \emph{factorization}: any $\gamma\colon e' \Rightarrow e''\colon D
    \to E$ is a composite $\gamma = \zeta \cdot \chi$ with $p\chi$ and $q\zeta$
    invertible.
  \end{enumerate}
  We denote by $\DFib{\K}(A,B)$ the full subcategory of $\Span{\K}(A,B)$ spanned
  by the discrete fibrations.
\end{definition}

\begin{example}
  \label{ex:two-sided-discrete-fibrations}
  We describe TS(D)Fs in our running examples. Note
  that the discreteness condition is not necessarily appropriate when working in
  locally discrete bicategories; in $\Set$ and $\Pos$, a more restricted class of TSFs
  is needed to capture the correct notion of relations. We still mention
  these examples because they provide a good first intuition, and because the
  rest of the theory works the same.
  \begin{itemize}
  \item\looseness=-1 In $\Set$, every span $A \xleftarrow{q} R \xrightarrow{p}
    B$ is a TSF, and even a TSDF. Relations $R \subseteq A \times B$ are in
    one-to-one correspondence with \emph{jointly monic} TSFs: those $A
    \leftarrow R \rightarrow B$ whose corresponding function $R \to A \times B$
    is injective.

  \item In $\Pos$, a span $(A,{\leq_A}) \xleftarrow{q} (R,{\leq_R})
    \xrightarrow{p} (B,{\leq_R})$ being a TSF means that the relation it induces
    is strengthening-closed. Restricting ourselves to those TSFs $(A, {\leq_A})
    \leftarrow (R,{\leq_R}) \rightarrow (B,{\leq_B})$ such that the monotone
    function $(R, {\leq_R}) \to (A, {\leq_A}) \times (B, {\leq_B})$ is injective
    and order-reflecting, we get a poset which is isomorphic to that of
    strengthening-closed relations between $(A, {\le_A})$ and $(B,
    {\le_B})$~\cite[Prop. 2.13]{bilkovaRelationLiftingsPreorders2011}, which
    themselves correspond to monotone functions $(B,{\ge_B}) \times (A, {\le_A})
    \to \two$.

  \item Via the Grothendieck construction, $\DFib{\Cat}(A,B)$ is equivalent to
    the category of \emph{profunctors} from $A$ to $B$, i.e., functors $\op B
    \times A \to \Set$~\cite[Thm.
    2.3.2]{loregianCategoricalNotionsFibration2020}, which we denote as $A \pro
    B$. Note how these categorify the previous example, where $\Set$ now plays
    the role of $\two$.
  \end{itemize}
\end{example}

A generic way to get TSDFs is via \emph{bicomma squares}\footnotemark.

\footnotetext{By \emph{square}, in a bicategory, we always mean a $2$-cell of type $xu \Rightarrow
  yv$.}

\hspace{-1.28em}
\begin{minipage}{.78\linewidth}
  \begin{definition}\label{def:bicomma-square}
    A \emph{bicomma object} of a cospan $A \xrightarrow{g} C \xleftarrow{f} B$
    in $\K$ is an object $\bicomma{f}{g}$ in $\K$ equipped with a span $A
    \xleftarrow{\cod} \bicomma{f}{g} \xrightarrow{\dom} B$ and a $2$-cell $f
    \dom \Rightarrow g \cod$ which is universal among all such squares, in a
    sense we spell out in Appendix \ref{app:A2}. We refer to $A
    \xleftarrow{\cod} \bicomma{f}{g} \xrightarrow{\dom} B$ as a \emph{bicomma
      span}.
  \end{definition}
\end{minipage}\begin{minipage}{.22\linewidth}
  \[ \begin{tikzcd}
      {\bicomma{f}{g}} & { B} \\
      A & C \arrow["\dom", from=1-1, to=1-2]\arrow["\cod"', from=1-1, to=2-1]\arrow["f", from=1-2, to=2-2]\arrow[between={0.3}{0.7}, Rightarrow, from=1-2, to=2-1]\arrow["g"', from=2-1, to=2-2]\end{tikzcd}\]\end{minipage}
\vspace{-4pt}

The universal property guarantees that, if it exists, a bicomma object is unique
up to an \emph{essentially\footnotemark{} unique} equivalence.

\footnotetext{\emph{Essentially} means ``up to an invertible $2$-cell''.}

\begin{example}
  We compute the bicomma object $\bicomma f g$ of a cospan $A \xrightarrow{g} C
  \xleftarrow{f} B$ in the bicategories of \Cref{ex:bicats}.
  \begin{itemize}
  \item In $\Set$, the set $\bicomma f g$ is the pullback $\set{ (a,b) \in A
      \times B | f(b) = g(a) }$ of $f$ and $g$.
  \item In $\Pos$, the poset $\bicomma f g$ is the set $\set{ (a,b) \in A \times
      B | f(b) \le_C g(a) }$ equipped with the pointwise order. The canonical
    2-cell witnesses that $f \dom \le g \cod$ pointwise.
  \item In $\Cat$, the category $\bicomma f g$ has $C$-arrows of the shape $h
    \colon f(b) \to g(a)$ as objects, and commuting squares $h' \circ f(r) =
    g(s) \circ h$ for $r\colon b \to b'$ in $B$ and $s\colon a \to a'$ in $A$ as
    arrows $h \to h'$. The canonical 2-cell is the natural transformation $f
    \dom \Rightarrow g \cod$ with component at $h\colon f(b) \to g(a)$ given by
    $h$ itself.
  \end{itemize}
\end{example}

In any bicategory, bicomma spans are always TSDFs~\cite[\textsection
3.44]{streetFibrationsBicategories1980}. In $\Cat$, in particular, every TSDF is
isomorphic to a bicomma span
\cite[Prop.~4.10]{carboniModulatedBicategories1994}, and the profunctor
corresponding to the bicomma span of $A \xrightarrow{g} C \xleftarrow{f} B$ is
given by the functor $C(f-_1,g-_2)\colon \op B \times A \to \Set$.

\looseness=-1 Simple examples of bicomma spans are those of the form $A
\leftarrow \bicomma{1_A}{1_A} \rightarrow A$: we write $A^\two$ for
$\bicomma{1_A}{1_A}$. For an arrow $f\colon A \to B$, the universal property of
$B^\two$ applied to the square $1_B f = 1_B f$ yields a morphism of TSDFs
$f^\two\colon A^\two \to B^\two$. We refer to the assignment $(-)^\two$ as
\emph{powering} by $\two$. Other simple examples of bicomma spans are the
following.

\begin{definition}[\protect{\cite[\S
    4.22]{carboniModulatedBicategories1994}}] \label{def:graph} For an arrow
  $f\colon A \to B$ in $\K$, we define its \emph{graph} and its \emph{cograph}
  respectively as the bicomma spans $A \leftarrow \bicomma{1_B}{f} \rightarrow
  B$ and $B \leftarrow \bicomma{f}{1_B} \rightarrow A$. By the universal
  property of bicomma objects, the assignments $f \mapsto \bicomma{1_B}{f}$ and
  $f \mapsto \bicomma{f}{1_B}$ extend to fully-faithful functors:
  \[ (-)_\diamond \colon \K(A,B) \to \DFib{\K}(A,B) \qquad  (-)^\diamond\colon \K(A,B) \to \op{\DFib{\K}(B,A)} \]
Spans in the image of $(-)_\diamond$ are also said to be \emph{representable}.
\end{definition}

\begin{example}
  In $\Set$ and $\Pos$, the relation corresponding to $f_\diamond$ is the usual
  graph of $f$: $\set{ (x,y) | y \leq f(x) }$. In $\Cat$, we can identify
  $f_\diamond$ with the profunctor $A \pro B$ defined by $f_\diamond (b,a) =
  B(b, f(a))$. \end{example}

\subsection{The bicategory of two-sided discrete fibrations}

We now explain how the categories $\DFib{\K}(A,B)$ are the hom-categories of
a bicategory with objects those of $\K$. In other words, we will
explain how to compose TSDFs, generalizing how relations can be composed in the
$1$-dimensional setting. We start by explaining how to compose spans, for which
we need the notion of \emph{bipullback}.

\hspace{-1.28em}
\begin{minipage}{.78\linewidth}
  \begin{definition}
    A \emph{bipullback} of a cospan $A \xrightarrow{g} C
    \xleftarrow{f} B$ in $\K$ is an object $A \times_C B$ equipped with
    projections $A \xleftarrow{\pi_A} A \times_C B \xrightarrow{\pi_B} B$ and an
    invertible $2$-cell $f\pi_B \cong g\pi_A$ which is universal among all such
    squares in the sense spelled out in Appendix \ref{app:A2}.
  \end{definition}
\end{minipage}
\begin{minipage}{.22\linewidth}
  \[\begin{tikzcd}[row sep = 17pt]
      {A \times_C B} & { B} \\
      A & C \arrow["\pi_B", from=1-1, to=1-2]\arrow["\pi_A"', from=1-1, to=2-1]\arrow["f", from=1-2, to=2-2]\arrow["\cong"', between={0.3}{0.7}, Rightarrow, from=1-2, to=2-1]\arrow["g"', from=2-1, to=2-2]\end{tikzcd}\]\end{minipage}
\vspace{-5pt}

As for bicomma objects, bipullbacks are unique up to essentially unique
equivalence if they exist.

\begin{example}
  In $\Set$ and $\Pos$, bipullbacks coincide with pullbacks. In $\Cat$, the
  bipullback of a cospan of functors $A \xrightarrow{g} C \xleftarrow{f} B$ can be identified, up to equivalence, with
  the full subcategory of $\bicomma{f}{g}$ spanned by those $C$-arrows $f(b)
  \to g(a)$ that are invertible.
\end{example}

In a bicategory with bipullbacks, composing spans is easy: the composite of $A
\xleftarrow{u} E \xrightarrow{v} B$ and $B \xleftarrow{x} F \xrightarrow{y} C$
is obtained by forming the bipullback of $E \xrightarrow{v} B \xleftarrow{x} C$
and appending $u$ and $y$.
\[\begin{tikzcd}[row sep=0pt, column sep=large]
	&& {E \times_B F} \\
	& E && F \\
	A && B && B
	\arrow["{{\pi_E}}"', from=1-3, to=2-2]
	\arrow["{{\pi_F}}", from=1-3, to=2-4]
	\arrow["\cong"{description}, draw=none, from=1-3, to=3-3]
	\arrow["u"', from=2-2, to=3-1]
	\arrow["v"', dotted, from=2-2, to=3-3]
	\arrow["x", dotted, from=2-4, to=3-3]
	\arrow["y", from=2-4, to=3-5]
\end{tikzcd}\]
By the universal property of bipullbacks, this construction
extends to a composition functor:
\[ \Span{\K}(B,C) \times \Span{\K}(A,B) \to \Span{\K}(A,C) \]However, the composite of two TSDFs need not be a TSDF. One way to fix this is
to require the existence of an \emph{orthogonal factorization system} on
$\K$. 

\begin{definition}[\protect{\cite{bettiFactorizationsBicategories1999}}]\label{def:orthogonal-factorization-system}
  A \emph{factorization system} on $\K$ is a pair $(\E, \M)$ of
  classes of arrows in $\K$ such that:
  \begin{itemize}
  \item every arrow $f$ admits an \emph{$(\E,\M)$-factorization}, i.e., an
    invertible $2$-cell $f \cong me$ with $m \in \M$ and $e \in \E$;
  \item every $\E$-arrow $e\colon X \to Y$ is \emph{orthogonal} to every
    $\M$-arrow $m\colon Z \to W$ in a suitable sense (deferred to
    \Cref{app:A2}), ensuring uniqueness of the above factorization up to an
    essentially unique equivalence.
  \end{itemize}
\end{definition}

The original approach to compose TSDFs, which does not
include the example of $\Cat$, consists in constructing a specific orthogonal factorization
system~\cite{carboniModulatedBicategories1994}. A more modern approach, reviewed
in~\cite{loregianCategoricalNotionsFibration2020}, asks instead for a
factorization system whose right class $\M$ is \emph{generated by} the TSDFs, in
the sense that
\begin{enumerate}
\item $\E$-arrows are exactly those orthogonal to the arrows $E
  \xrightarrow{\langle q,p \rangle} A \times B$ where $A \xleftarrow{q} E
  \xrightarrow{p} B$ is a TSDF, and
\item $\M$-arrows are those to
which all $\E$-arrows are orthogonal.
\end{enumerate}
In both approaches $\E$ is asked to be stable under bipullbacks and powering
by $\two$, and $\M$ contains all of those $E \to A \times B$ corresponding to
TSDFs. This makes it so that the $\M$-arrow in the $(\E,\M)$-factorization of an
arrow $E \to A \times B$ coming from a TSF itself corresponds to a
TSDF~\cite[Prop.\ 4.18]{carboniModulatedBicategories1994}. Since the
bipullback-composite of two TSFs is again a TSF~\cite[\textsection
4.16]{carboniModulatedBicategories1994}, we can define the composite of two
TSDFs by taking the $\M$-arrow factoring their composite as TSFs, thus yielding
a functor:
\[ \DFib{\K}(B,C) \times \DFib{\K}(A,B) \to \DFib{\K}(A,C) \]The various universal properties involved allow us to construct appropriate
invertible 2-cells realizing associativity and, for each object $A$, unitality
of $A \leftarrow A^\two \rightarrow A$.

\begin{theorem}[{\cite[Thm.
    4.3.4]{loregianCategoricalNotionsFibration2020}}]
  \label{thm:dfibs-bicategory}
  Let $\K$ be a bicategory with finite products, bipullbacks and bicomma
  objects, and suppose given a factorization system $(\E,\M)$ on $\K$ such
  that:
  \begin{itemize}
  \item $\M$ is generated by TSDFs;
  \item $\E$ is stable under bipullbacks and powering by $\two$.
  \end{itemize}
  The above composition defines a bicategory $\DFib{\K}$ with objects those of $\K$ and hom-categories the $\DFib{\K}(A,B)$.
\end{theorem}

We call a bicategory satisfying the conditions of \Cref{thm:dfibs-bicategory}
\emph{regular} (see also~\cite{bourkeTwodimensionalRegularityExactness2014} for
other notions of regularity in the $2$-dimensional setting). In that case, the
functors $(-)_\diamond$ and $(-)^\diamond$ of \Cref{def:graph} determine locally
fully-faithful pseudofunctors
\[ (-)_\diamond \colon \K \to \DFib{\K} \mbox{ and }(-)^\diamond \colon \op\K
  \to \co{\DFib{\K}}, \]acting as the identity on objects~\cite[\textsection 4.22 \&
4.24]{carboniModulatedBicategories1994}. In fact, each of them determines the
other up to pseudonatural equivalence since, for every $f\colon A \to B$,
$f_\diamond$ is left adjoint to $f^\diamond$ in $\DFib{\K}$~\cite[Prop.
4.23]{carboniModulatedBicategories1994}. These properties express the fact that
the pseudofunctor $(-)_\diamond\colon \K \to \DFib{\K}$ is a \emph{proarrow
  equipment} on $\K$~\cite{woodAbstractProArrows1982}.

\begin{example}\label{ex:bicats-of-tsdfs} Let us reconsider our running
  examples.
  \begin{itemize}
  \item Surjections and injections form such a factorization system in $\Set$:
    every function can be factored as a surjection onto the image followed by an
    injection into the codomain. This factorization system can be used to
    compose jointly monic spans, in a way that yields a bicategory
    biequivalent\footnotemark{} to the locally posetal category $\Rel$ of sets
    and relations: composition of relations $R \subseteq A \times B$ and $S
    \subseteq B \times C$ is obtained by first computing the set $R \times_B S =
    \set{ ((a,b),(b,c)) | aRb \land bSc }$ (their composite as spans) and then
    forgetting the $b$'s (taking the $(\E,\M)$-image of $R \times_B S
    \rightarrow A \times C$), yielding the relation $\set{ (a,c) | \exists b :
      aRb \land bRc } \subseteq A \times C$.

    \footnotetext{A \emph{biequivalence} $\K \simeq \K'$ is a pair of pseudofunctors
      $F \colon \K \to \K'$ and $G \colon \K' \to \K$ with pseudonatural
      transformations $FG \Rightarrow \id_{\K'}$ and $\id_{\K} \Rightarrow GF$
      whose components are themselves equivalences.
}

  \item Similarly, surjections and injective order-reflections arrange jointly
    order-embeddings spans in $\Pos$ into a bicategory biequivalent to the
    locally posetal category of posets and strengthening-closed relations, with
    the same formula for compositions.

  \item\looseness=-1 The categories $\DFib{\Cat}(A,B)$ can be arranged into a
    bicategory biequivalent to $\Prof$, the bicategory of profunctors~\cite[Thm.
    7.23]{arkorBicategoriesAlgebrasRelative2025}, where the composite of $F
    \colon \op B \times A \to \Set$ and $G \colon \op C \times B \to \Set$ is
    given on $(c,a)$ in $\op C \times A$ by the \emph{coend}:
    \[ \int^{b \in B} G(c, b)\times F(b,a)\]Concretely, this set is the quotient of the disjoint union $\sum_{b \in B}
    G(c,b)\times F (b, a)$ by the equivalence relation determined by identifying
    $(b, \phi,\psi) \sim (b', \phi',\psi')$ if there exists a (finite) zig-zag
    of $B$-arrows between $b$ and $b'$ whose $G(c, -)$-image maps $\phi$ to
    $\phi'$ and whose $F(-, a)$-image maps $\psi'$ to $\psi$~\cite[Prop.
    7.8.2]{borceuxBicategoriesDistributors1994}. However, we do not know of an
    explicit description of this composition in terms of a factorization system.
    In the following, we will thus work with $\coop\Cat$ instead: the
    factorization system on $\Cat$ consisting of essentially surjective and
    fully-faithful functors makes $\coop\Cat$ into a regular bicategory, such
    that $\DFib{\coop \Cat}$ is biequivalent to
    $\coop\Prof$~\cite[Cor.~4.3.6]{loregianCategoricalNotionsFibration2020}.
    Under this correspondence, $(-)_\diamond: \coop\Cat \to \DFib{\coop\Cat}
    \cong \coop\Prof$ is the usual equipment $\Cat \to \Prof$. In general, TSDFs
    in $\op \K$ correspond to cospans in $\K$ called \emph{two-sided codiscrete
    cofibrations}.
  \end{itemize}
\end{example}
 \section{Extending to two-sided discrete fibrations}\label{sec:03}

\looseness=-1
Throughout, let $\K$ be a regular bicategory with factorization system $(\mathcal E,\mathcal M)$. In this section we present our first main
contribution: we describe how to \emph{extend} --- in a sense we will make precise as we state the result --- pseudofunctors, pseudonatural
transformations, and modifications on $\K$ to {oplax functors}, oplax natural transformations, and modifications on the
bicategory $\DFib{\K}$. In particular, in the first two cases, we characterize when these extensions are themselves pseudofunctorial and pseudonatural in terms of \emph{exact squares}. Following the insight of
\cite{paquetEffectfulSemanticsBicategories2024} that computational effects should be
represented $2$-dimensionally by \emph{pseudomonads}, we then combine
these characterizations into an extension criterion for the latter. Proofs omitted from this section can be found in Appendix \ref{app:A3}.

\subsection{Exact squares}\label{ssec:exact-squares}

Our characterizations will be given in terms of \emph{exact squares} in $\K$,
\begin{wrapfigure}{r}{.15\linewidth}

\begin{tikzcd}[cramped]
	A & B \\
	C & D \arrow["u", from=1-1, to=1-2]\arrow["v"', from=1-1, to=2-1]\arrow["\gamma"', between={0.3}{0.7}, Rightarrow, from=1-2, to=2-1]\arrow["x", from=1-2, to=2-2]\arrow["y"', from=2-1, to=2-2]\end{tikzcd}\end{wrapfigure}
originally introduced by Guitart in~\cite{guitartRelationsCarresExacts1980}. For
any square $\gamma$ as on the right, there is a canonical $2$-cell in
$\DFib{\K}$, obtained by pre- and post-composing $\gamma_\diamond$ with the unit
$\eta_x$ of $x_\diamond \dashv x^\diamond$ and the counit $\varepsilon_v$ of
$v_\diamond \dashv v^\diamond$. Omitting the structural $2$-cells witnessing
associativity of compositions, we can write it as:
\[\begin{tikzcd}[column sep = 32pt]
	{u_\diamond v^\diamond } & {x^\diamond x_\diamond   u_\diamond v^\diamond } &
  {x^\diamond y_\diamond v_\diamond v^\diamond } & {x^\diamond y_\diamond} \arrow["{\eta_x u_\diamond v^\diamond }", Rightarrow, from=1-1, to=1-2]\arrow["{x^\diamond\gamma_\diamond v^\diamond}", Rightarrow, from=1-2, to=1-3]\arrow["{x^\diamond y_\diamond\varepsilon_v}", Rightarrow, from=1-3, to=1-4]\end{tikzcd}\]

\begin{definition}
  A square $\gamma\colon xu \Rightarrow yv$ in $\K$ is \emph{exact} when the
  above $2$-cell in $\DFib{\K}$ is invertible.\end{definition}

\begin{example}\label{ex:exactness-in-examples}
  More generally, we could define a notion of exactness in $\K$ with respect to any proarrow equipment $(-)_\diamond \colon \K \to \bicat M$; when unspecified, however, we will assume the equipment provided by $\DFib{\K}$. In particular, reconsidering our running examples from the previous section:
  \begin{itemize}
    \item In $\Set$, with respect to the equipment provided by the locally posetal category of relations, a square $xu = yv$ is exact if and only if it is a
    \emph{weak pullback}, meaning that for every $b \in B$ and $c \in C$ such
    that $x(b) = y(c)$, there is some $a \in A$ such that $x = u(a)$ and $c =
    v(a)$. In other words, the universal function from $A$ to the pullback of
    $x$ and $y$ is surjective.

  \item \looseness=-1 In $\Pos$, with respect to the equipment provided by the locally posetal category of strengthening-closed relations, a square $xu \le
    yv$ is exact if and only if, for every $b \in B$ and $c \in C$ such that
    $x(b) \le_D y(c)$, there is some $a \in A$ such that $b \le_A u(a)$ and
    $v(a) \le_C c$~\cite[Def. 5.1]{bilkovaRelationLiftingsPreorders2011}. The
    square $xu \le yv$ is then said to have the \emph{interpolation property}
    \cite[Def.\ 3.1]{goolDualityModelTheory2024}.

    \item In $\Cat$, with respect to the equipment provided by $\Prof$, the square $\gamma \colon x u \Rightarrow y v$ is exact if and only if, for every $b \in B$ and $c \in C$, the canonical function
  \[ \int^{a \in A} C(v(a), c) \times B ( b, u(a) ) \to D(x(b), y(c)),\]
defined by mapping (the equivalence class of) a triple $\braket{a, \phi \colon v(a) \to c, \psi\colon b \to u(a)}$ to the composite $y(\phi) \circ \gamma_a \circ x(\psi)  $, is a bijection. In particular, a square is exact in $\coop \Cat$ (with respect to the equipment provided by $\DFib{\coop\Cat}\simeq \coop\Prof$) if and only if it is so in $\Cat$.
  \end{itemize}
\end{example}

The following family of exact squares will play an important role in our proofs.

\begin{definition}
  An arrow $f$ in $\K$ is \emph{co-fully-faithful} if the square $1f = 1f$ is
  exact.
\end{definition}

\begin{example}With respect to the notions of exactness of \Cref{ex:exactness-in-examples}:
  \begin{itemize}
    \item  in $\Set$ and $\Pos$, a function is co-fully-faithful if and
    only if it is surjective;
    \item in $\Cat$, co-fully-faithful functors are also called \emph{lax
      epimorphisms} in \cite{adamekFunctorsLaxEpimorphisms2001}, and are shown
    there to correspond to \emph{absolutely dense} functors;
    \item in $\coop\Cat$, an
    arrow is co-fully-faithful if and only if it is fully-faithful as a functor in $\Cat$.
  \end{itemize}
\end{example}

\begin{lemma} \label{lemma:left-class-exact} In a regular bicategory,
  $\E$-arrows are co-fully-faithful.
\end{lemma}

A second important class of exact squares is the following.

\begin{lemma}\label{lemma:bicommas-exact}
    Bicomma squares are exact.
\end{lemma}

From this last fact it follows that:

\begin{corollary} \label{cor:composition-with-bicomma} Composition of TSDFs can
  be performed by means of bicomma squares instead of bipullbacks.
\end{corollary}

\subsection{Extension theorems}\label{ssec:extension-theorem}

We may now state our extension result which, for the sake of readability, we
split into three parts. First, we treat \emph{extensions of pseudofunctors}:
assuming preservation of co-fully-faithful arrows we can construct oplax
extensions, while we can characterize their pseudofunctoriality in terms of
\emph{preservation of exact squares}.

\begin{definition}
  Let $\ps F\colon \K \to \K'$ be a pseudofunctor between regular bicategories. We say that $\ps F$ satisfies the \emph{Beck-Chevalley condition} if it preserves exact squares, that is, if for every exact square $\gamma \colon xu \Rightarrow yv$ in $\K$, the square $\ps F \gamma \colon (\ps F x ) (\ps F u) \Rightarrow (\ps F y ) (\ps F v)$ is exact in $\K'$.
\end{definition}

\addtocounter{theorem}{1}
\begin{subtheorem}\label{thm:extension-pseudofunctors}
  Let $\ps F\colon \K \to \K'$ be a pseudofunctor between regular bicategories.
  \begin{enumerate}
  \item \label{thm:extension-pseudofunctors:oplax} If $\ps F$ preserves
    co-fully-faithful arrows, then it \emph{extends} to an oplax functor
    $\underline{\ps F}\colon \DFib{\K} \to \DFib{\K'}$, in the sense that there
    is an oplax natural transformation $\delta^{\ps F}\colon \u {\ps
      F}(-)_\diamond \Rightarrow (\ps F-)_\diamond$ having identity components
    on objects\footnote{Such a transformation is also called an \emph{icon}, see
      \cite{lackIcons2010}.}. In particular, we define $\u{\ps F}$ on arrows as
    \[ \underline{\ps F}(A \xleftarrow{q} E \xrightarrow{p} A) = (\ps F
      p)_\diamond (\ps F q)^\diamond \] \item \label{thm:extension-pseudofunctors:pseudo} Moreover, if $\ps F$
    satisfies the Beck-Chevalley condition, then $\u {\ps F}$ is a pseudofunctor
    and $\delta^{\ps F}$ is pseudonatural.
  \item\looseness=-1 \label{thm:extension-pseudofunctors:converse} Conversely,
    if there exist a pseudofunctor $\ps F': \DFib{\K} \to \DFib{\K}$ and a
    pseudonatural transformation $\ps F'(-)_\diamond \Rightarrow (\ps F
    -)_\diamond$ having identity components, then $\ps F$ satisfies the
    Beck-Chevalley condition, in which case $\ps F'$ is equivalent to $\u{\ps
      F}$.
  \end{enumerate}
\end{subtheorem}

\begin{proof}[Proof sketch]
  We can deduce this from the following, more general claim. Let $\K$ be a
  regular bicategory and let $\ps G\colon \K \to \bicat M$ be a pseudofunctor
  such that:
  \begin{itemize}
  \item for every $\K$-arrow $f\colon A \to B$, the $\bicat M$-arrow $\ps G f$
    has a right adjoint $(\ps G f)^r$;
  \item if $h: A \to B$ is co-fully-faithful, the corresponding $2$-cell $(\ps G
    h)(\ps G h)^r \Rightarrow 1_B$ is invertible.
  \end{itemize}
  Then, there exists an oplax functor $\underline{\ps G}\colon \DFib{\K} \to
  \bicat M$ equipped with an oplax natural transformation $\delta\colon
  \underline{\ps G}(-)_\diamond \Rightarrow \ps G$ having identity components.
  The setting of \eqref{thm:extension-pseudofunctors:oplax} is recovered by
  taking $\ps G$ to be the composite $(\ps F-)_\diamond \colon \K \to
  \DFib{\K'}$. In particular, we will construct the structural $2$-cells
  witnessing oplax functoriality and oplax naturality by means of $2$-cells
  between TSDFs induced by images of bicomma squares. In the case of
  \eqref{thm:extension-pseudofunctors:pseudo}, we will assume moreover that the
  $2$-cell $(\ps G u) (\ps G v)^r \Rightarrow (\ps G x)^r (\ps G y)$ induced by
  an exact square $xu \Rightarrow yv$ is invertible --- which, for $\ps G = (\ps
  F-)_\diamond \colon \K \to \DFib{\K'}$, corresponds to $\ps F$ preserving
  exact squares; in that case, since bicomma squares are exact, all structural
  $2$-cells constructed will be invertible. We defer the full proof of
  \Cref{thm:extension-pseudofunctors}, in particular that of
  \eqref{thm:extension-pseudofunctors:converse}, to Appendix \ref{app:A3}.

  First, we define $\u {\ps G}$ on objects by setting $\underline{\ps G}A = \ps
  F A$, and on arrows by setting:
  \[ \u{\ps G} (A \xleftarrow{q} E \xrightarrow{p} B) = (\ps G p) (\ps G q)^r
  \]Consider now a $2$-cell in $\DFib{\K}$ as $(a)$ below, and consider the
  induced diagram $(b)$ in $\bicat M$, where the action of $(-)^r$ on $2$-cells
  is induced by the universal property of adjunctions in a bicategory:
  \[(a) \quad \begin{tikzcd}[sep = 20pt]
      & E \\
      A && B \\
      & {E'} \arrow["q"', curve={height=12pt}, from=1-2, to=2-1]\arrow["p", curve={height=-12pt}, from=1-2, to=2-3]\arrow[""{name=0, anchor=center, inner sep=0}, "h"{description}, from=1-2,
      to=3-2]\arrow["{q'}", curve={height=-12pt}, from=3-2, to=2-1]\arrow["{p'}"', curve={height=12pt}, from=3-2, to=2-3]\arrow["\chi"', between={0.3}{0.7}, Rightarrow, from=0, to=2-1]\arrow["\zeta", between={0.3}{0.7}, Rightarrow, from=2-3, to=0]\end{tikzcd} \qquad\quad (b) \quad \begin{tikzcd}[sep = 40pt]
      & {\ps GE} \\
      {\ps G A} && {\ps G B} \\
      & {\ps F E'} \arrow["{\ps G p}", curve={height=-24pt}, from=1-2, to=2-3]\arrow[""{name=0, anchor=center, inner sep=0}, "{\ps G h}"{description},
      curve={height=-24pt}, from=1-2, to=3-2]\arrow["{\varepsilon_h}"{description}, between={0.3}{0.7}, Rightarrow,
      from=1-2, to=3-2]\arrow["{(\ps G q)^r}", curve={height=-24pt}, from=2-1, to=1-2]\arrow["{(\ps G{q'})^r}"', curve={height=24pt}, from=2-1, to=3-2]\arrow[""{name=1, anchor=center, inner sep=0}, "{(\ps G h)^r}"{description},
      curve={height=-24pt}, from=3-2, to=1-2]\arrow["{\ps G{p'}}"', curve={height=24pt}, from=3-2, to=2-3]\arrow["{(\ps G \chi)^r}", between={0.3}{0.7}, Rightarrow, from=2-1, to=1]\arrow["{\ps G \zeta}", between={0.3}{0.7}, Rightarrow, from=2-3, to=0]\end{tikzcd} \]We set the image of $(a)$ under $\underline{\ps G}$ to be the $2$-cell in
  $\bicat M$ obtained by pasting:
  \[ (\ps G p') \counit_h (\ps G q')^r \cdot (\ps G \zeta) (\ps G \chi)^r \]This action on $2$-cells of $\DFib{\K}$ is well-defined as it gives the same
  value for any two isomorphic morphisms of spans. By the coherence in $\bicat M$
  and by pseudofunctoriality of $\ps G$, one can see that $\underline{\ps G}$
  (strictly) preserves identity $2$-cells and their compositions, so that it
  defines a functor $\DFib{\K}(A,B) \to \bicat M(\ps G A, \ps G B)$ for every two
  objects $A, B$ of $\K$.

  For an object $A$ of $\K$, the $2$-cell $\underline{\ps G}(A \xleftarrow{\cod}
  A^\two \xrightarrow{\dom} A) \Rightarrow 1_{\ps G A}$, witnessing preservation
  of identities, is constructed as the composite
  \[ (\ps G \dom) (\ps G \cod)^r \Rightarrow (\ps G 1_A)^r (\ps G 1_A) \cong
    1_{\ps G A} \]where the first step is induced by the image of the bicomma square
  $1_A \dom \Rightarrow 1_A \cod$, and the second one by coherence in $\bicat M$
  and pseudofunctoriality of $\ps G$.

  Consider then two TSDFs and their composite:
  \[ e = A \xleftarrow{q} E \xrightarrow {p} B \quad f = B \xleftarrow{s} \ps G
    \xrightarrow{r} C \quad fe = A \xleftarrow{y} G \xrightarrow{x} C \]Let $E \xleftarrow{\cod} \bicomma{s}{p} \xrightarrow{\dom} F$ be the bicomma
  span of the cospan $E \xrightarrow{p} B \xleftarrow{s} F$, and write $h\colon
  \bicomma{s}{p} \to G$ for the arrow factoring $A \xleftarrow{q \cod}
  \bicomma{s}{p} \xrightarrow{p \dom} C$ into $gf$, which by
  \Cref{cor:composition-with-bicomma} lies in $\mathcal E$:
  \[\begin{tikzcd}[row sep=1pt, column sep=large]
      && {\bicomma{s}{p}} \\
      & E && F \\
      A && G && C \arrow["\cod"', from=1-3, to=2-2]\arrow["\dom", from=1-3, to=2-4]\arrow[""{name=0, anchor=center, inner sep=0}, "h", dashed, from=1-3,
      to=3-3]\arrow["q"', from=2-2, to=3-1]\arrow["r", from=2-4, to=3-5]\arrow["y", from=3-3, to=3-1]\arrow["x"', from=3-3, to=3-5]\arrow["\cong"{description, pos=0.2}, draw=none, from=0, to=3-5]\arrow["\cong"{description, pos=0.8}, draw=none, from=3-1, to=0]\end{tikzcd}\]The $2$-cell $\underline{\ps G}(fe) \Rightarrow (\underline{\ps
    G}f)(\underline{\ps G}e)$, witnessing preservation of composition, is
  constructed as the composite
  \begin{align*}
    (\ps G x) (\ps G y)^r
    &\cong (\ps G x) (\ps G h) (\ps G h)^r (\ps G y)^r \\
    &\cong (\ps G r)(\ps G \dom)(\ps G \cod)^r(\ps G q)^r \\
    &\Rightarrow (\ps G r)(\ps G s)^r(\ps G p)(\ps G q)^r
  \end{align*}
  using preservation of the co-fully-faithful morphism $h$ for the first step,
  pseudofunctoriality of $\ps G$ for the second step, and the $2$-cell induced
  by the image of the bicomma square $s \dom \Rightarrow p \cod$ for the third
  step.

  Finally, we construct the transformation $\delta \colon
  \underline{\ps G}(-)_\diamond \Rightarrow \ps G$. For an object $A$ of $\K$,
  we simply let $\delta_A \colon \ps G A \to \ps G A$ be the identity $1_{\ps G A}$, while for an arrow $f\colon A \to B$ in $\K$, setting $f_\diamond = A
  \xleftarrow{\cod} \bicomma{1_B}{f} \xrightarrow{\dom} B$, we let
  $\delta_f\colon \underline{\ps G} (f_\diamond) \Rightarrow \ps G f$ be the
  $2$-cell induced by the image of the bicomma square $1_B \dom \Rightarrow f
  \cod$.
\end{proof}

We now move to \emph{extensions of pseudonatural transformations}. While oplax
extensions always exist, we can characterize their pseudonaturality in terms of
\emph{exactness of naturality squares}.

\begin{definition}
  Let $\ps F, \ps F' \colon \K \to \K'$ be pseudofunctors between regular
  bicategories and let $\gamma\colon \ps F \Rightarrow \ps F'$ be a
  pseudonatural transformation. We say that $\gamma$ satisfies the
  \emph{Beck-Chevalley condition} if its naturality squares are exact, that is,
  if for every arrow $f \colon A \to B$ in $\K$ the square $\gamma_f \colon
  (\gamma_B) (\ps F f) \cong (\ps F' f ) (\gamma_A)$ is exact in $\K'$.
\end{definition}

\addtocounter{theorem}{-1}
\addtocounter{subtheorem}{1}
\begin{subtheorem}\label{thm:extension-psnat-trans}
  Let $\gamma\colon \ps F \Rightarrow \ps F'$ be a pseudonatural transformation
  between pseudofunctors.
  \begin{enumerate}
  \item Suppose $\ps F$ and $\ps F'$ have oplax extensions in the sense of
    \Cref{thm:extension-pseudofunctors}\eqref{thm:extension-pseudofunctors:oplax}.
    Then, there is an oplax natural transformation $\u\gamma \colon \u{\ps F}
    \Rightarrow \u{\ps F}'$ such that $\u{\gamma}_A = (\gamma_A)_\diamond$ and
    \[\begin{tikzcd}[cramped, column sep = 9pt, row sep=10pt]
        {\DFib{\K}} && {\DFib{\K'}} \\
        \\
        \K && {\K'} \arrow[""{name=0, anchor=center, inner sep=0}, "{{\u{ \ps F}}}",
        curve={height=-18pt}, from=1-1, to=1-3]\arrow[""{name=1, anchor=center, inner sep=0}, "{{\u{\ps
              F}'}}"{description}, curve={height=18pt}, from=1-1, to=1-3]\arrow[""{name=2, anchor=center, inner sep=0}, "{{(-)_\diamond}}",
        from=3-1, to=1-1]\arrow["{{\ps F'}}"', curve={height=18pt}, from=3-1, to=3-3]\arrow["{{(-)_\diamond}}"', from=3-3, to=1-3]\arrow["{{\u{\gamma}}\,}"', between={0.2}{0.8}, Rightarrow, from=0,
        to=1]\arrow["{{\delta^{\ps F'}}}"'{pos=0.7}, shift right, between={0.4}{0.8},
        Rightarrow, from=2, to=3-3]\end{tikzcd} = \begin{tikzcd}[cramped, column sep=9pt, row sep=10pt]
        {\DFib{\K}} && {\DFib{\K'}} \\
        \\
        \K && {\K'} \arrow["{{\u{ \ps F}}}", curve={height=-18pt}, from=1-1, to=1-3]\arrow["{{(-)_\diamond}}", from=3-1, to=1-1]\arrow[""{name=0, anchor=center, inner sep=0}, "{{\ps F}}"{description},
        curve={height=-18pt}, from=3-1, to=3-3]\arrow[""{name=1, anchor=center, inner sep=0}, "{{\ps F'}}"',
        curve={height=18pt}, from=3-1, to=3-3]\arrow[""{name=2, anchor=center, inner sep=0}, "{{(-)_\diamond}}"',
        from=3-3, to=1-3]\arrow["{{\delta^{\ps F}}}"{pos=0.1}, between={0}{0.5}, Rightarrow,
        from=1-1, to=2]\arrow["{\gamma\,}"', between={0.2}{0.8}, Rightarrow, from=0, to=1]\end{tikzcd} \]\item Moreover, suppose $\ps F$ and $\ps F'$ have extensions in the sense of
    \Cref{thm:extension-pseudofunctors}(2). Then,
    $\gamma$ satisfies the Beck-Chevalley condition if and only if $\u\gamma$ is pseudonatural, in which case
    $\u\gamma$ is the only pseudonatural transformations satisfying the above
    equalities.
  \end{enumerate}
\end{subtheorem}

Finally, we treat \emph{extensions of modifications}. As it turns out, no
further conditions are needed in this case.

\begin{subtheorem}\label{thm:extension-modifications}
  \looseness=-1 Let $\mathfrak m \colon \gamma \Rrightarrow \gamma'$ be a
  modification between pseudonatural transformations as in
  \Cref{thm:extension-psnat-trans}. Then, the $2$-cells $\u{\mathfrak m}_A
  \coloneqq (\mathfrak{m}_A)_\diamond$ define a modification $\u{\frak m} \colon
  \u \gamma \Rrightarrow \u{\gamma}'$.
\end{subtheorem}

\addtocounter{theorem}{1}
\begin{example}[Related work] \label{ex:relational-extensions-in-the-literature}

  Because we only consider TSDFs and not jointly monic or
  jointly order-embedding spans, we technically do not recover Barr's original
  relational extension theorem~\cite{barrRelationalAlgebras1970}, nor similar
  results in regular categories~\cite[\textsection
  4.3]{carboni2categoricalApproachChange1991}
  and~\cite[Cor.~1.5.7]{schubertLaxAlgebrasScenic2006}, $\Pos$~\cite[Thm.\
  5.3]{bilkovaRelationLiftingsPreorders2011}, and $\Pos$-regular
  categories~\cite[Thm. 6.9]{kurzStoneDualityRelations2021}. Still, the proofs
  of these results can all be considered instances of our $2$-dimensional proof:
  the only difference lies in which spans are considered as relations (recall
  \Cref{ex:two-sided-discrete-fibrations}).

  \looseness=-1
  These references all consider extensions of functors as both oplax or strict
  functors. It should be noted that Barr does not ask for the $\Set$-functor to
  be extended to preserve surjections, corresponding to the co-fully-faithful arrows, because this is automatic when the Axiom
  of Choice is assumed to hold; thus, in his case, an oplax extension is always possible. The question of extending natural transformations is explicitly considered only in the context of $\Set$ and that of regular categories. Dealing with modifications is instead one of the main novelties of our setting, as of course they trivialize in the
  locally posetal case.

  \looseness=-1 Our main new example is an extension theorem in the $2$-category
  of $\mathcal{V}$-enriched categories, where $\mathcal V$ is any complete and
  cocomplete symmetric monoidal closed category. Indeed,
  \Cref{thm:dfibs-bicategory} applies to $\op{\mathcal V\Cat}$ by~\cite[Lem.\
  4.3.5]{loregianCategoricalNotionsFibration2020}, and its TSDFs correspond to
  $\mathcal V$-enriched profunctors~\cite[Thm.\
  4.3.2]{loregianCategoricalNotionsFibration2020}. An analogous extension
  theorem was proved for \emph{strict} functors in the locally posetal case
  where $\mathcal{V}$ is a quantale --- a sup-lattice which is monoidal closed
  as a category --- in~\cite[Thm.\ 5.10]{bilkovaRelationLiftingApplication2013},
  where the result is also claimed, though without proof, for an abitrary base
  of enrichment $\mathcal V$.
\end{example}

\begin{remark}
  Shortly after the first version of this paper was made public, a similar
  extension theorem appeared independently in Street's
  \cite{streetHomodularPseudofunctorsBicategories2026}, who also points to
  \cite{gouzouCaracterisationDist1973,gouzouCaracterisationDistXBicategorieXSommes1975}
  for earlier, restricted versions of the result (notably for the case of $\K =
  \Cat$). While we work with stronger conditions than in
  \cite{streetHomodularPseudofunctorsBicategories2026}, two aspects of our work
  does not appear there. First, while Street only considers how to extend
  pseudofunctors to pseudofunctors, we also consider the weaker case when the
  extended functor is only oplax, and study how to also extend natural
  transformations and modifications. Second, we give explicit proofs for the
  coherence axioms in Appendix \ref{app:A3}, using the graphical language of
  \emph{string diagrams}.
\end{remark}

\subsection{Extensions of pseudomonads}

\looseness=-1
We conclude this section by characterizing when pseudomonads, which model effects in
bicategories~\cite{paquetEffectfulSemanticsBicategories2024},
extend to bicategories of TSDFs. The extension criterion is obtained by combining
\Cref{thm:extension-pseudofunctors,thm:extension-psnat-trans,thm:extension-modifications}. First, we introduce pseudomonads as a particular case of the \emph{skew monads} of \cite{streetSkewMonads2015} (see also~\cite{bungeCoherentExtensionsRelational1974,marmolejoDoctrinesWhoseStructure1997,lackCoherentApproachPseudomonads2000}). Skew monads are a generalization of (pseudo)monads where the functor, the unit and the multiplication are either \emph{lax} or \emph{oplax}, and where the structural constraints are not required to be invertible. One should not be surprised by the emergence of such a notion: concretely, these monads are already considered implicitly in \cite{barrRelationalAlgebras1970}, as the ultrafilter monad on $\Set$ only extends to a skew monad on $\Rel$.

\begin{definition}\label{def:skew-monad}
  A \emph{right skew monad} on a bicategory $\K$ consists of:
  \begin{itemize}
  \item an oplax functor  $\ps T
    \colon \bicat K \to \bicat K$, whose structural $2$-cells we denote by $\psi_A \colon \ps T 1_A \Rightarrow 1_{\ps T A}$ and $\psi_{g, f} \colon \ps T (g\circ f) \Rightarrow \ps T g \circ \ps T f$;
  \item an oplax natural transformation $\eta \colon 1 \Rightarrow \ps T$ called
    \emph{unit};
  \item an oplax natural transformation $\mu \colon \ps T ^2
    \Rightarrow\ps T$ called \emph{multiplication};
  \item three modifications
    \[\begin{tikzcd}[sep = 20pt]
	{\ps T^3} & {\ps T ^2} \\
	{\ps T^2} & {\ps T}
	\arrow["{{\ps T \mu }}", from=1-1, to=1-2]
	\arrow["{\mu  \ps T}"', from=1-1, to=2-1]
	\arrow["{\mathfrak m}"', between={0.2}{0.9}, Rightarrow, from=1-2, to=2-1]
	\arrow["{\mu }", from=1-2, to=2-2]
	\arrow["{\mu }"', from=2-1, to=2-2]
\end{tikzcd} \quad \begin{tikzcd}[sep = 20pt]
	& {\ps T } & \\
	{\ps T^2} & {\ps T} & {\ps T^2 }
	\arrow["{\eta \ps T }"', curve={height=12pt}, from=1-2, to=2-1]
	\arrow[""{name=0, anchor=center, inner sep=0}, equals, from=1-2, to=2-2]
	\arrow["{{\ps T}\eta }", curve={height=-12pt}, from=1-2, to=2-3]
	\arrow["{{{\mu }}}"', from=2-1, to=2-2]
	\arrow["{{{\mu }}}", from=2-3, to=2-2]
	\arrow["{{\mathfrak l}}"', between={0.2}{0.8}, Rightarrow, from=0, to=2-1]
	\arrow["{{\mathfrak r}}"', between={0.2}{0.8}, Rightarrow, from=2-3, to=0]
\end{tikzcd}\]
    \vspace{-10pt}
  \end{itemize}
satisfying the following coherence conditions:
\begin{enumerate}
  \item
\[\begin{tikzcd}
	{\ps T^4} & {\ps T^4} && {\ps T^3} \\
	& {\ps T^3} & {\ps T^2} & {\ps T^3} \\
	{\ps T^2} & {\ps T^2} & {\ps T} & {\ps T^2}
	\arrow[equals, from=1-1, to=1-2]
	\arrow[""{name=0, anchor=center, inner sep=0}, "{\ps T (\mu \circ \ps T \mu)}"', from=1-1, to=3-1]
	\arrow["{\mu\ps T^2}", from=1-2, to=1-4]
	\arrow[""{name=1, anchor=center, inner sep=0}, "{\ps T^2 \mu}"', from=1-2, to=2-2]
	\arrow[""{name=2, anchor=center, inner sep=0}, equals, from=1-4, to=2-4]
	\arrow["{\mu\ps T}", from=2-2, to=2-3]
	\arrow[""{name=3, anchor=center, inner sep=0}, "{\ps T \mu}"', from=2-2, to=3-2]
	\arrow[""{name=4, anchor=center, inner sep=0}, "\mu"{description}, from=2-3, to=3-3]
	\arrow["{\ps T \mu}"', from=2-4, to=2-3]
	\arrow[""{name=5, anchor=center, inner sep=0}, "{\mu \ps T}", from=2-4, to=3-4]
	\arrow[equals, from=3-1, to=3-2]
	\arrow["\mu"', from=3-2, to=3-3]
	\arrow["\mu", from=3-4, to=3-3]
	\arrow["\psi", between={0.3}{0.7}, Rightarrow, from=0, to=2-2]
	\arrow["{\mu_{\mu}}", between={0.3}{0.7}, Rightarrow, from=1, to=2]
	\arrow["{\mathfrak m}", between={0.3}{0.7}, Rightarrow, from=3, to=4]
	\arrow["{\mathfrak m}", between={0.3}{0.7}, Rightarrow, from=4, to=5]
\end{tikzcd}\quad = \quad \begin{tikzcd}
	{\ps T^4} & {\ps T^4} && {\ps T^3} \\
	&& {\ps T^3} \\
	{\ps T^2} & {\ps T^2} & {\ps T} & {\ps T^2}
	\arrow[equals, from=1-1, to=1-2]
	\arrow[""{name=0, anchor=center, inner sep=0}, "{{\ps T (\mu \circ \ps T \mu)}}"', from=1-1, to=3-1]
	\arrow["{{\mu\ps T^2}}", from=1-2, to=1-4]
	\arrow["{{\ps T \mu \ps T}}", from=1-2, to=2-3]
	\arrow[""{name=1, anchor=center, inner sep=0}, "{{\ps T (\mu \circ \mu \ps T)}}"{description}, curve={height=-6pt}, from=1-2, to=3-1]
	\arrow[""{name=2, anchor=center, inner sep=0}, "{{\mu \ps T}}", from=1-4, to=3-4]
	\arrow[""{name=3, anchor=center, inner sep=0}, "{{\ps T \mu}}"', from=2-3, to=3-2]
	\arrow[""{name=4, anchor=center, inner sep=0}, "{{\mu \ps T}}", from=2-3, to=3-4]
	\arrow[equals, from=3-1, to=3-2]
	\arrow["\mu"', from=3-2, to=3-3]
	\arrow["\mu", from=3-4, to=3-3]
	\arrow["{{\ps T \mathfrak m}}"{pos=0.3}, shift left=3, between={0.3}{0.7}, Rightarrow, from=0, to=1]
	\arrow["\psi", shift left=2, between={0.4}{0.7}, Rightarrow, from=1, to=2-3]
	\arrow["{{\mathfrak m\ps T}}", shift left=2, between={0.2}{0.8}, Rightarrow, from=2-3, to=2]
	\arrow["{{\mathfrak m}}"', between={0.3}{0.7}, Rightarrow, from=3, to=4]
\end{tikzcd} \]

\item

\[\begin{tikzcd}[column sep = 30 pt]
	{\ps T^2} && {\ps T^2} && \\
	{\ps T^2} && {\ps T^3} && {\ps T^2} \\
	&& {\ps T}
	\arrow[equals, from=1-1, to=1-3]
	\arrow[""{name=0, anchor=center, inner sep=0}, "{{\ps T1_{\ps T}}}"', from=1-1, to=2-1]
	\arrow[""{name=1, anchor=center, inner sep=0}, "{{\ps T(\mu \circ \eta\ps T)}}"{description}, curve={height=-6pt}, from=1-3, to=2-1]
	\arrow["{{\ps T\eta \ps T}}", from=1-3, to=2-3]
	\arrow[""{name=2, anchor=center, inner sep=0}, curve={height=-16pt}, equals, from=1-3, to=2-5]
	\arrow[""{name=3, anchor=center, inner sep=0}, "\mu"', curve={height=16pt}, from=2-1, to=3-3]
	\arrow["\ps T\mu", from=2-3, to=2-1]
	\arrow["{{\mu \ps T }}"', from=2-3, to=2-5]
	\arrow[""{name=4, anchor=center, inner sep=0}, "\mu", curve={height=-16pt}, from=2-5, to=3-3]
	\arrow["{{\ps T\mathfrak l}}",pos=0.3, between={0.3}{0.7}, Rightarrow, from=0, to=1]
	\arrow["{{\psi}}", between={0.4}{0.8},pos=0.6, Rightarrow, from=1, to=2-3]
	\arrow["{{\mathfrak m}}"', shift left=5, between={0.3}{0.7}, Rightarrow, from=3, to=4]
	\arrow["{{\mathfrak r \ps T}}"', between={0.2}{0.8}, Rightarrow, from=2-3, to=2]
\end{tikzcd} \quad = \quad \begin{tikzcd}
	{\ps T^2} && {\ps T^2} \\
	& {\ps T^2} \\
	& {\ps T}
	\arrow[""{name=0, anchor=center, inner sep=0}, "{{\ps T1_{\ps T}}}"', from=1-1, to=2-2]
	\arrow[equals, from=1-3, to=1-1]
	\arrow[""{name=1, anchor=center, inner sep=0}, equals, from=1-3, to=2-2]
	\arrow["\mu", from=2-2, to=3-2]
	\arrow["\psi", between={0.25}{0.7}, Rightarrow, from=0, to=1]
\end{tikzcd}\]

\item

\[\begin{tikzcd}
	& {\ps T^2} & {\ps T^3} \\
	{\ps T} & {\ps T^2} & {\ps T^2} \\
	& {\ps T}
	\arrow["{{\eta\ps T^2}}", from=1-2, to=1-3]
	\arrow[""{name=0, anchor=center, inner sep=0}, "\mu"', from=1-2, to=2-1]
	\arrow[""{name=1, anchor=center, inner sep=0}, "{{\ps T \mu}}"{description}, from=1-3, to=2-2]
	\arrow["{{\mu\ps T}}", from=1-3, to=2-3]
	\arrow["{{\eta \ps T}}", from=2-1, to=2-2]
	\arrow[""{name=2, anchor=center, inner sep=0}, equals, from=2-1, to=3-2]
	\arrow["{{\mathfrak m}}"', between={0.2}{0.8}, Rightarrow, from=2-2, to=2-3]
	\arrow["\mu", from=2-2, to=3-2]
	\arrow["\mu", from=2-3, to=3-2]
	\arrow["{{\eta_{\mu}}}", between={0.3}{0.7}, Rightarrow, from=0, to=1]
	\arrow["{{\mathfrak l}}", between={0.2}{0.8}, Rightarrow, from=2, to=2-2]
\end{tikzcd}\quad = \quad \begin{tikzcd}
	{\ps T^2} && {\ps T^3} \\
	& {\ps T^2} \\
	& {\ps T}
	\arrow["{\eta\ps T^2}", from=1-1, to=1-3]
	\arrow[""{name=0, anchor=center, inner sep=0}, equals, from=1-1, to=2-2]
	\arrow[""{name=1, anchor=center, inner sep=0}, "{\mu\ps T}", from=1-3, to=2-2]
	\arrow["\mu", from=2-2, to=3-2]
	\arrow["{\mathfrak l\ps T}", between={0.3}{0.7}, Rightarrow, from=0, to=1]
\end{tikzcd}\]

\item

\[\begin{tikzcd}
	{\ps T^2} & {\ps T^2} && {\ps T} \\
	& {\ps T^3} & {\ps T^2} & {\ps T} \\
	{\ps T^2} & {\ps T^2} & {\ps T} & {\ps T}
	\arrow[equals, from=1-1, to=1-2]
	\arrow[""{name=0, anchor=center, inner sep=0}, "{\ps T(\mu\circ \ps T\eta)}"', from=1-1, to=3-1]
	\arrow["\mu", from=1-2, to=1-4]
	\arrow[""{name=1, anchor=center, inner sep=0}, "{\ps T^2\eta}"', from=1-2, to=2-2]
	\arrow[""{name=2, anchor=center, inner sep=0}, equals, from=1-4, to=2-4]
	\arrow["{\mu\ps T}", from=2-2, to=2-3]
	\arrow[""{name=3, anchor=center, inner sep=0}, "{\ps T\mu}"', from=2-2, to=3-2]
	\arrow[""{name=4, anchor=center, inner sep=0}, "\mu"{description}, from=2-3, to=3-3]
	\arrow["{\ps T\eta}"', from=2-4, to=2-3]
	\arrow[""{name=5, anchor=center, inner sep=0}, "{\mu\ps T}", from=2-4, to=3-4]
	\arrow[equals, from=3-1, to=3-2]
	\arrow["\mu"', from=3-2, to=3-3]
	\arrow[equals, from=3-4, to=3-3]
	\arrow["\psi", between={0.3}{0.7}, Rightarrow, from=0, to=2-2]
	\arrow["{\mu_\eta}", between={0.3}{0.7}, Rightarrow, from=1, to=2]
	\arrow["{\mathfrak m}", between={0.3}{0.7}, Rightarrow, from=3, to=4]
	\arrow["{\mathfrak r}", between={0.3}{0.7}, Rightarrow, from=4, to=5]
\end{tikzcd} \quad = \quad \begin{tikzcd}
	{\ps T^2} & {\ps T^2} & {\ps T^2} \\
	& {\ps T^2} \\
	& {\ps T}
	\arrow[equals, from=1-1, to=1-2]
	\arrow[""{name=0, anchor=center, inner sep=0}, "{\ps T(\mu\circ \ps T\eta)}"', from=1-1, to=2-2]
	\arrow[equals, from=1-2, to=1-3]
	\arrow[""{name=1, anchor=center, inner sep=0}, "{\ps T1_{\ps T}}"{description}, from=1-2, to=2-2]
	\arrow[""{name=2, anchor=center, inner sep=0}, equals, from=1-3, to=2-2]
	\arrow["\mu", from=2-2, to=3-2]
	\arrow["{\ps T \mathfrak r}",pos=0.4, between={0.3}{0.7}, Rightarrow, from=0, to=1]
	\arrow["\psi", between={0.3}{0.7}, Rightarrow, from=1, to=2]
\end{tikzcd}\]

\item

\[\begin{tikzcd}
	& 1 & \\
	{\ps T} & {\ps T^2} & {\ps T} \\
	{\ps T} & {\ps T} & {\ps T}
	\arrow[""{name=0, anchor=center, inner sep=0}, "\eta"', from=1-2, to=2-1]
	\arrow[""{name=1, anchor=center, inner sep=0}, "\eta", from=1-2, to=2-3]
	\arrow["{\eta\ps T}", from=2-1, to=2-2]
	\arrow[""{name=2, anchor=center, inner sep=0}, equals, from=2-1, to=3-1]
	\arrow[""{name=3, anchor=center, inner sep=0}, "\mu"{description}, from=2-2, to=3-2]
	\arrow["{\ps T\eta}"', from=2-3, to=2-2]
	\arrow[""{name=4, anchor=center, inner sep=0}, equals, from=2-3, to=3-3]
	\arrow[equals, from=3-1, to=3-2]
	\arrow[equals, from=3-3, to=3-2]
	\arrow["{\eta_\eta }", between={0.3}{0.7}, Rightarrow, from=0, to=1]
	\arrow["{\mathfrak l}", between={0.3}{0.7}, Rightarrow, from=2, to=3]
	\arrow["{\mathfrak r}", between={0.3}{0.7}, Rightarrow, from=3, to=4]
\end{tikzcd}\quad = \quad \begin{tikzcd}
	1 \\
	{\ps T}
	\arrow[""{name=0, anchor=center, inner sep=0}, "\eta"', curve={height=12pt}, from=1-1, to=2-1]
	\arrow[""{name=1, anchor=center, inner sep=0}, "\eta", curve={height=-12pt}, from=1-1, to=2-1]
	\arrow["1_\eta", between={0.2}{0.8}, Rightarrow, from=0, to=1]
\end{tikzcd}\]
\end{enumerate}

We also define:
\begin{itemize}
\item a \emph{left skew monad} on $\K$ as a right skew monad on $\co \K$ --- concretely, this amounts to considering lax functors, lax transformations, and modifications in the opposite direction;
\item a \emph{right skew comonad} on $\K$ as a right skew monad on $\op \K$ --- concretely, this amounts to considering an oplax functor $\ps T$, lax natural \emph{counit} $\ps T \Rightarrow 1$ and \emph{comultiplication} $\ps T \Rightarrow\ps T^2$, and modifications in the direction above;
\item a \emph{left skew comonad} on $\K$ as a right skew monad on $\coop \K$ --- concretely, this amounts to considering lax functors, oplax natural counit and comultiplication, and modifications in the opposite direction.
\end{itemize}

In particular, a skew (co)monad $\braket{\ps T, \eta,\mu}$ is a
\emph{pseudo(co)monad} if $\ps T$ is a pseudofunctor, the transformations $\eta$
and $\mu$ are pseudonatural, and the modifications $\mathfrak m$, $\mathfrak l$
and $\mathfrak r$ are invertible (see
\cite{lackCoherentApproachPseudomonads2000,marmolejoDoctrinesWhoseStructure1997}).
\end{definition}

\begin{corollary}\label{cor:extension-pseudomonads}
\looseness=-1 Let $\monad{\ps T}$ be a pseudomonad on a regular bicategory $\K$.
\begin{enumerate}
\item \label{cor:extension-pseudomonads:right-monad} If $\ps T$ preserves
  co-fully-faithful arrows, then the pseudomonad $\monad{\ps T}$ extends to a
  right skew monad $\monad{\u{\ps T}}$ on $\DFib{\K}$.
\item \label{cor:extension-pseudomonads:pseudo-monad} In particular, the skew
  monad $\monad{\u{\ps T}}$ is a pseudomonad if and only if $\ps T$, $\eta^{\ps
    T}$ and $\mu^{\ps T}$ all satisfy the Beck-Chevalley condition.
\end{enumerate}
Dually, let $\monad{\ps T}$ be a pseudocomonad on $\K$.
\begin{enumerate}[resume]
\item \label{cor:extension-pseudomonads:right-comonad} If $\ps T$ preserves
  co-fully-faithful arrows, and $\eta^{\ps T}$ and $\mu^{\ps T}$ satisfy the
  Beck-Chevalley condition, then $\monad{\ps T}$ extends to a right skew comonad
  $\monad{\u{\ps T}}$ on $\DFib{\K}$.
\item \label{cor:extension-pseudomonads:pseudo-comonad} In particular, the skew
  comonad $\monad{\u{\ps T}}$ is a pseudomonad if and only if $\ps T$ also
  satisfies the Beck-Chevalley condition.
\end{enumerate}
\end{corollary}

\begin{proof}
  We first describe how the triple $\monad{\u{\ps T}}$ is constructed. First,
  since $\ps T$ preserves co-fully-faithful arrows, we can consider its oplax
  extension $\u{\ps T} \colon \DFib{\K} \to \DFib{\K}$ in the sense of
  \Cref{thm:extension-pseudofunctors}. The unit $\eta^{\ps T} \colon 1_{\K }
  \Rightarrow \ps T$ and the multiplication $\mu^{\ps T} \colon \ps T^2
  \Rightarrow \ps T$ then extend to oplax natural transformations $\u{\eta^{\ps
      T}} \colon \u {1}_{\K} \Rightarrow \u{\ps T}$ and $\u{\mu^{\ps T}}\colon
  \u{(\ps T^2)}\Rightarrow \u{\ps T}$ via \Cref{thm:extension-psnat-trans}. To
  obtain a unit and a multiplication for $\u{\ps T}$, note that there are two
  canonical oplax natural transformations $\omega^1 \colon 1_{\DFib{\K}}
  \Rightarrow \u{1}_{\K}$ and $\omega^{\ps T} \colon \u{\ps T}^2 \Rightarrow
  \u{(\ps T^2)}$ having identity components, which we can thus compose with
  $\u{\eta^{\ps T}}$ and $\u{\mu^{\ps T}}$. We now describe how they are defined.

  Fix a TSDF $e = A \xleftarrow{q} E \xrightarrow{p} B$. First, the structural
  $2$-cell $\omega^1_e$ is the canonical isomorphism $e \cong p_\diamond
  q^\diamond$ of \cite[Prop. 4.25]{carboniModulatedBicategories1994}, so in
  particular $\omega^1$ is pseudonatural. Note then that, in addition to the
  oplax natural transformation $\delta^{\ps T}: \u{\ps T}(-)_\diamond
  \Rightarrow (\ps T -)_\diamond$, there is also, dually, an oplax natural
  transformation $\theta^{\ps T}: \u{\ps T}(-)^\diamond \Rightarrow (\ps T
  -)^\diamond$. Recall now that, by construction, $\u{\ps T}e = (\ps T
  p)_\diamond (\ps T q)^\diamond$, and $\u{(\ps T^2)}e = (\ps T^2 p)_\diamond
  (\ps T^2 q)^\diamond$. Therefore, we set $\omega^{\ps T}_e$ as the composite
  $2$-cell:
  \[\begin{tikzcd}[column sep=huge]
      {\u{\ps T}((\ps T p)_\diamond (\ps T q)^\diamond)} &[-10pt] {\u{\ps T}((\ps T
        p)_\diamond) \u{\ps T}((\ps T q)^\diamond)} & {(\ps T^2 p)_\diamond (\ps
        T^2 q)^\diamond} \arrow["\psi_{(\ps T p)_\diamond, (\ps T q)^\diamond}",Rightarrow, from=1-1, to=1-2]\arrow["{\delta^{\ps T}_{\ps T p} \theta^{\ps T}_{\ps T
          q}}", Rightarrow, from=1-2, to=1-3]\end{tikzcd}\]where we denote by $\psi$ the oplax associator of $\u{\ps T}$.

  The structural modifications of $\monad{\u{\ps T}}$ then arise
  straightforwardly by those of $\monad{\ps T}$ via
  \Cref{thm:extension-modifications} --- so, in particular, they are invertible
  --- and the validity of the coherence conditions also follows directly by the
  analogous conditions for $\monad{\ps T}$.

  To show \eqref{cor:extension-pseudomonads:pseudo-monad}, since its structural
  modifications are invertible, note first that the triple $\monad{\u{\ps T}}$
  is a pseudomonad if and only if $\u{\ps T}$ is a pseudofunctor and the
  transformations $\eta^{\u{\ps T}}$ and $\mu^{\u {\ps T}}$ are pseudonatural.
  By \Cref{thm:extension-pseudofunctors}, the first condition is equivalent to
  the Beck-Chevalley condition for $\ps T$. In that case, the transformation
  $\omega^{\ps T}$ is pseudonatural since so are $\delta^{\ps T}$ and
  $\theta^{\ps T}$, so that by \Cref{thm:extension-psnat-trans} the second
  condition becomes equivalent to the Beck-Chevalley condition for $\eta^{\ps
    T}$ and $\mu^{\ps T}$.

  The dual result about comonads is analogous: the only difference is that, in
  order to construct the lax natural counit and comultiplication for $\u{\ps
    T}$, we need the Back-Chevalley condition to hold \emph{a priori} for
  $\eta^{\ps T}$ and $\mu^{\ps T}$.
\end{proof}

 \section{Pseudomonads on \texorpdfstring{$\CAT$}{CAT} and extensions to \texorpdfstring{$\PROF$}{PROF}}\label{sec:04}

\looseness=-1 In this section, we focus on the case of
$\CAT$, the $2$-category of locally small categories. As recalled in
\Cref{ex:bicats-of-tsdfs}, $\coop\Cat$ is a regular bicategory, and
the equipment $\Cat \to \Prof$ is retrieved as $\coop\Cat \to {\DFib{\coop\Cat}}$.
In the same way, $\coop\CAT$ is a regular bicategory, and
$\coop{\DFib{\coop\CAT}}$ is biequivalent to $\PROF$, the bicategory of locally
small categories and \emph{small} profunctors: a profunctor $F\colon C \pro D$
is small if, for every object $c$ of $C$, the presheaf $F(-,c)\colon \op D \to
\Set$ is a small colimit of representables\footnotemark{} (in which case we say
that this presheaf is \emph{small}).

Note that because we start with pseudomonads on $\CAT$, we get
pseudo\emph{co}monads on $\coop\CAT$. When the assumptions of
\Cref{cor:extension-pseudomonads}\eqref{cor:extension-pseudomonads:right-comonad} are satisfied, these extend to right skew
comonads on $\coop\PROF$, which in turn correspond to left skew monads on
$\PROF$.

\footnotetext{A presheaf is \emph{representable} if it is of the form $D(-,d)$ for some $d$ in $D$.}

\looseness=-1 We begin by introducing a wide class of pseudomonads on $\CAT$
which extend to skew monads on $\PROF$ via \Cref{cor:extension-pseudomonads}.
One of these pseudomonads is the $2$-dimensional version of the ultrafilter
monad we will focus on in \Cref{sec:05}. More generally, these pseudomonads are
given, modularly as higher-dimensional analogues of $\Set$-monads, and thus are a
natural choice to model effects on $\CAT$. In \Cref{ssec:psdist-laws} we then
tie our extension result with the theory of pseudodistributive laws.

\subsection{A class of examples}\label{ssec:loke}
\looseness=-1 Technically, the pseudomonads we now consider arise as \emph{left
  oplax Kan extensions}~\cite{tarantinoUltracategoriesKanExtensions2025} of
\emph{relative
  2-monads}~\cite{altenkirchMonadsNeedNot2010,fioreRelativePseudomonadsKleisli2018,arkorBicategoriesAlgebrasRelative2025}.
To avoid discussing relative $2$-monads and oplax Kan extensions entirely, we
focus on a simple case of interest: namely, those relative $2$-monads obtained
by composing $\Set$-based monads with the inclusion $\Set \hookrightarrow\CAT$
which views sets as discrete categories. For ease of notation, we leave this inclusion implicit in the
following. 

\begin{definition}\label{def:left-oplax-kan-extension}
  Let $T$ be a monad on $\Set$. For a category $C$, denote by ${\ps L}_T
  C$ the category having:
  \begin{itemize}
  \item as objects, triples $(X, h, \nu)$ of a set $X$, a functor $h \colon X
    \to C$, and an element $\nu \in T(X)$;
  \item as morphisms $(X,h,\nu) \to (X',h',\nu')$, pairs of a function $g \colon
    X' \to X$ such that $Tg(\nu') = \nu$ and of a (natural) transformation $\alpha \colon
    h \circ g \Rightarrow h'$, i.e.\ simply an $X'$-indexed family of $C$-arrows $(\alpha_x \colon hg(x)\to h'(x))_{x\in X'}$.
  \end{itemize}

  This definition naturally extends to a pseudofunctor\footnote{In fact, a \emph{$2$-functor}.} $\ps L_T\colon \CAT\to \CAT$ where, for a functor $f\colon C \to D$, we set $\ps L_Tf (X,h,\nu) = (X,  f\circ h,
  \nu)$. Moreover, $\ps L_T$ carries the structure of a pseudomonad~\cite[Thm.\
  4.13]{tarantinoUltracategoriesKanExtensions2025}, described in Appendix
  \ref{app:A4}.
\end{definition}

\begin{example}
  Consider the following $\Set$-monads:
  \begin{itemize}
  \item the \emph{powerset monad} $\mathcal P$, which takes a set $X$ to the set of
    its subsets, and a function $f\colon X \to Y$ to the function $\mathcal
    P(f)$ mapping $\nu \subseteq X$ to the direct image $f[\nu]\subseteq Y$;
  \item the \emph{finite distribution monad} $\mathcal D$, which takes a set $X$ to
    the set of finite distributions on $X$ --- i.e., finitely-supported
    functions $X \to [0,1]$ having total mass $1$ --- and a function $f\colon X
    \to Y$ to the function $\mathcal D (f)$ mapping $\nu \colon X \to [0,1]$ to the distribution on $Y$ defined by $y \mapsto \sum_{x
      \in f^{-1}(y)} \nu(x)$;
  \item\looseness=-1 the \emph{filter monad} $\mathcal F$, which takes a set $X$ to the set of
    (possibly improper) filters on $X$ --- i.e., nonempty families of subsets of
    $X$ closed under binary intersection and enlargement --- and a function
    $f\colon X \to Y$ to the function $\mathcal F(f)$ mapping a filter $\nu$ on $X$ to the filter on $Y$ defined by $\set{
      Y_0 \subseteq Y | f^{-1}(Y_0) \in \nu }$;
  \item the \emph{ultrafilter monad} $\beta$, defined by restricting the filter
    monad $\mathcal F$ to ultrafilters --- i.e., maximal filters.
  \end{itemize}

  \looseness=-1
  Unfolding the definitions, objects of ${\ps L}_{\mathcal P} C$, ${\ps
    L}_{\mathcal D}$, ${\ps L}_{\mathcal F} C$ and ${\ps L}_{\beta} C$ can
 be thought of, respectively, formal subsets, probability distributions,
  filters and ultrafilters of objects of $C$.
\end{example}

\begin{wrapfigure}{r}{0.2\linewidth}
      \vspace{-0.2em}\hspace{1pt}
  \begin{tikzcd}[column sep = 12pt, row sep = 14pt,cramped]
    {A} && {A} \\
    & {\ps T A}
    \arrow[""{name=0, anchor=center, inner sep=0}, equals, from=1-1, to=1-3]
    \arrow["{\eta_{A}^{\ps T}}"'{inner sep=.8ex}, from=1-1, to=2-2]
    \arrow["a"'{inner sep=.8ex}, from=2-2, to=1-3]
    \arrow["\,\Gamma", between={0.2}{0.8}, Rightarrow, from=0, to=2-2]
  \end{tikzcd}

  \vspace{8pt}

  \hspace{-8pt}
  \begin{tikzcd}[row sep = 22pt,column sep = 30pt,cramped]
    {\ps T^2 A} & {\ps T A} \\
    {\ps T A} & {A}
    \arrow["{\ps T a}"{inner sep=.8ex}, from=1-1, to=1-2]
    \arrow["{\mu_{A}^{\ps T}}"'{inner sep=.8ex},  from=1-1, to=2-1]
    \arrow["\Delta"', between={0.2}{0.8}, Rightarrow, from=1-2, to=2-1]
    \arrow["a"{inner sep=.8ex},  from=1-2, to=2-2]
    \arrow["a"'{inner sep=.8ex},  from=2-1, to=2-2]
  \end{tikzcd}
  \vspace{-10pt}
\end{wrapfigure}
Intuitively, the pseudomonad $\ps L_T$ is thus a $2$-dimensional version of the monad $T$. This can also
be stated formally, as we explain now.

Recall that an (Eilenberg-Moore) algebra for a $\Set$-monad $T$ is a function
$TA \to A$ compatible with the monad structure, so that it can be thought of as
evaluating effectful contexts to a value. This notion can be extended to the
$2$-dimensional setting by relaxing the monad compatibility equations to
appropriate $2$-cells. For instance, a \emph{lax algebra} for a pseudomonad
$\monad {\ps T}$ on a bicategory $\K$ is defined by an arrow $a \colon \ps T A
\to A$ in $\K$ together with (coherent) $2$-cells as on the right. $2$-cells in the opposite direction define \emph{colax
  algebras}, while  invertible ones define \emph{pseudoalgebras}.
Similarly we can define \emph{lax}, \emph{colax}, and \emph{pseudomorphisms};
evident notions of \emph{algebra $2$-cells} then yield bicategories of all
combinations of algebras and their morphisms. For a detailed account of these definitions, we refer the reader to, e.g., \cite[\S 2.4]{stepan-thesis} (but see also \Cref{rem:lax-algebras-are-okay,rem:colax-morphisms-are-okay}).

\looseness=-1
That ${\ps L}_T$ is a $2$-dimensional version of $T$ can then be formalized by the following lemma,
expressing how its pseudoalgebras expand the algebras of $T$ with an additional dimension.

\begin{lemma}[\protect{\cite[Ex.\ 1.18]{tarantinoUltracategoriesKanExtensions2025}}]\label{lem:discrete-algebras-of-loke}
  The category $\Alg{T}$ of $T$-algebras embeds into the $2$-category $\PsAlg{\ps
    L_T}$ of pseudo-$\ps L_T$-algebras, pseudomorphisms, and algebra $2$-cells as
  the locally discrete full sub-$2$-category spanned by those algebras whose carrier category is small and discrete.
\end{lemma}

\begin{example}\label{ex:discrete-algebras-of-loke}In the four cases above, \Cref{lem:discrete-algebras-of-loke} specializes to embeddings:
  \begin{itemize}
  \item of $\mathcal
    P$-algebras, i.e.\ complete
    join-semi-lattices, into pseudo-${\ps L}_{\mathcal P}$-algebras;
  \item of $\mathcal
    D$-algebras, i.e.\ {convex spaces}
    (see, e.g., \cite{fritzConvexSpacesDefinition2015}), into pseudo-${\ps L}_{\mathcal D}$-algebras;
  \item of $\mathcal
    F$-algebras, i.e.\ continuous lattices~\cite{scottContinuousLattices1972}
    by \cite{dayFilterMonadsContinuous1975}, into pseudo-${\ps L}_{\mathcal F}$-algebras;
  \item of $\beta$-algebras, i.e.\ compact Hausdorff
    spaces by~\cite{manesTripleTheoreticConstruction1969}, into pseudo-${\ps L}_{\beta}$-algebras.
  \end{itemize}
\end{example}

\Cref{cor:extension-pseudomonads} applies to any pseudomonad of the form
${\ps{L}_T}$. 

\begin{theorem}\label{thm:loke-extend}
  For any monad $T$ on $\Set$, ${\ps L_T}$ extends to a left skew monad
  ${\u{\ps L}_T}$ on $\PROF$.
\end{theorem}
\begin{proof}
  The claim follows by
  \Cref{cor:extension-pseudomonads}.\eqref{cor:extension-pseudomonads:right-comonad}: we show here that the $2$-functor $\ps L_T$ preserves fully-faithful functors, while a proof that the Beck-Chevalley condition holds for its unit and multiplication is deferred to Appendix \ref{app:A4}
  (\Cref{prop:lokes-allow-for-lax-algebras-2,prop:lokes-allow-for-lax-algebras-3}).
  Suppose $f \colon C \to D$ is fully-faithful and consider two objects
  $(X,h,\nu)$ and $(X',h',\nu')$ in $\ps L_TC$. Then, an $\ps L_T D$-morphism
  $\ps L_Tf (X,h,\nu) \to \ps L_T f (X',h',\nu')$ consists of a function $g
  \colon X' \to X$ such that $T g (\nu') = \nu$ and a natural transformation
  $\alpha \colon f h g \Rightarrow f h'$: by fully-faithfulness of $f$, such a
  transformation $\alpha$ is given by $f*\bar\alpha$ for a unique transformation
  $\bar\alpha \colon h g \Rightarrow h'$. Thus, the map $\ps L_T C ((X,h,\nu),
  (X',h',\nu') ) \to \ps L_T D ((X,fh,\nu), (X',fh',\nu')) $ given by
  functoriality of $\ps L_T f$ is a bijection, i.e.\ $\ps L_T f$ is
  fully-faithful.
\end{proof}

\looseness=-1 Let us now describe the action of $\u{\ps L}_T$ on profunctors.
Since we are working with the regular bicategory $\coop\CAT$, we represent a
profunctor $F\colon C \pro D$ as a TSDF therein, i.e., as a cospan $C
\xrightarrow{\cod} R \xleftarrow{\dom} D$ in $\CAT$ (see, e.g.,
\cite{streetFibrationsBicategories1980}). In particular, we can describe the
category $R$ as having
\begin{itemize}
  \item as objects, objects of $D$ and objects of $C$;
  \item as morphisms $x \to x'$,
  \begin{itemize}
    \item $D$-arrows $x \to x'$ in case $x,x' \in D$,
    \item $C$-arrows $x \to x'$ in case $x,x'\in C$,
    \item elements of $F(x,x')$ in case $x\in D$ and $x' \in C$.
  \end{itemize}
  \end{itemize}
The two functors $\dom$ and $\cod$ are then the evident inclusions.

\begin{remark}\label{rem:explicit-description-on-profunctors}
  Identifiying a profunctor $F\colon C \pro D$ with the cospan $C
  \xrightarrow{\cod} R \xleftarrow{\dom} D$ as above, the profunctor $\u{\ps L}_T F \colon \op{\ps L_T D} \times \ps L_T C \to \Set$ is defined as the composite:
  \[\begin{tikzcd}[column sep = large]
      {\ps L_T C} & {\ps L_T R} & {\ps L_T D}
      \arrow["{(\ps L_T \cod)_\diamond}"{inner sep=.8ex}, "\shortmid"{marking}, from=1-1, to=1-2]
      \arrow["{(\ps L_T \dom)^\diamond}"{inner sep=.8ex}, "\shortmid"{marking}, from=1-2, to=1-3]
    \end{tikzcd}\]
  This means that, fixing two triples $(X, h, \nu)$ in $\ps L_T D$ and $(Y, k,
  \xi)$ in $\ps L_T C$, an element of the set $\u{\ps L}_TF ( (X,h,\nu), (Y,k,\xi) )$ is
  given by an equivalence class of:
  \begin{itemize}
  \item an object $(Z, l, \chi)$,
  \item a morphism $(p, \gamma) \colon (X, \dom h, \nu) \to (Z,  l, \chi)$,
  \item and a morphism $(q, \delta) \colon (Z, l, \chi) \to (Y, \cod k , \xi)$
\end{itemize}
  in $\ps L_T R$, where two such triples are equivalent if there exists a zig-zag of morphisms
  in $\ps L_T R$ making the evident triangles commute. Up to equivalence, we can (and will) always represent such a class either with the triple $\braket{ (Y, d_1k, \xi) , (q,\delta)\circ (p,\gamma) , \id }$ or with the triple $\braket{ (X, d_0h , \nu), \id, (q,\delta)\circ (p,\gamma) }$.
\end{remark}

\subsection{Quotient pseudomonads}\label{ssec:quotients}

The $2$-dimensional version of the ultrafilter monad we will work with in
\Cref{sec:05} is not quite ${\ps L}_\beta$, but rather a quotient thereof. In
fact, at least in our four examples, there are obvious quotients of
${\ps L_T}$ that can be considered. The key idea is to identify morphisms
which agree `almost everywhere' with respect to some measure defined by the
monad $T$, according to the intuition that sets `of measure zero' should not
matter to model the desired behavior.
\Cref{thm:loke-extend}, \Cref{lem:discrete-algebras-of-loke} and
\Cref{ex:discrete-algebras-of-loke} can then be easily seen to carry over to these quotients.

\begin{definition}
  Let $T \in \set{ \mathcal F , \beta, \mathcal P, \mathcal D}$. For a category
  $C$, denote by $\bb T C$ the quotient of $\ps L_T C$ obtained by identifying
  two morphisms $(f,\alpha), (f',\alpha') \colon (X,h,\nu) \to (X',h',\nu')$:
\begin{itemize}
\item for $T =\mathcal P$, so that $\nu'$ is a subset of $X'$, when $f(x) =
  f'(x)$ and $\alpha_x = \alpha'_x$ for all $x\in \nu'$;
\item for $T = \mathcal D$, so that $\nu'$ is a distribution on $X'$, when $f(x)
  = f'(x)$ and $\alpha_x = \alpha'_x$ for all $x\in \supp(\nu')$;
\item for $T = \mathcal F$ (resp.\ $T = \beta$), so that $\nu'$ is a filter
  (resp.\ ultrafilter) on $X'$, when there exists a subset $V \in \nu'$ such
  that $f(x) = f'(x)$ and $\alpha_x = \alpha'_x$ for all $x\in V$.
\end{itemize}
This quotient extends to a pseudofunctor $\bb T \colon \CAT \to \CAT$ inheriting the structure of a pseudomonad from that of ${\ps L_T}$.
\end{definition}

\begin{corollary}\label{cor:quotients-extend}
  For $T \in \set{ \mathcal F , \beta,\mathcal P,\mathcal D}$, the pseudomonad
  ${\bb T}$ extends to a left skew monad on $\PROF$.
\end{corollary}

\begin{remark}
  \looseness=-1
  The projection functors $\ps L_T C \twoheadrightarrow \bb T C$ assemble into morphisms of pseudomonads on $\CAT$ (see, e.g., \cite[Def.\ 2.1]{marmolejoCoherencePseudodistributiveLaws2008} and
  \cite[Def.\ 2.1]{gambinoFormalTheoryPseudomonads2021}) and, although the notion doesn't yet appear in the literature, morphisms of skew monads on $\PROF$. We leave a more
  systematic study of such quotients and
  the relations between their algebras to future work.
\end{remark}

\begin{remark}\label{rem:algebras-filter}
    As they can be identified with the $\mathcal F$-algebras, continuous lattices coincide with the small and discrete pseudo-$\bb F$-algebras. Morphisms of pseudo-$\bb F$-algebras between them coincide with
  Scott-continuous maps \emph{which also preserve
    all meets}, since the latter correspond to morphisms of $\mathcal F$-algebras.
    Therefore, strictly speaking, $\PsAlg{\bb F}$  does not categorify the category of continuous lattices proved
  to be cartesian closed in \cite{scottContinuousLattices1972}. We conjecture
  that a cartesian closed $2$-category of pseudo-$\bb F$-algebras (cf.\
  \cite{savilleCartesianClosedBicategories2019}) can be defined by weakening
  their morphisms appropriately, thus giving a bicategorical domain-theoretic
  model of $\lambda$-calculus.

  Similarly, we can consider the \emph{proper} filter monad $\mathcal F_+$ on $\Set$, i.e.\ where $\mathcal F_+(X)$ is the set of
proper filters on a set $X$: the category $\Alg{ \mathcal F_+}$, isomorphic to that of \emph{continuous domains} by \cite{wylerAlgebraicTheoriesContinuous1985}
(see also
\cite{escardoInjectiveSpacesFilter1997}), contains the cartesian closed category of Scott domains.
\end{remark}

\subsection{Pseudodistributive laws}\label{ssec:psdist-laws}

In the $1$-dimensional setting, $\Rel$ is equivalent to the Kleisli category of
the powerset monad: this allows to characterize relational extensions of
$\Set$-monads in terms of distributive laws~\cite{beckDistributiveLaws1969}. In
the $2$-dimensional setting, one can similarly define the \emph{Kleisli
  bicategory} $\Kl{\ps T}$ of a pseudomonad ${\ps T}$ on a bicategory $\K$ as a
bicategory with objects those of $\K$ and arrows $X \klarrow Y$ the $\ps
T$-effectful ones $X \to \ps T Y$~\cite{chengPseudodistributiveLaws2003}. To see
how $\PROF$ can also be identified with a Kleisli bicategory, recall the
\emph{small presheaf} pseudomonad $\psh$.

\begin{definition}\label{ex:small-presheaves}
 For a category $C$, denote by $\psh C$ the full subcategory of
  $[\op C, \Set]$ spanned by small presheaves; this assignment extends to a pseudomonad on $\CAT$ where, for a functor $f \colon C \to D$, the functor $\psh f \colon \psh C \to \psh D$ acts by left Kan extension along $\op f$ (see, e.g., \cite{dayLimitsSmallFunctors2007}). At a category $C$, its unit is given by the Yoneda embedding $C \hookrightarrow \psh C$, while its multiplication $\psh^2 C \to \psh C$ is given by the left Kan extension of the identity functor $1_{\psh C}$ along the Yoneda embedding ${\psh C} \hookrightarrow \psh^2 C$.
\end{definition}

\looseness=-1 The category $\psh C$ is the free cocompletion of $C$ under small
colimits: in particular, $\psh C=[\op C, \Set]$ if $C$ is small. Currying, the datum of a
functor $C \to \psh D$ is equivalent to that of a {small}
profunctor $\op D \times C \to \Set$: this correspondence gives rise to a biequivalence $\Kl\psh
\simeq \PROF$ (see, e.g., \cite{walkerDistributiveLawsAdmissibility2019}). The
pseudomonads on $\PROF$ provided by \Cref{cor:extension-pseudomonads}(2) are therefore extensions, in the sense
of~\cite{chengPseudodistributiveLaws2003}, to $\Kl\psh$. Thus, by \cite[Thm.\
4.3]{chengPseudodistributiveLaws2003}, our \Cref{cor:extension-pseudomonads} also characterizes when a pseudomonad \emph{pseudodistributes} over $\psh$; see Appendix \ref{app:A5} for more details.

\begin{definition}[\cite{marmolejoDistributiveLawsPseudomonads1999}]
    Let $\monad{\ps T}$ and $\monad{\ps S}$ be pseudomonads on a bicategory $\bicat{ K}$. A \emph{pseudodistributive law} of $\ps T$ over $\ps S$ consists of a pseudonatural transformation $\lambda \colon \ps T \circ \ps S \Rightarrow \ps S \circ \ps T$ equipped with four invertible modifications
    \[\begin{tikzcd}[column sep = 15pt, row sep = 12pt]
	{\ps T\ps S} && {\ps S \ps T} \\
	& {\ps S}
	\arrow[""{name=0, anchor=center, inner sep=0}, "\lambda", from=1-1, to=1-3]
	\arrow["{ \eta^{\ps T} \ps S}", from=2-2, to=1-1]
	\arrow["{\ps S \eta^{\ps T}}"', from=2-2, to=1-3]
	\arrow["\cong\,"', between={0.2}{0.8}, Rightarrow, from=0, to=2-2]
\end{tikzcd} \qquad \begin{tikzcd}[row sep = 12pt, column sep = 15pt]
	{\ps T\ps S} && {\ps S \ps T} \\
	& {\ps T}
	\arrow[""{name=0, anchor=center, inner sep=0}, "\lambda", from=1-1, to=1-3]
	\arrow["{\ps T \eta^{\ps S}}", from=2-2, to=1-1]
	\arrow["{\eta^{\ps S}\ps T}"', from=2-2, to=1-3]
	\arrow["\cong\,"', between={0.2}{0.8}, Rightarrow, from=0, to=2-2]
\end{tikzcd}\]
\[ \begin{tikzcd}[column sep = 15pt]
	{\ps T^2 \ps S} & {\ps T \ps S \ps T} & {\ps S \ps T^2} \\
	{\ps T\ps S} && {\ps S \ps T}
	\arrow["{\ps T \lambda }", from=1-1, to=1-2]
	\arrow["{\mu^{\ps T} \ps S}"', from=1-1, to=2-1]
	\arrow["{\lambda \ps T}", from=1-2, to=1-3]
	\arrow["{\ps S \mu^{\ps T}}", from=1-3, to=2-3]
	\arrow[""{name=0, anchor=center, inner sep=0}, "\lambda"', from=2-1, to=2-3]
	\arrow["\cong\,"', between={0.2}{0.8}, Rightarrow, from=1-2, to=0]
\end{tikzcd} \quad \begin{tikzcd}[column sep = 15pt]
	{\ps T \ps S^2} & {\ps S \ps T \ps S} & {\ps S^2 \ps T} \\
	{\ps T \ps S} && {\ps S \ps T}
	\arrow["{\lambda \ps S}", from=1-1, to=1-2]
	\arrow["{\ps T \mu^{\ps S}}"', from=1-1, to=2-1]
	\arrow["{\ps S \lambda}", from=1-2, to=1-3]
	\arrow["{\mu^{\ps S}\ps T}", from=1-3, to=2-3]
	\arrow[""{name=0, anchor=center, inner sep=0}, "\lambda"', from=2-1, to=2-3]
	\arrow["\cong\,"', between={0.2}{0.8}, Rightarrow, from=1-2, to=0]
\end{tikzcd} \]
satisfying eight coherence conditions spelled out in \cite{marmolejoCoherencePseudodistributiveLaws2008}. If such a $\lambda$ exists, we say that $\ps T$ \emph{pseudodistributes over $\ps S$}.
\end{definition}

\begin{corollary}\label{cor:extension-on-cat}
  For a pseudomonad $\braket{\ps T,\eta,\mu}$ on $\CAT$, the following are
  equivalent:
  \begin{enumerate}
  \item ${\ps T}$, $\eta$ and $\mu$ satisfy the Beck-Chevalley condition;
  \item $\braket{\ps T, \eta, \mu}$ extends to a pseudomonad on $\PROF$;
  \item $\braket{\ps T, \eta, \mu}$ pseudodistributes over $\psh$.
\end{enumerate}
  \begin{proof}
    Combine \Cref{cor:extension-pseudomonads} with \cite[Thm.\ 4.3]{chengPseudodistributiveLaws2003}.
  \end{proof}
\end{corollary}

In particular, we reobtain \cite[Cor.\ 49]{walkerDistributiveLawsAdmissibility2019} in the
archetypal case of the small presheaf pseudomonad:

\begin{corollary}
    Up to isomorphism, there is at most one pseudodistributive law of a pseudomonad on $\CAT$ over $\psh$.
\end{corollary}

 \section{Ultraconvergence spaces, algebraically}\label{sec:05}

\looseness=-1
In this section, we focus on the pseudomonad $\bbbeta$ constructed in \Cref{ssec:quotients}. First, we discuss \emph{ultracategories}, introducing them as its pseudoalgebras. Then, we apply our extension theorem so as to recover \emph{ultraconvergence spaces}
\cite{goolToposesEnoughPoints2025} as the \emph{profunctorial $\bbbeta$-algebras} --- that is, suitable algebras for the skew extension of $\bbbeta$ to $\PROF$.

\subsection{Ultracategories}

\looseness=-1
The pseudomonad $\bbbeta$ is known as the \emph{ultracompletion} pseudomonad, and it was first considered in \cite{rosoliniUltracompletions2024} (based on the work of \cite{garnerUltrafiltersFiniteCoproducts2020}) to study \emph{ultracategories}. Originally introduced by Makkai \cite{makkaiStoneDualityFirst1987}, ultracategories can be thought of as categories endowed with structure allowing to compute abstract \emph{ultraproducts} of tuples of objects. In \cite{lurieUltracategories2018}, Lurie gave a new axiomatization for the concept, simplifying Makkai's definition and extending his results. While \cite{marmolejoUltraproductsContinuousFamilies1995} constituted a first attempt towards an algebraic theory of ultracategories, in the recent \cite{rosoliniUltracompletions2024} they were re-defined, in even simpler terms, as pseudo-$\bbbeta$-algebras (cf.\ also \cite[\S 3]{saadiaExtendingConceptualCompleteness2025}). Intuitively, we can think of $\bbbeta C$ as the category of \emph{formal ultraproducts} of objects of $C$, which we also refer to as \emph{ultrafamilies}: an ultracategory is thus a category $C$ in which these formal ultraproducts can be evaluated by means of a pseudoalgebra functor $\bbbeta C \to C$.

\begin{definition}\label{def:ultracategory}
    An \emph{ultracategory} is a pseudo-$\bbbeta$-algebra.    We denote by $\UltCat$ the $2$-category $\PsAlgco{\bbbeta}$ of pseudo-$\bbbeta$-algebras, colax morphisms, and algebra $2$-cells.
\end{definition}

\begin{example}\label{ex:ultracat-of-points}
    The archetypal example of an ultracategory is given by the category of models of a \emph{coherent} theory $\mathbb T$ --- that is, a first-order theory whose axioms are of the form $\forall \vec{x} (\phi(\vec{x}) \Rightarrow \psi(\vec{x}))$ where $\phi(\vec{x})$ and $\psi(\vec{x})$ are built using only $\top$, $\bot$, $\land$, $\lor$, and $\exists$. Note that every first-order theory is equivalent to a coherent theory (see, e.g., \cite[\S D1.5.13]{johnstoneSketchesElephantTopos2002}).

    The algebra functor $\bbbeta (\Mod (\bb T)) \to \Mod (\bb T)$ maps an
    ultrafamily $(X, M, \nu)$ of models to their \emph{ultraproduct}
    $\prod_{x:\nu} M_x$, which is a model of $\bb T$ by Łoś's theorem (see,
    e.g., \cite[\S 4]{changModelTheory2012}). The model $\prod_{x:\nu} M_x$ is
    characterized by the fact that it satisfies precisely those properties
    shared by some subset of models $\set{M_x | x \in X_0}$ with $X_0 \in \nu$: borrowing intuition from the topological case, we can think of it as the \emph{limit} of the ultrafamily $(X, M, \nu)$.

\end{example}

\begin{remark}
  As noted in \Cref{ex:discrete-algebras-of-loke}, by Manes' theorem
  \cite[Prop.\ 5.5]{manesTripleTheoreticConstruction1969}, discrete
  ultracategories coincide with $\beta$-algebras. This fact, originally proved
  by Lurie in \cite[Thm.\ 3.1.5]{lurieUltracategories2018} for his
  axiomatization of ultracategories, here entails that the pseudomonad $\bbbeta$ is neither \emph{lax-idempotent}
  nor \emph{colax-idempotent} \cite{kellyPropertylikeStructures1997}.
  Intuitively, this means that the datum of an ultracategory consists of
  additional \emph{structure} imposed on a category, rather than a property
  thereof.
\end{remark}

The main result of \cite{makkaiStoneDualityFirst1987,lurieUltracategories2018} expresses how a coherent theory can be recovered from its category of models, once the latter is endowed with its canonical ultracategory structure. Exploiting the biequivalence established in \cite{hamadUltracategoriesColaxAlgebras2025}, we can state it as follows.

\begin{theorem}[\protect{\cite[Thm.\ 2.2.2]{lurieUltracategories2018}}] \label{thm:reconstruction-coherent}
Let $\bb T$ be a coherent theory. Then, $\UltCat( \Mod (\bb T), \Set)$ is the classifying topos\footnote{Recall that classifying toposes play the role, for first-order logic, of the propositional syntactic algebras.} of $\bb T$.
\end{theorem}

\subsection{Ultraconvergence spaces as algebras}\label{ssec:ultraconv-spaces}

\looseness=-1
In the recent works \cite{saadiaExtendingConceptualCompleteness2025,hamadGeneralisedUltracategoriesConceptual2025,goolToposesEnoughPoints2025},
ultracategories were generalized by drawing inspiration from Barr's generalization of Manes' theorem: that is, by relaxing the `limit map'
$\bbbeta C \to C$ defining an ultracategory to a `convergence relation', in the
shape of a profunctor between $C$ and ultrafamilies in $C$. Concretely, all three generalizations arise in the attempt to categorify Barr's description of
topological spaces as \emph{relational $\beta$-algebras} --- that is, lax algebras for a right skew
monad $\u\beta\colon \Rel \to \Rel$ extending the ultrafilter monad. Although they all succeed in conveying a
reconstruction theorem for geometric logic, the precise axiomatizations of such generalizations
are somewhat arbitrary: in fact, while those of
\cite{saadiaExtendingConceptualCompleteness2025, goolToposesEnoughPoints2025} coincide, that of
\cite{hamadGeneralisedUltracategoriesConceptual2025} is weaker by one axiom.

\looseness=-1
Here, we can take a different approach: we can let a profunctorial version of ultracategories emerge naturally via our extension theorem,  thus promoting the inspiration behind \cite{saadiaExtendingConceptualCompleteness2025,hamadGeneralisedUltracategoriesConceptual2025,goolToposesEnoughPoints2025} to an instance of a $2$-dimensional extension result \emph{à la} Barr. It turns out that the resulting notion coincides with the \emph{ultraconvergence spaces} (or \emph{virtual ultracategories}) of \cite{saadiaExtendingConceptualCompleteness2025,goolToposesEnoughPoints2025}, up to restricting to \emph{small} profunctors. This provides a categorical justification for the notion introduced in \emph{loc.\ cit.}, and paves the way to an algebraic study of ultraconvergence spaces.

\looseness=-1 We begin by recalling the definition given in
\cite{saadiaExtendingConceptualCompleteness2025, goolToposesEnoughPoints2025} so
as to fix notations. Below, we denote by $1$ both the singleton set $\{*\}$ and the unique
ultrafilter $\eta^{\beta}_1(*)$ on it. Moreover, for an ultrafilter $\nu$ on a set $X$ and an $X$-indexed family
of ultrafilters $(\xi_x \in \beta Y_x)_{x \in X}$, we denote by
$\sum_{x:\nu} \xi_x$ their \emph{direct sum}, i.e., the ultrafilter on the disjoint union
$\sum_{x\in X}Y_x$ defined by those subsets $S$ such that $S \cap Y_x \in \xi_x$ for all $x \in X_0$ for some $X_0 \in \nu$. This construction is functorial in the following sense, which we describe for $\bbbeta 1$ for the sake of readability although it is clear how it extends to an arbitrary $\bbbeta C$.
\begin{itemize}
	\item Let $g \colon (X, \nu) \to (Y, \xi)$ be a $\bbbeta 1$-arrow and let $((Z_x, \chi_x))_{x:\nu}$ be an ultrafamily in $\bbbeta 1$. Then, the map $(y,z) \mapsto (g(y), z)$ determines an arrow
\end{itemize}
\[ \textstyle \left(\sum_{x\in X} Z_x, \sum_{x:\nu}\chi_x\right)\to \left(\sum_{y\in Y}Z_{g(y)}, \sum_{y:\xi}\chi_{g(y)}\right)\]
\begin{itemize}
	\item[]in $\bbbeta 1$, which we denote by $g\otimes \id$.

	\item Let $(g_x \colon (Y_x, \xi_x) \to (Z_x,\chi_x))_{x:\nu}$ be an ultrafamily of $\bbbeta 1$-arrows. Then, the map $(x,z)\mapsto (x, g_x(z))$ determines an arrow
\end{itemize}
\[ \textstyle \left(\sum_{x\in X} Y_x, \sum_{x:\nu}\xi_x\right)\to \left(\sum_{x\in X}Z_x, \sum_{x:\nu}\chi_{x}\right)\]
\begin{itemize}
	\item[]in $\bbbeta 1$, which we denote by $\id \otimes (g_x)_{x:\nu}$.
\end{itemize}

To provide intuition behind the following definition, recall Barr's characterization of topological spaces as convergence relations: a
topology on a set $X$ is the datum of a relation $\beta X \pro X$, between
ultrafilters on $X$ and their limits, satisfying certain axioms --- for
instance, the \emph{principal} ultrafilter $\eta^\beta_X(x)$ at $x \in X$ should
converge to $x$. An ultraconvergence space is a $2$-dimensional generalization
of this description.

\begin{definition}\label{def:uc-space}
  An \emph{ultraconvergence space} $A$ is a discrete category $A_0$ equipped
  with an \emph{ultraconvergence structure}, i.e., the datum of:
\begin{itemize}
\item a profunctor $A \colon \bbbeta A_0 \pro A_0$, where we refer to elements
  $r \in A(a, (X, b, \nu))$ as \emph{ultraconvergence arrows}, denoted as $r \colon a \ult
  (b_x)_{x:\nu}$, and where we denote by $r[g]$ the ultraconvergence arrow $A(\id, g)(r)$ for $g \colon (X, b, \nu) \to (Y,c,\xi)$ in $\bbbeta A_0$;
\item for each object $a \in A_0$, an \emph{identity} ultraconvergence arrow $\id_a \colon a \ult (a)_{*:1}$;
\item for each ultraconvergence arrow $r \colon a \ult (b_x)_{x:\nu}$ and each
  ultrafamily $(X, s,\nu)$ of ultraconvergence arrow $s_x \colon b_x \ult
  (c_{x,y})_{y:\xi_x}$, a \emph{composite} ultraconvergence arrow
  $(s_x)_{x:\nu} \cdot r \colon a \ult (c_{x,y})_{(x,y): \sum_{x:\nu}\xi_x}$,
\end{itemize}
satisfying:
\begin{enumerate}
	\item \emph{left naturality}, $(s_{g(y)})_{y:\xi} \cdot r[g] = ( (s_x)_{x:\nu} \cdot r)[g\otimes \id]$;
	\item \emph{right naturality}, $(s_x[g_x])_{x:\nu} \cdot r = ((s_x)_{x:\nu} \cdot r)[\id \otimes (g_x)_{x:\nu}]$;
	\item \emph{left unitality}, $(r)_{*:1} \cdot \id_a = r$;
	\item\emph{right unitality}, $(\id_{b_x})_{x:\nu} \cdot r = r$;
	\item \emph{associativity}, $(t_{x,y})_{(x,y):\sum_{x:\nu}\xi_x} \cdot ( ( s_x)_{x:\nu}\cdot r) = ((t_{x,y})_{y:\xi_x} \cdot s_x )_{x:\nu}\cdot r$,
\end{enumerate}
for any $r$, $(s_x)_{x:\nu}$ and $g$ as above, any ultrafamily of ultraconvergence arrows $(t_{x,y})_{(x,y):\sum_{x:\nu}\xi_x}$ of the appropriate type, and any ultrafamily of $\bbbeta A_0$-arrows $(g_x \colon (Y_x,c_x, \xi_x) \to (Z_x,d_x,\chi_x))_{x:\nu}$.  Moreover, in this paper we
will assume that the profunctor $A$ is \emph{small}.
\end{definition}

\looseness=-1 Ultraconvergence arrows $a \ult (b_x)_{x:\nu}$ can be thought of
as witnessing convergence of the ultrafamily $(X, b, \nu)$ in $A_0$ to the point
$a$ in $A_0$. As there may be several of these witnesses, ultraconvergence
spaces are a \emph{proof-relevant} generalization of topological spaces, in
which convergence is two-valued.

\begin{example}[\protect{\cite[Rem.\ 3.8]{goolToposesEnoughPoints2025}}] \label{ex:top-spaces-as-ultspaces}
    Topological spaces coincide with those ultraconvergence spaces $A \colon \bbbeta A_0 \pro A_0$ such that:
    \begin{enumerate}
        \item $A_0$ is small, i.e., a set (carrying the topology);
        \item there is at most one ultraconvergence arrow of each type;
        \item for any morphism $g \colon (X, b, \nu) \to (Y, c, \xi)$ in $\bbbeta A_0$, if there exists an ultraconvergence arrow $a \ult (c_y)_{y:\xi}$ then there exists one $a \ult (b_{x})_{x:\nu}$.
    \end{enumerate}
\end{example}

\begin{example}[\protect{\cite[Ex.\ 3.10]{goolToposesEnoughPoints2025}}] \label{ex:ultrasp-of-points}
The archetypal example of an ultraconvergence space is given by the (discrete) category of models of a geometric theory $\bb T$ --- that is, a theory in \emph{infinitary} first-order logic whose axioms are of the form $\forall \vec x (\phi(\vec x) \Rightarrow \psi(\vec x))$ where $\phi(\vec x)$ and $\psi(\vec x)$ are built using only finitary $\land$, infinitary $\bigvee$, and $\exists$.

\looseness=-1
Unlike in \Cref{ex:ultracat-of-points}, the ultraproduct $\prod_{x:\nu} N_x$ of an ultrafamily $(X, N, \nu)$ of models of $\bb T$ may not be itself a model of $\bb T$. However, we can think of $\prod_{x:\nu} N_x$ as a possibly-undefined limit of the ultrafamily $(X, N, \nu)$, and define an ultraconvergence structure on $\Mod (\bb T)$ by setting ultraconvergence arrows $M \ult (N_x)_{x:\nu}$ to be \emph{$\Sigma$-structure homomorphisms} $M \to \prod_{x:\nu} N_x$ (\cite[Def.\ D1.2.1]{johnstoneSketchesElephantTopos2002}) with $\Sigma$ the signature of $\bb T$.
\end{example}

\begin{wrapfigure}{r}{0.24\linewidth}
      \vspace{-8pt}\hspace{3pt}
  \begin{tikzcd}[column sep = 11pt, row sep = 13pt,cramped]
	{A_0} && {A_0} \\
	& {\bbbeta A_0}
	\arrow[""{name=0, anchor=center, inner sep=0}, equals, from=1-1, to=1-3]
	\arrow["{\eta_{A_0}^{\bbbeta}}"'{inner sep=.8ex}, from=1-1, to=2-2]
	\arrow["A"'{inner sep=.8ex}, "\shortmid"{marking}, from=2-2, to=1-3]
	\arrow["\,\Gamma", between={0.3}{0.8}, Rightarrow, from=0, to=2-2]
  \end{tikzcd}

  \vspace{6pt}
  \hspace{-9pt}
  \begin{tikzcd}[row sep = 22pt,column sep = 30pt,cramped]
	{\bbbeta^2 A_0} & {\bbbeta A_0} \\
	{\bbbeta A_0} & {A_0}
	\arrow["{\u\bbbeta A}"{inner sep=.8ex}, "\shortmid"{marking}, from=1-1, to=1-2]
	\arrow["{\mu_{A_0}^{\bbbeta}}"'{inner sep=.8ex}, from=1-1, to=2-1]
	\arrow["\Delta"', between={0.3}{0.7}, Rightarrow, from=1-2, to=2-1]
	\arrow["A"{inner sep=.8ex}, "\shortmid"{marking}, from=1-2, to=2-2]
	\arrow["A"'{inner sep=.8ex}, "\shortmid"{marking}, from=2-1, to=2-2]
  \end{tikzcd}
  \vspace{-10pt}
\end{wrapfigure}
Recall that, by  \Cref{cor:quotients-extend}, $\bbbeta$ extends to a left skew monad $\u\bbbeta$ on $\PROF$. For an ultraconvergence space $A \colon \bbbeta A_0 \pro A_0$, it is immediate that the datum of its identity ultraconvergence arrows corresponds to that of a natural transformation $\Gamma$ as
in the top diagram on the right.\footnote{For diagrams in $\PROF$, we identify a functor $f \colon C \to D$ with the representable profunctor $f_\diamond \colon C \pro D$ when possible, denoting it with an undecorated arrow.} With some work, one can also see that the datum
of its composite ultraconvergence arrows is interdefinable with that of a
transformation $\Delta$ as in the bottom diagram on the right. Thus, the datum of an ultraconvergence structure on a discrete category $A_0$ coincides with what should be that of a \emph{lax $\u\bbbeta$-algebra}. However, it is not immediately clear how such an algebra \emph{should be defined} considering that $\u{\bbbeta}$ is a \emph{left} skew monad (cf.\ the lax algebras for right skew monads defined in \cite{bungeCoherentExtensionsRelational1974,streetSkewMonads2015}). More precisely, it is not evident how to write appropriate coherence axioms, at least such that ``free algebras are algebras'' --- i.e., such that, for each category $C$, the profunctor $(\mu^{\bbbeta}_{C})_\diamond \colon \bbbeta^2 C \pro \bbbeta C$ endows the category $\bbbeta C$ with the structure of a lax algebra, whose unitor and associator are determined by the structural modifications of the monad. In the following definition, we thus enucleate sufficient conditions ensuring as much: in essence, these conditions make a left skew monad ``close enough'' to a pseudomonad that we can invert enough structure in order to define lax algebras.

\begin{definition}\label{def:allowing-lax-algebras}
  Let $\monad{\u{\ps T}}$ be a left skew monad on $\PROF$ extending a pseudomonad $\monad {\ps T}$ on $\CAT$ via \Cref{cor:extension-pseudomonads}.\eqref{cor:extension-pseudomonads:right-comonad}. We say that $\u{\ps T}$ \emph{allows for lax algebras} if:
  \begin{enumerate}
    \item the transformations $\delta^{\ps T} \colon \u{\ps T}(-)_\diamond \Rightarrow (\ps T -)_\diamond$ and $\omega^{\ps T} \colon \u{ \ps T}^2 \Rightarrow \u{(\ps T^2)}$ are pseudonatural;
\item there exists a natural family of natural transformations $\phi_{G,f_\diamond} \colon \u{\ps T}(G \circ f_\diamond) \Rightarrow \u{\ps T} G \circ \u{\ps T} f_\diamond$, for any functor $f \colon A \to B$ and any profunctor $G \colon B \pro C$, such that
    \begin{enumerate}
      \item the diagram
      \[\begin{tikzcd}
	{\u{\ps T}(G \circ (f'f)_\diamond)} & {\u{\ps T}G \circ \u{\ps T}(f' f)_\diamond} & {\u{\ps T}G \circ \u{\ps T}(f'_\diamond \circ  f_\diamond)} \\
	{\u{\ps T}(G \circ f'_\diamond\circ f_\diamond)} & {\u{\ps T}(G \circ f'_\diamond) \circ \u{\ps T}f_\diamond} & {\u{\ps T}G \circ \u {\ps T}  f'_\diamond \circ \u{\ps T}f_\diamond}
	\arrow["{{\phi_{G, (f'f)_\diamond}}}", from=1-1, to=1-2]
	\arrow["\cong"', from=1-1, to=2-1]
	\arrow["\cong", from=1-2, to=1-3]
	\arrow["{{\u{\ps T}G* \phi_{f'_\diamond, f_\diamond}}}", from=1-3, to=2-3]
	\arrow["{{\phi_{G\circ f'_\diamond, f_\diamond}}}"', from=2-1, to=2-2]
	\arrow["{{\phi_{G,f'_\diamond}*\u{\ps T }f_\diamond}}"', from=2-2, to=2-3]
    \end{tikzcd}\]
      commutes, and
      \item $\phi_{f'_\diamond,f_\diamond} = \psi_{f'_\diamond,f_\diamond}^{-1}$, where $\psi_{f'_\diamond,f_\diamond} \colon \u{\ps T} f'_\diamond \circ \u{\ps T} f_\diamond\Rightarrow \u{\ps T}(f'_\diamond \circ f_\diamond)$ is the lax associator for $\u{\ps T}$
    \end{enumerate}
    for each pair of functors $f\colon A \to B, f'\colon B \to C$ and each profunctor $G \colon C \pro D$.

\end{enumerate}
\end{definition}

Intuitively, the third condition corresponds to a semi-\emph{oplax} associator for $\u{\ps T}$, restricted to the case where the first arrow is representable, and acting as an inverse to its lax associator when both arrows are representable (cf.\ \cite[Def.\ 4.2]{bungeCoherentExtensionsRelational1974}). Note also that, by axiom (1), each lax unitor $\psi_A \colon 1_{\ps TA}\Rightarrow \u{\ps T}1_A$ for $\u{\ps T}$ is invertible, since the diagram of natural transformations
\[\begin{tikzcd}
{(1_{\ps TA})_\diamond} & {(\ps T1_A)_\diamond} \\
{1_{\ps T A}} & {\u{\ps T}(1_A)_\diamond}
\arrow["\cong", from=1-1, to=1-2]
\arrow["\cong"', from=1-1, to=2-1]
\arrow["{\delta^{\ps T}_{1_A}}", from=1-2, to=2-2]
\arrow["{\psi_A}"', from=2-1, to=2-2]
\end{tikzcd}\]
commutes by lax unitality of $\delta^{\ps T}$.

\begin{remark}\label{rem:lax-algebras-allowed-identifying-delta}
  \looseness=-1
  In essence, this means that $\u{\ps T}$ restricts to a pseudofunctor on representable profunctors by mapping a profunctor $f_\diamond$ to $(\ps T f)_\diamond$ and a natural transformation $\alpha \colon f_\diamond \Rightarrow g_\diamond$ to $(\delta_g^{\ps T})^{-1} \circ \u{\ps T}\alpha \circ \delta^{\ps T}_f \colon (\ps T f)_\diamond\Rightarrow(\ps T g)_\diamond$. Therefore, in the following, in order to simplify notations we will identify the profunctor $\u {\ps T} (f_\diamond)$ with the representable $(\ps T f)_\diamond$, thus omitting $\delta^{\ps T}$.

  Similarly, we will identify the components of $\eta^{\u{\ps T}}$ and $\mu^{\u{\ps T}}$ at a representable $f_\diamond$ with $(\eta^{\ps T}_f)_\diamond$ and $(\mu^{\ps T}_f)_\diamond$ respectively, thus omitting the components of the canonical transformations $\omega^1 \colon 1_{\PROF}\Rightarrow \u{1_{\CAT}}$ and $\omega^{\ps T} \colon \u{\ps T}^2\Rightarrow \u{(\ps T^2)}$ of \Cref{cor:extension-pseudomonads}. Note also that, by the a-priori assumption that $\eta^{\ps T}$ and $\mu^{\ps T}$ satisfy the Beck-Chevalley condition, axiom (1) entails that $\eta^{\u{\ps T}}$ and $\mu^{\u{\ps T}}$ are pseudonatural.

\end{remark}

For such a skew monad, we can indeed define lax algebras in a sensible way; we refer the reader to Appendix \ref{app:A5} for more details.

\begin{definition}\label{def:lax-algebras-allowed}
  Let $\monad{\u{\ps T}}$ be a left skew monad on $\PROF$ extending a
  pseudomonad $\monad{\ps T}$ on $\CAT$ via
  \Cref{cor:extension-pseudomonads}\eqref{cor:extension-pseudomonads:right-comonad}
  and suppose that $\u{\ps T}$ allows for lax algebras. Then, a \emph{lax
    $\u{\ps T}$-algebra} consists of:
  \begin{itemize}
    \item a category $A_0$,
    \item a (small) profunctor $A \colon {\ps T} A_0 \pro A_0$,
    \item a natural transformation $\Gamma \colon 1_{A_0} \Rightarrow A \circ (\eta^{\ps T }_{A_0})_\diamond$ called \emph{unitor},
    \item a natural transformation $\Delta \colon A \circ \u{\ps T} A \Rightarrow A \circ (\mu^{\ps T}_{A_0})_\diamond$ called \emph{multiplicator},
  \end{itemize}
satisfying the following coherence conditions:
\begin{enumerate}
  \item
  \[\begin{tikzcd}
	& {\ps TA_0} & {\ps T^2A_0} \\
	{A_0} & {\ps T A_0} & {\ps T A_0} \\
	& {A_0}
	\arrow["{\eta^{\ps T}_{\ps TA_0}}", from=1-2, to=1-3]
	\arrow[""{name=0, anchor=center, inner sep=0}, "A"'{inner sep=.8ex}, "\shortmid"{marking}, from=1-2, to=2-1]
	\arrow[""{name=1, anchor=center, inner sep=0}, "{{\u{\ps T}A}}"{inner sep=.8ex}, "\shortmid"{marking}, from=1-3, to=2-2]
	\arrow["{\mu^{\ps T}_{A_0}}", from=1-3, to=2-3]
	\arrow["{\eta^{\ps T}_{A_0}}", from=2-1, to=2-2]
	\arrow[""{name=2, anchor=center, inner sep=0}, equals, from=2-1, to=3-2]
	\arrow["A"{inner sep=.8ex}, "\shortmid"{marking}, from=2-2, to=3-2]
	\arrow[""{name=3, anchor=center, inner sep=0}, "A"{inner sep=.8ex}, "\shortmid"{marking}, from=2-3, to=3-2]
	\arrow["{{(\eta_A^{\u{\ps T}})^{-1}}}", between={0.3}{0.7},pos=0.6, Rightarrow, from=0, to=1]
	\arrow["{{{\Delta}}}", between={0.3}{0.7}, Rightarrow, from=1, to=3]
	\arrow["{{{\Gamma}}}", between={0.3}{0.9}, Rightarrow, from=2, to=2-2]
\end{tikzcd}=
\begin{tikzcd}
	{\ps T A_0} && {\ps T^2A_0} \\
	& {\ps T A_0} \\
	& {A_0}
	\arrow["{\eta^{\ps T}_{\ps T A_0}}", from=1-1, to=1-3]
	\arrow[""{name=0, anchor=center, inner sep=0}, equals, from=1-1, to=2-2]
	\arrow[""{name=1, anchor=center, inner sep=0}, "{\mu^{\ps T}_{A_0}}", from=1-3, to=2-2]
	\arrow["A"{inner sep=.8ex}, "\shortmid"{marking}, from=2-2, to=3-2]
	\arrow["{\mathfrak{l}_{A_0}^{-1}}", between={0.3}{0.7}, Rightarrow, from=0, to=1]
\end{tikzcd}\]

\item
\[\begin{tikzcd}[column sep = 30 pt, row sep = 26pt]
	{\ps T A_0} && {\ps T A_0} && \\
	{\ps T A_0} && {\ps T^2 A_0} && {\ps T A_0} \\
	&& {A_0}
	\arrow[equals, from=1-1, to=1-3]
	\arrow[""{name=0, anchor=center, inner sep=0}, "{\u{\ps T} 1_{A_0}}"'{inner sep =.8ex},"\shortmid"{marking}, from=1-1, to=2-1]
	\arrow[""{name=1, anchor=center, inner sep=0}, "{\u{\ps T} (A \circ (\eta^{\ps T}_{A_0})_\diamond ) }"{description}, curve={height=6pt}, from=1-3, to=2-1]
	\arrow["{\u{\ps T}( \eta^{\ps T}_{A_0})_\diamond }"{inner sep =.8ex},"\shortmid"{marking}, from=1-3, to=2-3]
	\arrow[""{name=2, anchor=center, inner sep=0}, curve={height=-18pt}, equals, from=1-3, to=2-5]
	\arrow[""{name=3, anchor=center, inner sep=0}, "A"'{inner sep=.8ex}, "\shortmid"{marking}, curve={height=18pt}, from=2-1, to=3-3]
	\arrow["{\u{\ps T} A}"{inner sep=.8ex}, "\shortmid"{marking}, from=2-3, to=2-1]
	\arrow["{\mu^{\ps T }_{A_0}}"', from=2-3, to=2-5]
	\arrow[""{name=4, anchor=center, inner sep=0}, "A"{inner sep=.8ex}, "\shortmid"{marking}, curve={height=-18pt}, from=2-5, to=3-3]
	\arrow["{\u{\ps T}\Gamma}"{pos=0.4}, between={0.1}{0.4}, Rightarrow, from=0, to=1]
	\arrow["{\phi_{A, (\eta^{\ps T}_{A_0})_\diamond}}"'{pos=0.6}, between={0.5}{0.9}, Rightarrow, from=1, to=2-3]
	\arrow["{{{\Delta}}}"', shift left=5, between={0.3}{0.7}, Rightarrow, from=3, to=4]
	\arrow["{{{\mathfrak{r}_{A_0}^{-1}}}}"', between={0.3}{0.7}, Rightarrow, from=2-3, to=2, shift right=2]
\end{tikzcd}= \begin{tikzcd}[column sep = 30 pt, row sep = 26pt]
	{\ps T A_0} && {\ps T A_0} \\
	& {\ps T A_0} \\
	& {A_0}
	\arrow[equals, from=1-1, to=1-3]
	\arrow[""{name=0, anchor=center, inner sep=0}, "{\u{\ps T}1_{A_0}}"'{inner sep =.8ex},"\shortmid"{marking}, from=1-1, to=2-2]
	\arrow[""{name=1, anchor=center, inner sep=0}, equals, from=1-3, to=2-2]
	\arrow["A"'{inner sep =.8ex},"\shortmid"{marking}, from=2-2, to=3-2]
	\arrow["{{\psi_A^{-1}}}", between={0.3}{0.7}, Rightarrow, from=0, to=1]
\end{tikzcd}\]
\item
\[\begin{tikzcd}[column sep = 15pt]
	{\ps T^3 A_0} && {\ps T^2 A_0} \\
	{\ps T^2A_0} & {\ps T A_0} & {\ps T^2 A_0} \\
	{\ps T A_0} & {A_0} & {\ps T A_0}
	\arrow["{\mu^{\ps T}_{\ps T A_0}}", from=1-1, to=1-3]
	\arrow[""{name=0, anchor=center, inner sep=0}, "{\u{\ps T}^2 A }"'{inner sep=.8ex}, "\shortmid"{marking}, from=1-1, to=2-1]
	\arrow[""{name=1, anchor=center, inner sep=0}, equals, from=1-3, to=2-3]
	\arrow["{\mu^{{\ps T}}_{A_0}}", from=2-1, to=2-2]
	\arrow[""{name=2, anchor=center, inner sep=0}, "{\u{\ps T}A}"'{inner sep=.8ex}, "\shortmid"{marking}, from=2-1, to=3-1]
	\arrow[""{name=3, anchor=center, inner sep=0}, "A"'{inner sep=.8ex}, "\shortmid"{marking}, from=2-2, to=3-2]
	\arrow["{\u{\ps T}A}"'{inner sep=.8ex}, "\shortmid"{marking}, from=2-3, to=2-2]
	\arrow[""{name=4, anchor=center, inner sep=0}, "{\mu^{\ps T}_{A_0}}", from=2-3, to=3-3]
	\arrow["A"'{inner sep=.8ex}, "\shortmid"{marking}, from=3-1, to=3-2]
	\arrow["A"{inner sep=.8ex}, "\shortmid"{marking}, from=3-3, to=3-2]
	\arrow["{(\mu^{\u{\ps T}}_A)^{-1}}", between={0.3}{0.7}, Rightarrow, from=0, to=1]
	\arrow["\Delta", between={0.3}{0.7}, Rightarrow, from=2, to=3]
	\arrow["\Delta", between={0.3}{0.7}, Rightarrow, from=3, to=4]
\end{tikzcd} = \begin{tikzcd}[column sep = 20pt]
	{\ps T^3 A_0} &[-10pt] {\ps T^3 A_0} & {\ps T^3 A_0} && {\ps T^2 A_0} \\
	{\ps T^2 A_0} &&& {\ps T^2 A_0} \\
	{\ps T A_0} & {\ps T A_0} & {\ps T A_0} & {A_0} & {\ps T A_0}
	\arrow[equals, from=1-1, to=1-2]
	\arrow["{\u{\ps T}^2 A}"'{inner sep=.8ex}, "\shortmid"{marking}, from=1-1, to=2-1]
	\arrow[equals, from=1-2, to=1-3]
	\arrow[""{name=0, anchor=center, inner sep=0}, "{{\u{\ps T} (A\circ \u{\ps T} A)}}"{description, pos=0.2}, "\shortmid"{marking}, from=1-2, to=3-2]
	\arrow["{\mu^{\ps T}_{\ps T A_0}}", from=1-3, to=1-5]
	\arrow["{\u{\ps T}(\mu^{\ps T}_{A_0})_\diamond}"{inner sep=.8ex}, "\shortmid"{marking}, from=1-3, to=2-4]
	\arrow[""{name=1, anchor=center, inner sep=0}, "{{\u{\ps T} (A\circ (\mu^{\ps T}_{A_0})_\diamond )}}"{inner sep=.8ex}, "\shortmid"{marking}, curve={height=-6pt}, from=1-3, to=3-2]
	\arrow[""{name=2, anchor=center, inner sep=0}, "{\mu_{A_0}^{\ps T}}", from=1-5, to=3-5]
	\arrow["{\u{\ps T}A}"'{inner sep=.8ex}, "\shortmid"{marking}, from=2-1, to=3-1]
	\arrow[""{name=3, anchor=center, inner sep=0}, "{\u{\ps T}A}"{inner sep=.8ex}, "\shortmid"{marking}, from=2-4, to=3-3]
	\arrow[""{name=4, anchor=center, inner sep=0}, "{\mu_{A_0}^{\ps T}}", from=2-4, to=3-5]
	\arrow[equals, from=3-1, to=3-2]
	\arrow[equals, from=3-2, to=3-3]
	\arrow["A"'{inner sep=.8ex}, "\shortmid"{marking}, from=3-3, to=3-4]
	\arrow["A"{inner sep=.8ex}, "\shortmid"{marking}, from=3-5, to=3-4]
	\arrow["{{\u{\ps T} \Delta }}"'{pos=0.7}, shift left=3, between={0.3}{0.7}, Rightarrow, from=0, to=1]
	\arrow["{\phi_{A, (\mu^{\ps T}_{A_0})_\diamond}}", shift left=2, between={0.4}{0.7}, Rightarrow, from=1, to=2-4]
	\arrow["{\psi_{A, \u{\ps T}A}}", between={0.3}{0.7}, Rightarrow, from=2-1, to=0]
	\arrow["{{\mathfrak m}^{-1}_{A_0}}", shift left=2, between={0.2}{0.8}, Rightarrow, from=2-4, to=2]
	\arrow["\Delta"', between={0.3}{0.7}, Rightarrow, from=3, to=4]
\end{tikzcd}\]
\end{enumerate}
\end{definition}
\begin{remark}\label{rem:lax-algebras-are-okay}
  If a pseudomonad $\monad{\ps T}$ satisfies the stronger assumptions of
  \Cref{cor:extension-pseudomonads}\eqref{cor:extension-pseudomonads:pseudo-comonad},
  thus extending to a pseudomonad $\monad{\u{\ps T}}$, then the latter clearly
  allows for lax algebras. In that case, the previous definition is equivalent
  to the usual definition of a lax algebra for a pseudomonad on $\PROF$ (see,
  e.g., \cite[\S 2.4]{stepan-thesis}). More generally, the coherence conditions
  above are formally equivalent to the usual coherence conditions defining lax
  algebras (and, in particular, pseudoalgebras) for a pseudomonad on an
  arbitrary bicategory.
\end{remark}

\begin{proposition}\label{prop:lokes-allow-for-lax-algebras}
  For any monad $T$ on $\Set$, the left skew monad $\u{\ps L}_T$ allows for lax algebras.

  Moreover, for $T \in \set{ \mathcal F, \beta, \mathcal P, \mathcal D}$, the same holds for $\u{\bb T}$.
\end{proposition}

The axioms of an
ultraconvergence space can then be seen to correspond to those making a tuple
$\braket{A_0,A,\Gamma,\Delta}$ as above a {lax $\u\bbbeta$-algebra}, so that we
obtain the following.

\begin{theorem}\label{thm:ucspaces-algebraically}
    Let $A_0$ be a discrete category and let $A \colon \bbbeta A_0 \pro A_0$ be a small profunctor. Then, there is a bijection between:
    \begin{enumerate}
        \item ultraconvergence structures based on $A$;
        \item lax-$\u\bbbeta$-algebra structures based on $A$.
    \end{enumerate}
    \begin{proof}
      For concreteness, we here show that the datum of a natural transformation $\Delta \colon  A \circ \u\bbbeta A \Rightarrow A \circ \mu^{\u\bbbeta}_{A_0}$ corresponds exactly to a (suitably natural) choice of composite ultraconvergence arrows. We begin by describing the composite profunctor $A \circ \u\bbbeta A \colon \bbbeta^2 A_0 \pro A_0$: fix $a \in A_0$ and $(Y, C, \xi) \in \bbbeta^2 A_0$. An element of $A \circ \u\bbbeta A ( a , (Y, C, \xi))$ is given by an equivalence class of triples of:
      \begin{enumerate}
        \item an object $(X, b, \nu)$ in $\bbbeta A_0$,
        \item an ultraconvergence arrow $r \colon a \ult (b_x)_{x:\nu}$, and
        \item an element of $\u\bbbeta A ((X, b,\nu), (Y, C, \xi))$ which,
          identifying $A\colon \bbbeta A_0 \pro A_0$ with a two-sided codiscrete
          cofibration $\bbbeta A_0 \xrightarrow{d_1} R \xleftarrow{d_0} A_0$, we
          can represent as in \Cref{rem:explicit-description-on-profunctors} by
          a suitable equivalence class of a $\bbbeta R$-arrow $(p, \gamma)
          \colon (X, d_0b,\nu) \to (Y, d_1C, \xi)$.
      \end{enumerate}
      Two such pairs are identified if there exists a zig-zag of $\bbbeta A_0$-arrows making the evident triangles commute.

Suppose given a choice of composite ultraconvergence arrows. Let $a \in A_0$ and let $(Y,C,\xi)$ in $\bbbeta^2 A_0$; for each $y\in Y$, write $C(y) \coloneqq (Z_y, c_y, \chi_y)$. Consider an element of $(A\circ \u\bbbeta A) ( a , (Y, C, \xi))$, represented as above by some $(X,b,\nu)$ in $\bbbeta A_0$, an ultraconvergence arrow $r \colon a \ult (b_x)_{x:\nu}$, and a $\bbbeta R$-arrow $(p,\gamma) \colon (X, d_0b , \nu) \to (Y, d_1C, \xi)$. Note that $p$ trivially defines a $\bbbeta A_0$-arrow $p \colon (X, b,\nu) \to (Y, bp, \xi)$, so that we can consider the ultraconvergence arrow $r[p] \colon a \ult (b_{p(y)})_{y:\xi}$. Note also that, for each $y \in Y$, the component $\gamma_y \in A(b_{p(y)} , C(y))$ is an ultraconvergence arrow $b_{p(y)} \ult (c_{y,z})_{z:\chi_y}$. The composite $(\gamma_y)_{y:\xi} \cdot r[p]$ is therefore an ultraconvergence arrow $a \ult (c_{y,z})_{(y,z) : \sum_{y:\xi}\chi_y}$, i.e., an element of $A(a, \mu^{\bbbeta}_{A_0}(Y, C, \xi)) = (A \circ\mu^{\u\bbbeta}_{A_0} )(a, (Y,C, \xi))$. This assignment defines a transformation $\Delta \colon  A \circ \u\bbbeta A \Rightarrow A \circ \mu^{\u\bbbeta}_{A_0}$.

Conversely, suppose given a transformation $\Delta \colon  A \circ \u\bbbeta A \Rightarrow A \circ \mu^{\u\bbbeta}_{A_0}$. Let $r \colon a \ult (b_x)_{x:\nu}$ be an ultraconvergence arrow and let $(s_x \colon b_x \ult (c_{x,y})_{y:\xi_x})_{x:\nu}$ be an ultrafamily of ultraconvergence arrows. Consider then the object $(X, C, \nu)$ in $\bbbeta^2 A_0$ defined by $C(x) = (Y_x, c_x, \xi_x)$, so that $s_x \in A(b_x, C(x))$ for each $x \in X$: setting $\gamma_x \coloneqq s_x$ we have that the pair $(\id_X, \gamma)$ defines a $\bbbeta R$-arrow $(X, d_0b, \nu) \to (X, d_1 C, \nu)$. Applying the corresponding component of $\Delta$ to the triple given by $(X,b,\nu)$, $r$, and $(\id_X, \gamma)$ we obtain an element of $(A \circ \mu^{\u\bbbeta}_{A_0})( a , (X, C, \nu)) = A( a, \mu^{\bbbeta}_{A_0}(X, C, \nu))$, i.e., an ultraconvergence arrow $a \ult (c_{x,y})_{(x,y):\sum_{x:\nu}\xi_x}$ which we can set as the composite $(s_x)_{x:\nu}\cdot r$.

The two naturality axioms of an ultraconvergence structure then correspond, in the algebraic description, to the naturality of the multiplicator $\Delta$, while the other axioms correspond to the axioms of a lax algebra.

    \end{proof}

\end{theorem}

\subsection{Transformations via normalization}

We now move to arrows and $2$-cells between ultraconvergence spaces; first, let us recall the definitions from \cite{saadiaExtendingConceptualCompleteness2025,goolToposesEnoughPoints2025}.

\begin{definition}
A \emph{continuous map} of ultraconvergence spaces $ A \to A'$ consists of a functor $f \colon A_0 \to A_0'$ together with a \emph{continuity structure} on it, that is, a family of functions
\[ A(a , (X,b,\nu)) \to A'(f(a), (X, fb, \nu))\]
 also denoted by $f$, satisfying
\begin{enumerate}
	\item $f(r[g]) = f(r)[g]$;
	\item $f(\id_a) = \id_{f(a)}$;
	\item $f((s_x)_{x:\nu}\cdot r) =( f(s_x))_{x:\nu} \cdot f(r)$,
\end{enumerate}
for any ultraconvergence arrow $r \colon a \ult (b_x)_{x:\nu}$, any ultrafamily of ultraconvergence arrows $(s_x \colon b_x \ult (c_{x,y})_{y:\xi_x} )_{x:\nu}$, and $\bbbeta A_0$-arrow $g\colon (X,b,\nu) \to (Y,c,\xi)$.

A \emph{transformation} $f \Rightarrow g$ between continuous maps $f,g \colon A \to A'$ consists
of an ultraconvergence arrow $\alpha_a \colon f(a) \ult (g(a))_{*:1}$ in $A'$
for each $a\in A_0$, such that
\[ g(r) \cdot \alpha_a = (\alpha_{b_x})_{x:\nu} \cdot f(r)\]
for each ultraconvergence arrow $r \colon a \ult (b_x)_{x:\nu}$ in $A$.

\end{definition}

\begin{example}[\protect{\cite[Rem.\ 3.13]{goolToposesEnoughPoints2025}}]
  \looseness=-1
 \Cref{ex:top-spaces-as-ultspaces} extends to an embedding $\Top \hookrightarrow\UltSp$: continuous maps recover the topologically continuous maps, and transformations recover their (pointwise) \emph{specialization order}.
\end{example}

The datum of a continuity structure on $f \colon A_0 \to A_0'$ coincides with
what should be the structure of a \emph{colax morphism} of lax
$\u\bbbeta$-algebras on the representable profunctor $f_\diamond \colon A_0 \pro
A_0'$ --- that is, a natural transformation $A'\circ \u\bbbeta f_\diamond
\Rightarrow f_\diamond \circ A$. However, this notion, too, is not generally
well-defined for a left skew monad: as for lax algebras, it is not evident how
to write appropriate coherence axioms. Crucially, a notion of
\emph{representable colax morphism} can be defined for monads allowing for lax
algebras; we refer the reader to Appendix \ref{app:A5} for more details.

\begin{definition}\label{def:colax-morphisms-lax-algebras-allowed}
  Let $\monad{\u{\ps T}}$ be a left skew monad on $\PROF$ extending a pseudomonad $\monad{\ps T}$ on $\CAT$ via \Cref{cor:extension-pseudomonads} and suppose that $\u{\ps T}$ allows for lax algebras. Let $\braket{A_0, A, \Gamma,\Delta}$ and $\braket{A_0', A', \Gamma',\Delta'}$ be lax $\u{\ps T}$-algebras. Then, a \emph{representable colax morphism} $A \to A'$ consists of:
\begin{itemize}
  \item a functor $f\colon A_0 \to A_0'$, and
  \item a natural transformation $\Theta \colon f_\diamond \circ A \Rightarrow  A'\circ (\ps T f)_\diamond$,
\end{itemize}
satisfying the following coherence conditions:
\begin{enumerate}
  \item
  \[ \begin{tikzcd}
	& {A_0'} & {\ps T A_0'} \\
	{A_0} & {\ps T A_0} & {A_0'} \\
	& {A_0}
	\arrow["{A'}"'{inner sep=.8ex}, "\shortmid"{marking}, from=1-3, to=1-2]
	\arrow[""{name=0, anchor=center, inner sep=0}, "f", from=2-1, to=1-2]
	\arrow[""{name=1, anchor=center, inner sep=0}, "{\ps T f}"', from=2-2, to=1-3]
	\arrow["A"'{inner sep=.8ex}, "\shortmid"{marking}, from=2-2, to=2-1]
	\arrow["{\eta^{\ps T}_f}"', between={0.3}{0.7}, Rightarrow, from=2-2, to=2-3]
	\arrow["{\eta^{\ps T}_{A_0'}}"', from=2-3, to=1-3]
	\arrow[""{name=2, anchor=center, inner sep=0}, equals, from=3-2, to=2-1]
	\arrow["{\eta^{\ps T}_{A_0}}"', from=3-2, to=2-2]
	\arrow["f"', from=3-2, to=2-3]
	\arrow["\Theta", between={0.3}{0.7}, Rightarrow, from=0, to=1]
	\arrow["{{\Gamma}}"', between={0.3}{0.8}, Rightarrow, from=2, to=2-2]
\end{tikzcd} = \begin{tikzcd}
	{A_0'} && {\ps T A_0'} \\
	& {A_0'} \\
	& {A_0}
	\arrow["{A'}"'{inner sep=.8ex}, "\shortmid"{marking}, from=1-3, to=1-1]
	\arrow[""{name=0, anchor=center, inner sep=0}, equals, from=2-2, to=1-1]
	\arrow[""{name=1, anchor=center, inner sep=0}, "{{\eta^{\ps T}_{A_0'}}}"', from=2-2, to=1-3]
	\arrow["f"', from=3-2, to=2-2]
	\arrow["{{\Gamma'}}", between={0.3}{0.7}, Rightarrow, from=0, to=1]
\end{tikzcd}\]
\item
\[\begin{tikzcd}[column sep = 35pt, row sep = 25pt]
	& {\ps T^2 A_0} & \\
	{\ps T A_0} && {\ps T^2 A_0'} \\
	{A_0} & {\ps T A_0'} & {\ps T A_0'} \\
	& {A_0'}
	\arrow["{\u{\ps T }A}"', from=1-2, to=2-1]
	\arrow["{\ps T^2 f}", from=1-2, to=2-3]
	\arrow[""{name=0, anchor=center, inner sep=0}, "{\u{\ps T} (f_\diamond \circ A) }"{description, pos=0.25}, "\shortmid"{marking}, shift right=1, curve={height=12pt}, from=1-2, to=3-2]
	\arrow[""{name=1, anchor=center, inner sep=0}, "{\u{\ps T}(A' \circ (\ps T f)_\diamond)}"{description, pos=0.75}, "\shortmid"{marking}, shift left=1, curve={height=-12pt}, from=1-2, to=3-2]
	\arrow["A"'{inner sep=.8ex}, "\shortmid"{marking}, from=2-1, to=3-1]
	\arrow["{\ps T f}"', from=2-1, to=3-2]
	\arrow["{\u{\ps T} A'}"{inner sep=.8ex}, "\shortmid"{marking}, from=2-3, to=3-2]
	\arrow["{\mu^{\ps T }_{A_0'}}", from=2-3, to=3-3]
	\arrow["\Theta"', between={0.2}{0.8}, Rightarrow, from=3-1, to=3-2]
	\arrow["f"', from=3-1, to=4-2]
	\arrow["{{{\Delta_b}}}"', between={0.2}{0.8}, Rightarrow, from=3-2, to=3-3]
	\arrow["{A'}"'{inner sep=.8ex}, "\shortmid"{marking}, from=3-2, to=4-2]
	\arrow["{A'}"{inner sep=.8ex}, "\shortmid"{marking}, from=3-3, to=4-2]
	\arrow["{\phi_{A',(\ps Tf)_\diamond}}"', between={0.2}{0.8}, Rightarrow, from=1, to=2-3]
	\arrow["{\u{\ps T}\Theta }", between={0.2}{0.8}, Rightarrow, from=0, to=1]
	\arrow["{{{\psi_{f_\diamond,A}}}}"', between={0.2}{0.8}, Rightarrow, from=2-1, to=0]
\end{tikzcd} = \begin{tikzcd}[column sep = 35pt, row sep = 25pt]
	& {\ps T^2A_0} & \\
	{\ps TA_0} & {\ps T A_0} & {\ps T ^2 A_0'} \\
	{A_0} && {\ps T A_0'} \\
	& {A_0'}
	\arrow["{\u{\ps T }A}"'{inner sep=.8ex}, "\shortmid"{marking}, from=1-2, to=2-1]
	\arrow["{\mu^{\ps T}_{A_0}}", from=1-2, to=2-2]
	\arrow["{\ps T^2 f}", from=1-2, to=2-3]
	\arrow["\Delta"', between={0.2}{0.8}, Rightarrow, from=2-1, to=2-2]
	\arrow["A"'{inner sep=.8ex}, "\shortmid"{marking}, from=2-1, to=3-1]
	\arrow["{\mu_f^{\ps T}}"', between={0.2}{0.8}, Rightarrow, from=2-2, to=2-3]
	\arrow["A"{inner sep=.8ex}, "\shortmid"{marking}, from=2-2, to=3-1]
	\arrow["{\ps T f}"', from=2-2, to=3-3]
	\arrow["{\mu^{\ps T}_{A_0'}}", from=2-3, to=3-3]
	\arrow["\Theta"', between={0.3}{0.7}, Rightarrow, from=3-1, to=3-3]
	\arrow["f"', from=3-1, to=4-2]
	\arrow["{A'}"{inner sep=.8ex}, "\shortmid"{marking}, from=3-3, to=4-2]
\end{tikzcd}\]
\end{enumerate}

Given two such morphisms $\braket{f,\Theta}, \braket{f', \Theta'}$, an \emph{algebra $2$-cell} $\sigma \colon f \Rightarrow f'$ is a natural transformation $f \Rightarrow f'$ such that
\[ \begin{tikzcd}[row sep = 14pt]
	{\ps T A_0} && {\ps T A'_0} \\
	\\
	{A_0} && {A_0'}
	\arrow[""{name=0, anchor=center, inner sep=0}, "{{\ps T f}}"{description}, curve={height=18pt}, from=1-1, to=1-3]
	\arrow[""{name=1, anchor=center, inner sep=0}, "{{\ps T f'}}", curve={height=-18pt}, from=1-1, to=1-3]
	\arrow["A"'{inner sep=.8ex}, "\shortmid"{marking}, from=1-1, to=3-1]
	\arrow[""{name=2, anchor=center, inner sep=0}, "{A'}"{inner sep=.8ex}, "\shortmid"{marking}, from=1-3, to=3-3]
	\arrow["f"', curve={height=18pt}, from=3-1, to=3-3]
	\arrow["{{\ps T \sigma}}", between={0.3}{0.7}, Rightarrow, from=0, to=1]
	\arrow["\Theta"'{pos=0.35}, between={0.2}{0.7}, Rightarrow, from=3-1, to=2]
		\end{tikzcd} \qquad \begin{tikzcd}[row sep = 14pt]
	{\ps T A_0} && {\ps T A_0'} \\
	\\
	{A_0} && {A_0'}
	\arrow["{\ps T f'}", curve={height=-18pt}, from=1-1, to=1-3]
	\arrow[""{name=0, anchor=center, inner sep=.8ex},"\shortmid"{marking}, "A"', from=1-1, to=3-1]
	\arrow["{A'}"{inner sep =.8ex},"\shortmid"{marking}, from=1-3, to=3-3]
	\arrow[""{name=1, anchor=center, inner sep=0}, "{{{f'}}}"{description}, curve={height=-18pt}, from=3-1, to=3-3]
	\arrow[""{name=2, anchor=center, inner sep=0}, "f"', curve={height=18pt}, from=3-1, to=3-3]
	\arrow["{{{\Theta'}}}"{pos=0.65}, between={0.3}{0.8}, Rightarrow, from=0, to=1-3]
	\arrow["\sigma", between={0.3}{0.7}, Rightarrow, from=2, to=1]
		\end{tikzcd} \]

We denote by ${\nLaxAlgrco{\u{\ps T}}}$ the $2$-category of lax $\u{\ps T}$-algebras, representable colax morphisms, and algebra $2$-cells (see \Cref{prop:2cat-lax-algebras-allowed}).
\end{definition}

\begin{remark}\label{rem:colax-morphisms-are-okay}
  As in \Cref{rem:lax-algebras-are-okay}, note that if $\u{\ps T}$ is a pseudomonad then these definitions are equivalent to the usual definitions of (representable) colax morphisms between lax algebras and algebra $2$-cells between them. More generally, the coherence conditions above are formally equivalent to the usual coherence conditions defining colax morphisms (and, in particular, pseudomorphisms) of lax algebras for a pseudomonad on an arbitrary bicategory.
\end{remark}

With this definition, \Cref{thm:ucspaces-algebraically} extends straightforwardly to continuous maps, recovered as the {representable colax morphisms} of discrete lax $\u\bbbeta$-algebras.

\begin{proposition}\label{prop:continuous-maps-algebraically}
    Let $A,A'$ be ultraconvergence spaces and let $f \colon A_0 \to A_0'$ be a functor. Then, there is a bijection between:
    \begin{enumerate}
        \item continuity structures $A \to A'$ on $f$;
        \item structures of a colax morphism of lax $\u\bbbeta$-algebras $A\to A'$ on $f$.
    \end{enumerate}
\end{proposition}
\begin{proof}
  Note that, for $a' \in A_0'$ and $(X,b,\nu)$ in $\bbbeta A_0$, an element of $(f_\diamond \circ A ) (a', (X,b,\nu))$ is given by an ultraconvergence arrow $r \colon a \ult (b_x)_{x:\nu}$ in $A$ for some $a \in A_0$ such that $f(a) = a'$. Thus, the datum of a transformation $\Theta \colon f_\diamond \circ A \Rightarrow A' \circ (\bbbeta f)_\diamond$ corresponds exactly to a continuity structure $A \to A'$ on $f$, as both assignments map $r$ to some ultraconvergence arrow $f(r) \colon a' \ult (f(b_x))_{x:\nu}$, which is an element of $A'(a', (X, fb, \nu)) = (A'\circ (\bbbeta f)_\diamond) (a', (X, fb, \nu))$. One can then see that the naturality of $\Theta$ corresponds to the naturality of the continuity structure, and that the axioms of a colax morphism correspond exactly to preservation of identities and composites.
\end{proof}

This algebraic description of ultraconvergence spaces breaks when dealing with
$2$-cells. Indeed, the natural notion of an algebra $2$-cell, between continuous
maps of ultraconvergence spaces seen as colax morphisms of lax
$\u\bbbeta$-algebras, trivializes: the only such $2$-cells are identities, since
they are based on natural transformations between functors which are themselves
between discrete categories. To recover transformations between continuous maps
algebraically, we draw inspiration from the theory of \emph{generalized
  multicategories} \cite{cruttwellUnifiedFrameworkGeneralized2010}. We cannot
entirely place ourselves in this framework as our structures of interest are
naturally described by means of skew monads, rather than the strict monads considered
therein. However, we can adapt the correspondence between
\emph{object-discreteness} and \emph{normalization} studied \emph{ibid.} as follows.

\looseness=-1 First recall that a topology on a set $X$ equips its
$0$-dimensional structure, only consisting of disconnected elements, with
additional $1$-dimensional structure: namely, the specialization order. In the same way,
an ultraconvergence structure on a discrete category $A_0$ equips it with
$1$-dimensional structure, packaged in the \emph{category of points} of the
corresponding lax algebra.

\begin{definition}
  For a lax $\u\bbbeta$-algebra $A\colon \bbbeta A_0 \pro A_0$, the
  \emph{category of points}\footnote{In multicategorical terms, this is the
    \emph{underlying category} of the generalized multicategory $A$.} of $A$ is
  the category $\pt(A)$ whose objects are those of $A_0$ and whose morphisms $a
  \to a'$ are ultraconvergence arrows $a \ult (a')_{*:1}$.
\end{definition}

\begin{remark}
  In case $A_0$ is discrete and hence $A$ is an ultraconvergence space, $\pt(A)$
  is the category of continuous maps $1 \to A$ --- indeed, the \emph{points} of
  $A$~\cite[Def.\ 4.6]{saadiaExtendingConceptualCompleteness2025}.
\end{remark}

\looseness=-1 The profunctor $A\colon \bbbeta A_0 \pro A_0$ can be seen to
extend to a profunctor $\bbbeta (\pt ( A)) \pro \pt( A)$ which also naturally
carries a lax $\u \bbbeta$-algebra structure; with a slight abuse of notation, we denote with $\pt(A)$ both the algebra and its carrier category. The operation $A \mapsto \pt(A)$ forms a \emph{closure operator} on lax $\u \bbbeta$-algebras: indeed, $\pt(\pt(A))
\cong \pt(A)$, and the unitor $\Gamma\colon 1_{A_0} \Rightarrow A \circ
(\eta_{A_0}^\bbbeta)_\diamond$ of $A$ induces a comparison functor $\gamma
\colon A_0 \to \pt ( A )$ acting as the identity on objects and carrying the structure of a colax morphism.

\looseness=-1 The information that an ultraconvergence space encompasses is
exactly that of the profunctor on its category of points. One could consider
arbitrary lax $\u \bbbeta$-algebras instead of just the discrete ones, but it
may happen that two such algebras have the same categories of points. In this
sense, the $1$-dimensional data of the carrier category of an algebra is
superfluous, or ``unwelcome''~\cite[\S 4.3]{hylandElementsTheoryAlgebraic2014}.
The definition of ultraconvergence spaces hence requires the carrier category to
be discrete, precisely so as to exclude this data. Another approach described in \emph{loc.\ cit.}, which also
singles out the categories of points, is to ask for the comparison functor $\gamma\colon A_0 \to
\pt(A)$ to be fully-faithful, hence an isomorphism of categories, so that $A
\cong \pt(A)$. In other words, this amounts to restricting to those lax $\u\bbbeta$-algebras which are \emph{closed} with respect to the operator $\pt (-)$.

\begin{definition}
    A lax $\u\bbbeta$-algebra is \emph{normalized} if its comparison functor $\gamma$ is fully-faithful.
\end{definition}

The two descriptions, in terms of discreteness and of normalization, are equivalent for lax $\u\bbbeta$-algebras and their representable colax morphisms.

\begin{proposition}\label{prop:discrete-iff-normalized}
  There is a bijection between:
  \begin{enumerate}
  \item discrete lax $\u\bbbeta$-algebras, and
  \item normalized lax $\u\bbbeta$-algebras.
  \end{enumerate}
  This bijection extends to an equivalence of categories by considering
  representable colax morphisms.
\end{proposition}

\begin{remark}\label{rem:ultrasp-of-points}
  \Cref{ex:ultrasp-of-points} can now be more easily described as a lax
  $\u\bbbeta$-algebra structure on the actual category of models of a geometric
  theory $\bb T$, rather than on the discrete category underneath. This
  definition also shows how the profunctor $\bbbeta (\Mod (\bb T)) \pro \Mod
  (\bb T)$ defining ultraconvergence
  arrows satisfies our extra assumption of smallness: indeed, this follows from $\Mod{(\bb T)}$
  being accessible \cite[Cor.\ 4.3.2]{borceuxHandbookCategoricalAlgebra1994} and
  hence admitting a small dense subcategory (cf.\ \cite[Thm.\
  5.1]{kellyBasicConceptsEnriched2005}).
\end{remark}

Since the carrier categories of discrete lax $\u \bbbeta$-algebras do not
include the $1$-dimensional data present in the categories of points, the algebra $2$-cells between morphisms thereof,
as we noted, trivialize. On the other hand,
$2$-cells in the normalized description recover exactly the transformations
between continuous maps. 

\begin{proposition}\label{prop:transformations-algebraically}
  Let $A,A'$ be ultraconvergence spaces seen as
  normalized lax $\u\bbbeta$-algebras, and let $f,g \colon A \to A'$ be
  continuous maps seen as representable colax morphisms of lax
  $\u\bbbeta$-algebras. Then, there is a bijection between:
  \begin{enumerate}
  \item transformations $f \Rightarrow g$;
  \item algebra $2$-cells $f_\diamond \Rightarrow g_\diamond$.
  \end{enumerate}
\end{proposition}

In essence, by passing to a description in terms of normalization rather
than discreteness, we arrive at our second main contribution: we recover the $2$-category of ultraconvergence spaces in a concise and
algebraic way --- the only difference with \cite{saadiaExtendingConceptualCompleteness2025, goolToposesEnoughPoints2025} being the smallness bound.

\begin{theorem}\label{thm:ucspaces-algebraically-in-full}
    The $2$-category $\UltSp$ of ultraconvergence spaces is isomorphic to the $2$-category $\nLaxAlgrco{\u\bbbeta}$ of normalized lax $\u\bbbeta$-algebras, representable colax morphisms, and algebra $2$-cells.
\end{theorem}

In particular, we can characterize ultracategories among ultraconvergence spaces. Indeed, as in \cite[Thm.\ 9.2]{cruttwellUnifiedFrameworkGeneralized2010}, an ultraconvergence space is defined by a representable profunctor if and only if it corresponds to a \emph{colax $\bbbeta$-algebra}. Among these, we can single out those corresponding to pseudo-$\bbbeta$-algebras (that is, ultracategories), which we can identify as the \emph{representable} ultraconvergence spaces in the sense of \cite[Cor.\ 9.4]{cruttwellUnifiedFrameworkGeneralized2010}. Although the embedding was proved by hand in \cite[Thm.\ 4.4]{saadiaExtendingConceptualCompleteness2025}, we highlight how the algebraic framework makes it straightforward (cf.\ also \cite[Rem.\ 4.9]{saadiaExtendingConceptualCompleteness2025}).

\begin{corollary}\label{cor:ucats-as-uspaces}
    There is a 2-fully-faithful embedding $\UltCat \hookrightarrow \UltSp$ whose essential image is spanned by those ultraconvergence spaces $A \colon \bbbeta A_0 \pro A_0$ such that $A = \Phi_\diamond$ for some pseudo-$\bbbeta$-algebra $\Phi \colon \bbbeta A_0 \to A_0$.
\end{corollary}

\looseness=-1
For completeness, we state here the main result of \cite{saadiaExtendingConceptualCompleteness2025,goolToposesEnoughPoints2025}, expressing how a \emph{$\Set$-complete}\footnote{By this we mean a theory satisfying a completeness theorem with respect to its class of $\Set$-based models. This restriction is necessary, cf.\ \cite[\S 1]{goolToposesEnoughPoints2025}.}
geometric theory can be recovered from its category of models, once the latter is endowed with its canonical ultraconvergence structure. It is now evident how, via \Cref{cor:ucats-as-uspaces}, it specializes to \Cref{thm:reconstruction-coherent} in the coherent case, for which the ultraconvergence space of models is actually an ultracategory.

\begin{theorem}\label{thm:reconstruction-geometric}
    Let $\bb T$ be a $\Set$-complete geometric theory. Then, $\UltSp(\Mod {(\bb T)}, \Set)$ is the classifying topos of $\bb T$.
\end{theorem}

\subsection{Variations on the theme}\label{ssec:conv-spaces}

The discussions of the previous sections place the theory of ultracategories and ultraconvergence spaces within a much more general framework: namely, paralleling \cite{barrRelationalAlgebras1970}, that of \emph{profunctorial algebras} of a pseudomonad on $\CAT$, which we can define as the normalized lax algebras of its skew extension to $\PROF$ in the sense of \Cref{cor:extension-pseudomonads} --- of course, provided that the latter exists and allows for lax algebras. In this section, we follow this line of thought by introducing a variation of ultraconvergence spaces based on filters rather than ultrafilters, setting the grounds for future work.
Recall indeed that, by \Cref{cor:quotients-extend} and \Cref{prop:lokes-allow-for-lax-algebras}, the pseudomonad $\bb F$ also extends to a left skew monad
$\u{\bb F}$ on $\PROF$ allowing for lax algebras, so that we can give the following definition.

\begin{definition}
  \looseness=-1
  We denote by $\ConvSp$ the $2$-category of \emph{convergence spaces}: normalized lax $\u{\bb F}$-algebras,
  representable colax morphisms, and algebra $2$-cells.
\end{definition}

\looseness=-1
Exactly as in \Cref{ex:top-spaces-as-ultspaces}, note that the category of relational $\mathcal F$-algebras embeds into $\ConvSp$. As shown in
\cite{dayFilterMonadsContinuous1975}, the category $\Clos$ of closure spaces
embeds (reflectively) into the former, so that we deduce an embedding of $\Clos$ into the latter. Intuitively, we can thus think of our convergence spaces as a $2$-dimensional analogue of closure spaces.

\begin{corollary}
    The category $\Clos$ of closure spaces embeds fully-faithfully into $\ConvSp$.
\end{corollary}

\begin{remark}
  \looseness=-1
  More precisely, we are here considering relational $\mathcal F$-algebras as lax algebras for the \emph{Barr
extension} of $\mathcal F$ to $\Rel$ described in \cite[Def.\
1.10.1]{clementinoMonoidalTopologyCategorical2014}. As described in \cite[Ex.\ 1.10.3]{clementinoMonoidalTopologyCategorical2014},
there are two other (lax) extensions of $\mathcal F$ to $\Rel$, called
respectively \emph{canonical} and \emph{opcanonical} in
\cite{sealCanonicalOpcanonicalLax2005} (see also
\cite{schubertExtensionsTheoryLax2008}). While the three analogous extensions coincide for
$\beta$, they don't coincide for $\mathcal F$: in particular, lax algebras for its
canonical extension coincide with closure spaces, while those for its
opcanonical extension coincide with topological spaces.
\end{remark}

  \looseness=-1
Extending \Cref{ex:ultrasp-of-points}, note that the category of models of a geometric theory $\bb T$ in a signature $\Sigma$ carries the structure of a convergence space too. In the same way, \emph{convergence arrows} $M \ult (N_x)_{x:\nu}$ are defined as $\Sigma$-structure homomorphisms $M \to \prod_{x:\nu} N_x$: however, $\prod_{x:\nu} N_x$ is now a \emph{reduced product} of models, a construction which generalizes that of ultraproducts to arbitrary filters (see, e.g., \cite[\S 4, 6]{changModelTheory2012}).

\begin{question}
    Can a reconstruction theorem for geometric logic be proved in terms of reduced products, avoiding non-constructive generation of ultrafilters? The main difficulty in lifting the proof of \cite{goolToposesEnoughPoints2025} to the setting of convergence spaces lies in the use of \cite[Lem.\ 3.2]{goolToposesEnoughPoints2025}, whose analogous for filters is easily seen to be false.

    One possible workaround could be to redefine convergence spaces as the profunctorial algebras for a pseudomonad $\bb F_\leq$ obtained, with similar considerations to those of \Cref{sec:04}, starting from the functor $\mathcal F_\leq$ mapping a set $X$ to the \emph{poset} of (possibly improper) filters on $X$ ordered by
  inclusion. In essence, the only difference with $\bb F$ lies in relaxing the requirement
  on the functions defining morphisms: for a category $C$, $\bb F_\leq$-morphisms $(g,\alpha)
  \colon (X,h,\nu) \to (X',h',\nu')$ must satisfy $\nu \subseteq \mathcal F_\leq g
  (\nu')$ instead of $\nu = \mathcal F_\leq g (\nu')$. We leave this subject to future work.

\end{question}

\printbibliography

\newpage
\appendix
\section{Deferred definitions from \texorpdfstring{\Cref{sec:02}}{Section II}}\label{app:A2}
In this appendix we provide more details on the definitions from \Cref{sec:02}. Let us begin by spelling out the definitions of \emph{bicategory}, \emph{lax}/\emph{oplax}/\emph{pseudo}-\emph{functor}, \emph{lax}/\emph{oplax}/\emph{pseudo}-\emph{natural transformation}, and \emph{modification}.

\begin{definition}[\protect{\Cref{def:bicat}}]
  A \emph{bicategory} $\K$ consists of:
  \begin{itemize}
  \item a class of objects;
  \item for any two objects $A, B$, a \emph{hom}-category
    $\K(A,B)$ whose objects are called \emph{arrows} and denoted as
    $f \colon A {\to} B$, and whose morphisms are called \emph{$2$-cells} and denoted as
    $\sigma \colon f {\Rightarrow} g$;
  \item for any object $A$, an \emph{identity} arrow $1_A\colon A \to A$;
  \item for any three objects $A,B,C$, a \emph{composition} functor
  $\circ \colon \K(B,C) \times \K(A,B) \to \K(A,C)$;
  \item \looseness=-1 for any arrow $f\colon A \to B$, invertible $2$-cells $\rho_f\colon f \circ 1_A \cong f$ and
    $\lambda_f\colon 1_A \circ f \cong f$;
  \item for any three arrows $f\colon A \to B$, $g\colon B
    \to C$ and $h\colon C \to D$,
   an invertible $2$-cell $\alpha_{h,g,f}\colon h \circ (g \circ f) \cong
    (h \circ g) \circ f$.
  \end{itemize}
  This data is required to satisfy
  \begin{enumerate}
\item an \emph{identity coherence} axiom,
  \end{enumerate}
\small
\[\begin{tikzcd}
	{g \circ (1_B\circ f)} && {(g\circ 1_B)\circ f} \\
	& {g\circ f}
	\arrow["{\alpha_{g,1_B,f}}", from=1-1, to=1-3]
	\arrow["{g*\lambda_f}"', from=1-1, to=2-2]
	\arrow["{\rho_g*f}"', from=2-2, to=1-3]
\end{tikzcd}\]
\normalsize
\begin{enumerate}
\item[2)] and an \emph{associativity coherence} axiom,
\end{enumerate}
\small
\[\begin{tikzcd}[column sep = 12pt]
    & {(k\circ h)\circ (g\circ f)} \\
    {k\circ (h\circ (g\circ f))} && {((k\circ h)\circ g)\circ f} \\
    {k\circ ((h\circ g)\circ f)} && {(k\circ (h\circ g))\circ f}
    \arrow["{\alpha_{k\circ h, g,f}}", from=1-2, to=2-3]
    \arrow["{\alpha_{k,h,g\circ f}}", from=2-1, to=1-2]
    \arrow["{k*\alpha_{h,g,f}}"', from=2-1, to=3-1]
    \arrow["{\alpha_{k,h\circ g, f}}"', from=3-1, to=3-3]
    \arrow["{\alpha_{k,h,g}*f}"', from=3-3, to=2-3]
\end{tikzcd}\]
\normalsize
for all arrows $f\colon A \to B$, $g\colon B \to C$, $h\colon C \to D$ and $k \colon D \to E$ in $\K$.

In particular, $\K$ is a \emph{2-category} if each structural 2-cell is the identity.
\end{definition}

\begin{definition}
  An \emph{adjunction} $f \dashv g$ in a bicategory is a pair of arrows $f\colon
  A \to B$ (the \emph{left adjoint}) and $g\colon B \to A$ (the \emph{right
    adjoint}) together with \emph{unit} and \emph{counit} $2$-cells $\eta\colon
  1_A \Rightarrow gf$ and $\varepsilon\colon fg \Rightarrow 1_B$ satisfying:
  \[ g\varepsilon \cdot \eta g = 1_g \qquad \varepsilon f \cdot f \eta = 1_f \]
\end{definition}

Of course, adjunctions in $\Cat$ coincide with the ordinary notion of adjunction
of functors. The usual rules of the calculus of adjoint functors still stand in
an arbitrary bicategory (up to the coherence isomorphisms).

\begin{definition}\label{def:lax-functor}
    Let $\K,\K'$ be bicategories. A \emph{lax functor} $\ps{F}\colon \bicat K \to \bicat K'$ consists of
\begin{itemize}
  \item a mapping $A \mapsto \ps{F} A$ of objects of $\bicat K$ to objects of $\bicat K'$;
  \item for any two objects $A,B$ of $\bicat K$, a functor $\ps{F}_{A,B}\colon \bicat{K}(A,B) \to \bicat{K}'(\ps{F}A, \ps{F}B)$;
  \item \looseness=-1 for every object $A$ of $\bicat K$, a $2$-cell $\phi_{A}\colon 1_{\ps{F} A} \Rightarrow
    \ps{F}_{A,A}(1_A)$ in $\bicat {K}'$;
  \item for any two arrows $f\colon A \to B$ and $g\colon B \to C$ in $\bicat K$,
    a $2$-cell $\phi_{g,f}\colon \ps{F}(g) \ps{F}(f) \Rightarrow
    \ps{F}(g f)$ in $\bicat K'$.
\end{itemize}
    This data is required to satisfy:
    \begin{enumerate}
    \item two \emph{identity coherence} axioms,
    \end{enumerate}
\small
\[\begin{tikzcd}
	{F(1_B)\circ Ff} & {F(1_B\circ f)} \\
	{1_{FB}\circ Ff} & Ff
	\arrow["{\phi_{1_B,f}}", from=1-1, to=1-2]
	\arrow["{F(\lambda_f)}", from=1-2, to=2-2]
	\arrow["{\phi_{B}*Ff}", from=2-1, to=1-1]
	\arrow["{\lambda'_{Ff}}"', from=2-1, to=2-2]
\end{tikzcd} \, \begin{tikzcd}
	{Ff\circ 1_{FA}} & {Ff \circ F(1_A)} \\
	Ff & {F(f\circ 1_A)}
	\arrow["{Ff*\phi_{FA}}", from=1-1, to=1-2]
	\arrow["{\rho'_{Ff}}"', from=1-1, to=2-1]
	\arrow["{\phi_{f,1_A}}", from=1-2, to=2-2]
	\arrow["{F(\rho_f)}", from=2-2, to=2-1]
\end{tikzcd}\]
\normalsize
    \begin{enumerate}
    \item[2)]  and a \emph{composition coherence} axiom,
    \end{enumerate}
\small
\[\begin{tikzcd}
{Fh\circ (Fg\circ Ff)} && {(Fh\circ Fg)\circ Ff} \\
{Fh\circ F(g\circ f)} && {F(h\circ g)\circ Ff} \\
{F(h\circ (g\circ f))} && {F((h\circ g)\circ f)}
\arrow["{\alpha'_{Fh,Fg,Ff}}", from=1-1, to=1-3]
\arrow["{Fh*\phi_{g,f}}"', from=1-1, to=2-1]
\arrow["{\phi_{h,g}*Ff}", from=1-3, to=2-3]
\arrow["{\phi_{h,g\circ f}}"', from=2-1, to=3-1]
\arrow["{\phi_{h\circ g,f}}", from=2-3, to=3-3]
\arrow["{F(\alpha_{h,g,f})}"', from=3-1, to=3-3]
\end{tikzcd}\]
\normalsize
for all arrows $f\colon A \to B$, $g \colon B \to C$ and $h \colon C \to D$ in $\K$.

Similarly we define \emph{oplax functors}, by considering structural $2$-cells in the opposite direction, and \emph{pseudofunctors}, by taking them to be invertible. Between $2$-categories, we also speak of a \emph{$2$-functor} as a pseudofunctor whose structural $2$-cells are identities.

\end{definition}

\begin{definition}\label{def:lax-natural-transformations}
    Let $\ps F, \ps F' \colon \K \to \K'$ be lax functors. A \emph{lax natural transformation} $\sigma \colon \ps F \Rightarrow \ps F'$ consists of
    \begin{itemize}
        \item for any object $A$ of $\K$, an arrow $\sigma_A \colon \ps F A \to \ps F' A$ in $\K'$;
        \item for any arrow $f \colon A \to B$ in $\K$, a 2-cell $\sigma_f \colon \ps F' f \circ \sigma_A \Rightarrow \sigma_B \circ \ps F f$ in $\K'$.
    \end{itemize}
    This data is required to satisfy
\begin{enumerate}
    \item a \emph{naturality} axiom,
\end{enumerate}
\small
\[ \begin{tikzcd}
	{\ps FA} && {\ps FB} \\
	\\
	{\ps F'A} && {\ps F'B}
	\arrow[""{name=0, anchor=center, inner sep=0}, "{\ps F f}"{description}, curve={height=18pt}, from=1-1, to=1-3]
	\arrow[""{name=1, anchor=center, inner sep=0}, "{\ps F f'}", curve={height=-18pt}, from=1-1, to=1-3]
	\arrow["{\sigma_A}"', from=1-1, to=3-1]
	\arrow[""{name=2, anchor=center, inner sep=0}, "{\sigma_{B}}", from=1-3, to=3-3]
	\arrow["{\ps F' f}"', curve={height=18pt}, from=3-1, to=3-3]
	\arrow["{\ps F\alpha}", between={0.3}{0.7}, Rightarrow, from=0, to=1]
	\arrow["{\sigma_f}"'{pos=0.35}, between={0.2}{0.7}, Rightarrow, from=3-1, to=2]
		\end{tikzcd} = \begin{tikzcd}
		{\ps FA} && {\ps FB} \\
	\\
	{\ps F' A} && {\ps F' B}
	\arrow["{\ps F f'}", curve={height=-18pt}, from=1-1, to=1-3]
	\arrow[""{name=0, anchor=center, inner sep=0}, "{\sigma_A}"', from=1-1, to=3-1]
	\arrow["{\sigma_{B}}", from=1-3, to=3-3]
	\arrow[""{name=1, anchor=center, inner sep=0}, "{\ps F'f'}"{description}, curve={height=-18pt}, from=3-1, to=3-3]
	\arrow[""{name=2, anchor=center, inner sep=0}, "{\ps F'f}"', curve={height=18pt}, from=3-1, to=3-3]
	\arrow["{\sigma_{f'}}"{pos=0.65}, between={0.3}{0.8}, Rightarrow, from=0, to=1-3]
	\arrow["{\ps F'\alpha}", between={0.3}{0.7}, Rightarrow, from=2, to=1]
		\end{tikzcd} \]
\normalsize
\begin{enumerate}
    \item[2)]a \emph{unitality} axiom,
\end{enumerate}
\small
\[ \begin{tikzcd}
	{\ps FA} && {\ps F A} \\
	\\
	{\ps F'A} && {\ps F'B}
	\arrow[""{name=0, anchor=center, inner sep=0}, "{1_{\ps F A}}"{description}, curve={height=18pt}, from=1-1, to=1-3]
	\arrow[""{name=1, anchor=center, inner sep=0}, "{\ps F(1_A)}", curve={height=-18pt}, from=1-1, to=1-3]
	\arrow["{\sigma_A}"', from=1-1, to=3-1]
	\arrow[""{name=2, anchor=center, inner sep=0}, "{\sigma_A}"{description}, curve={height=12pt}, from=1-1, to=3-3]
	\arrow["{\sigma_A}", from=1-3, to=3-3]
	\arrow["{1_{\ps F' A}}"', curve={height=18pt}, from=3-1, to=3-3]
	\arrow["{\phi_A}", between={0.3}{0.7}, Rightarrow, from=0, to=1]
	\arrow["{\lambda_{\sigma_A}'}"', curve={height=6pt}, between={0}{0.8}, Rightarrow, from=3-1, to=2]
	\arrow["{{\rho'}^{-1}_{\sigma_A}}"', curve={height=-6pt}, between={0.2}{0.8}, Rightarrow, from=3-3, to=0]
		\end{tikzcd} = \begin{tikzcd}
	{\ps FA} && {\ps F A} \\
	\\
	{\ps F' A} && {\ps F' A}
	\arrow["{\ps F (1_A)}", curve={height=-18pt}, from=1-1, to=1-3]
	\arrow[""{name=0, anchor=center, inner sep=0}, "{{\sigma_A}}"', from=1-1, to=3-1]
	\arrow["{{\sigma_{A}}}", from=1-3, to=3-3]
	\arrow[""{name=1, anchor=center, inner sep=0}, "{{\ps F' (1_A)}}"{description}, curve={height=-18pt}, from=3-1, to=3-3]
	\arrow[""{name=2, anchor=center, inner sep=0}, "{1_{\ps F' A}}"', curve={height=18pt}, from=3-1, to=3-3]
	\arrow["{{\sigma_{1_A}}}"{pos=0.65}, between={0.3}{0.8}, Rightarrow, from=0, to=1-3]
	\arrow["{\phi'_A}", between={0.3}{0.7}, Rightarrow, from=2, to=1]
		\end{tikzcd} \]
\normalsize
\begin{enumerate}
    \item[3)]and an \emph{associativity} axiom,
\end{enumerate}
\small
\[ \begin{tikzcd}[column sep =12pt]
	{\ps FA} && {\ps F C} \\
	& {\ps F B} \\
	{\ps F' A} && {\ps F' C} \\
	& {\ps F' B}
	\arrow[""{name=0, anchor=center, inner sep=0}, "{\ps F (g\circ f)}", from=1-1, to=1-3]
	\arrow["{\ps F f}"', from=1-1, to=2-2]
	\arrow["{{\sigma_A}}"', from=1-1, to=3-1]
	\arrow["{\sigma_C}", from=1-3, to=3-3]
	\arrow["{\ps Fg}"', from=2-2, to=1-3]
	\arrow["{\sigma_B}"{description}, from=2-2, to=4-2]
	\arrow["{\sigma_f}"', between={0.2}{0.8}, Rightarrow, from=3-1, to=2-2]
	\arrow["{\ps F' f}"', from=3-1, to=4-2]
	\arrow["{\sigma_g}"', between={0.4}{0.6}, Rightarrow, from=4-2, to=1-3]
	\arrow["{\ps F' g}"', from=4-2, to=3-3]
	\arrow["{\phi_{g,f}}", between={0.2}{0.8}, Rightarrow, from=2-2, to=0]
		\end{tikzcd} = \begin{tikzcd}[column sep =12pt,row sep =20.6pt]
	{\ps FA} && {\ps F C} \\
	\\
	{\ps F' A} && {\ps F' C} \\
	& {\ps F' B}
	\arrow["{\ps F (g\circ f)}", from=1-1, to=1-3]
	\arrow["{{\sigma_A}}"', from=1-1, to=3-1]
	\arrow["{\sigma_C}", from=1-3, to=3-3]
	\arrow["{{\sigma_{g\circ f}}}", between={0.3}{0.7}, Rightarrow, from=3-1, to=1-3]
	\arrow[""{name=0, anchor=center, inner sep=0}, "{{\ps F' (g\circ f)}}", from=3-1, to=3-3]
	\arrow["{\ps F' f}"', from=3-1, to=4-2]
	\arrow["{\ps F' g}"', from=4-2, to=3-3]
	\arrow["{\phi_{g,f}'}", between={0.2}{0.8}, Rightarrow, from=4-2, to=0]
		\end{tikzcd} \]
\normalsize
for all arrows $f ,f' \colon A\to B$, $g \colon B \to C$ and all $2$-cells $\alpha \colon f \Rightarrow f'$ in $\K$.

Similarly we define \emph{oplax natural transformations}, by considering structural $2$-cells in the opposite direction, and \emph{pseudonatural transformations}, by taking them to be invertible. Moreover, it is evident how to define all three kinds of transformations between oplax functors and between pseudofunctors. We also speak of a \emph{pseudonatural equivalence} as a pseudonatural transformation whose components are equivalences.

\end{definition}

\begin{definition}\label{def:modification}
    Let $\sigma, \sigma ' \colon \ps F \Rightarrow \ps F'$ be lax natural transformations between lax functors $\ps F, \ps F' \colon \K \to \K'$. A \emph{modification} $\mathfrak m\colon \sigma \Rrightarrow \sigma'$ consists of a 2-cell $\mathfrak m_A \colon \sigma_A \Rightarrow \sigma'_A$ in $\K'$ for any object $A$ of $\K$, such that
\small
\[ \begin{tikzcd}
	{\ps FA} && {\ps F'A} \\
	\\
	{\ps F B} && {\ps F'B}
	\arrow[""{name=0, anchor=center, inner sep=0}, "{\sigma_A'}"{description}, curve={height=18pt}, from=1-1, to=1-3]
	\arrow[""{name=1, anchor=center, inner sep=0}, "{\sigma_A}", curve={height=-18pt}, from=1-1, to=1-3]
	\arrow["{\ps F f}"', from=1-1, to=3-1]
	\arrow[""{name=2, anchor=center, inner sep=0}, "{\ps F' f}", from=1-3, to=3-3]
	\arrow["{\sigma_B'}"', curve={height=18pt}, from=3-1, to=3-3]
	\arrow["{\mathfrak m_A}"', between={0.3}{0.7}, Rightarrow, from=1, to=0]
	\arrow["{{\sigma'_f}}"{pos=0.65}, between={0.3}{0.8}, Rightarrow, from=2, to=3-1]
		\end{tikzcd} = \begin{tikzcd}
	{\ps FA} && {\ps F' A} \\
	\\
	{\ps F B} && {\ps F' B}
	\arrow["{\sigma_A}", curve={height=-18pt}, from=1-1, to=1-3]
	\arrow[""{name=0, anchor=center, inner sep=0}, "{\ps F f}"', from=1-1, to=3-1]
	\arrow["{\ps F'f}", from=1-3, to=3-3]
	\arrow[""{name=1, anchor=center, inner sep=0}, "{\sigma_B}"{description}, curve={height=-18pt}, from=3-1, to=3-3]
	\arrow[""{name=2, anchor=center, inner sep=0}, "{\sigma'_B}"', curve={height=18pt}, from=3-1, to=3-3]
	\arrow["{{\sigma_{f}}}"'{pos=0.35}, between={0.2}{0.7}, Rightarrow, from=1-3, to=0]
	\arrow["{\mathfrak m_B}"', between={0.3}{0.7}, Rightarrow, from=1, to=2]
		\end{tikzcd} \]
\normalsize
for all arrows $f\colon A \to B$ in $\K$. Similarly we define modifications between all sorts of transformations introduced in \Cref{def:lax-natural-transformations}. A modification is \emph{invertible} if so is each of its structural $2$-cells.
\end{definition}

We now give the definitions concerning \emph{fibrations}, following
\cite{carboniModulatedBicategories1994,loregianCategoricalNotionsFibration2020}:
fix a bicategory $\K$. First, we define what it means for a 2-cell to be
\emph{cartesian} with respect to a span.

Note that \cite{carboniModulatedBicategories1994}
depicts a span from $B$ to $A$ as $A \leftarrow \cdot \rightarrow B$, instead of
$B \leftarrow \cdot \rightarrow A$ as done here and in
\cite{loregianCategoricalNotionsFibration2020}. The ``left'' properties we
define next thus refer to the arrows that we depict on the right side of spans,
and conversely ``right'' properties refer to the arrows on the left.

\begin{definition}
  Let $A \xleftarrow{q} E \xrightarrow{p} B$ be a span in $\K$ and let $e,e'
  \colon D \to E$. A 2-cell $\chi \colon e' \Rightarrow e$ is \emph{left
    cartesian} (to the span) if:
    \begin{enumerate}
        \item $q\chi$ is invertible;
        \item for each triple $(g,\xi,\alpha)$ of an arrow $g \colon L \to K$
          and 2-cells $\xi \colon e'' \Rightarrow eg$ and $\alpha \colon pe''
          \Rightarrow pe' g$ such that $p \xi = p\chi g \cdot \alpha$, there
          exists a unique 2-cell $\xi' \colon e'' \Rightarrow e' g$ such that
          $\xi = \chi g \cdot \xi'$ and $p \xi' = \alpha$.
    \end{enumerate}

    In $\co\K$, the span becomes $B \xleftarrow{p} E \xrightarrow{q} A$: a
    2-cell is \emph{right cartesian} with respect to $A \xleftarrow{q} E
    \xrightarrow{p} B$ if it is left cartesian in $\co\K$ with respect to $B
    \xleftarrow{p} E \xrightarrow{q} A$.
\end{definition}

\begin{definition}
  A span $A \xleftarrow{q} E \xrightarrow{p} B$ in $\K$ is a \emph{left
    fibration} when the following \emph{left path lifting} condition holds: for
  all arrows $e \colon D \to E$ and 2-cells $\gamma \colon a \Rightarrow pe$,
  there exists a \emph{left cartesian lift} $\chi \colon \gamma^* e \Rightarrow
  e$ for which there is an invertible 2-cell $\psi \colon a \cong p(\gamma^* e)$
  such that $\gamma = p\chi \cdot \psi$.

  Dually, $A \xleftarrow{q} E \xrightarrow{p} B$ is a \emph{right fibration}
  when the following \emph{right path lifting} condition holds: for all arrows
  $e \colon D \to E$ and 2-cells $\gamma \colon qe \Rightarrow b$, there exists
  a \emph{right cartesian lift} $\chi \colon e \Rightarrow \gamma_!e$ for which
  there is an invertible 2-cell $\psi \colon q(\gamma_!e) \cong b$ such that
  $\gamma = \psi \cdot q \chi$.

  A span is a \emph{two-sided fibration (TSF)} if it is both a left and a right
  fibration.
\end{definition}

\begin{definition}
  A span $A \xleftarrow{q} E \xrightarrow{p} B$ in $\K$ is \emph{discrete} if
  for all arrows $e,e' \colon D \to E$ and 2-cells $\xi,\zeta \colon e
  \Rightarrow e'$, if $q\xi = q\zeta$ and $p\xi = p\zeta$ with $q\xi$ and $p\xi$
  invertible, then $\xi = \zeta$ and they are invertible.

  A \emph{two-sided discrete fibration (TSDF)} is a discrete TSF.
\end{definition}

We now turn to the definitions of \emph{bicomma square} and \emph{bipullback} in $\K$.

\begin{definition}
  A \emph{bicomma object} of a cospan $A \xrightarrow{g} C
  \xleftarrow{f} B$ in $\K$ is an object $\bicomma{f}{g}$ in $\K$  equipped with a span $A \xleftarrow{\cod} \bicomma{f}{g} \xrightarrow{\dom} B$ and a $2$-cell $\gamma \colon f \dom \Rightarrow g \cod$ satisfying the following universal properties:
  \begin{enumerate}
  \item for every span $A \xleftarrow{p} X \xrightarrow{q} B$ and 2-cell
    $\sigma\colon fp \Rightarrow gq$, there are an arrow $u\colon X \to \bicomma{f}{g}$ and two 2-cells
    $\chi\colon \dom u \cong p$ and $\zeta\colon \cod u \cong q'$ such that
    $\sigma = g\zeta \cdot \gamma u \cdot f \chi$;
  \item for every pair of arrows $u, v\colon X \to \bicomma{f}{g}$ and of 2-cells $\chi\colon \dom u
    \Rightarrow \dom v$ and $\zeta\colon \cod u \to \cod v$ such that $\gamma
    v \cdot f \chi = g \zeta \cdot \gamma u$, there is a unique 2-cell $\sigma\colon
    u \Rightarrow v$ such that $\dom \sigma = \chi$ and $\cod \sigma = \zeta$.
  \end{enumerate}
  We refer to the span $A \xleftarrow{\cod} \bicomma{f}{g} \xrightarrow{\dom} B$ as a \emph{bicomma span}. If it exists, a bicomma object is unique up to essentially unique equivalence.

\end{definition}

\begin{definition}
  A \emph{bipullback} of a cospan $A \xrightarrow{g} C
  \xleftarrow{f} B$ in $\K$ is an object $A \times_C B$ equipped with a cospan $A \xleftarrow{\cod} A \times_C B \xrightarrow{\dom} B$ and an invertible $2$-cell $\gamma \colon f  \dom \cong g \cod$ satisfying the following universal properties:
    \begin{enumerate}
  \item for every span $A \xleftarrow{q} X \xrightarrow{p} B$ and 2-cell
    $\sigma\colon fq \cong gp$, there are an arrow $u\colon X \to A \times_C B$ and two 2-cells
    $\chi\colon \dom u \cong q$ and $\zeta\colon \cod u \cong p$ such that
    $\sigma = g\zeta \cdot \gamma u \cdot f \chi$;
  \item for every pair of arrows $u, v\colon X \to A \times_C B$ and of 2-cells $\chi\colon \dom u
    \Rightarrow \dom v$ and $\zeta\colon \cod u \Rightarrow \cod v$ such that $\gamma
    v \cdot f \chi = g \zeta \cdot \gamma u$, there is a unique 2-cell $\sigma\colon
    u \Rightarrow v$ such that $\dom \sigma = \chi$ and $\cod \sigma = \zeta$.
  \end{enumerate}
If it exists, a bipullback is unique up to essentially unique equivalence.

We now recall the precise definition and properties of orthogonal factorization systems on bicategories from \cite{bettiFactorizationsBicategories1999}. First recall that, for each object $A$ in $\K$, we have two \emph{representable} pseudofunctors $\K(A,-)\colon \K \to
\CAT$ and $\K(-,A)\colon \op\K \to \CAT$. The former, for
instance, takes an object $B$ to the hom-category $\K(A,B)$, an arrow $f\colon B
\to C$ to the functor $\K(A,B) \to \K(A,C)$ that post-composes by $f$,
\[\begin{tikzcd}[column sep=large]
	A & B && A &[-18pt]&[-18pt] C
	\arrow[""{name=0, anchor=center, inner sep=0}, "u", curve={height=-12pt}, from=1-1, to=1-2]
	\arrow[""{name=1, anchor=center, inner sep=0}, "v"', curve={height=12pt}, from=1-1, to=1-2]
	\arrow["{\K(A,f)}", between={0.2}{0.8}, maps to, from=1-2, to=1-4]
	\arrow[""{name=2, anchor=center, inner sep=0}, "{fu}"', curve={height=12pt}, from=1-4, to=1-6]
	\arrow[""{name=3, anchor=center, inner sep=0}, "{fv}", curve={height=-12pt}, from=1-4, to=1-6]
	\arrow["\gamma", between={0.2}{0.8}, Rightarrow, from=0, to=1]
	\arrow["{f\gamma}", between={0.2}{0.8}, Rightarrow, from=3, to=2]
\end{tikzcd}\] and a $2$-cell $\delta\colon f \Rightarrow g\colon B \to C$ to
the natural transformation $\K(A,f) \Rightarrow \K(A,g)$ with component at
$u\colon A \to B$ given by the $2$-cell $\delta u\colon fu \Rightarrow
gu$.

\begin{definition}[\Cref{def:orthogonal-factorization-system}]
A \emph{factorization system} on $\K$ is a pair $(\E, \M)$ of
  classes of arrows in $\K$ such that:
  \begin{enumerate}
  \item every arrow $f$ admits an \emph{$(\E,\M)$-factorization}, i.e.\ an
    invertible $2$-cell $f \cong me$ with $m \in \M$ and $e \in E$;
  \end{enumerate}
  \vspace{-1em}
 \begin{minipage}{.62\linewidth}
    \begin{enumerate}
    \item[(2)] every $\E$-arrow $e\colon X \to Y$ is \emph{orthogonal} to every $\M$-arrow $m\colon Z \to W$, in the sense that the square on the right is
      a bipullback in $\CAT$.
    \end{enumerate}
    \end{minipage}
    \begin{minipage}{.34\linewidth}
      \vspace{-0.4em}
      \[\quad \begin{tikzcd}[column sep=3.1em, row sep=scriptsize]
          {\K(Y,Z)} & {\K(Y,W)} \\
          {\K(X,Z)} & {\K(X,W)} \arrow["{\K(Y,m)}", from=1-1, to=1-2]
          \arrow[""{name=0, anchor=center, inner sep=0}, "{\K(e,Z)}"', from=1-1,
          to=2-1] \arrow[""{name=1, anchor=center, inner sep=0}, "{\K(e,W)}",
          from=1-2, to=2-2] \arrow["{\K(X,m)}"', from=2-1, to=2-2]
          \arrow["\cong"{description}, draw=none, from=0, to=1]
        \end{tikzcd}\]
    \end{minipage}

  In that case:
  \begin{itemize}
    \item $\E \cap \M$ consists of the equivalences in $\K$;
    \item $\E$ contains all arrows which are orthogonal to every $\M$-arrow;
    \item the factorization $f \cong m e$ is unique up to an equivalence which is itself uniquely determined up to a unique invertible 2-cell.
  \end{itemize}
\end{definition}

\end{definition}
 \section{Deferred proofs from \texorpdfstring{\Cref{sec:03}}{Section III}}\label{app:A3}

In this appendix we give the proofs of the extension theorems of \Cref{sec:03}.

\subsection{String diagrams for adjunctions in a bicategory}

We will write down $2$-cells in the bicategories of interest using the graphical
calculus of string diagrams, which we now recall briefly. Additional background
on string diagrams for bicategories can be found
in~\cite{streetCategoricalStructures1996}, and formal development of category
theory using string diagrams, notably the theory of adjunctions,
in~\cite{hinzeIntroducingStringDiagrams2023}.

Formally, a diagram in a bicategory is a $2$-graph: objects are vertices of the
graphs, arrows are (oriented) edges and $2$-cells are cells. The
string-diagrammatic depiction of such a diagram is then given by the dual graph:
$2$-cells are represented as vertices (called \emph{nodes}), arrows as edges
(called \emph{strings}) and objects as cells. For instance, the diagram
\[\begin{tikzcd}
	A && B && C
	\arrow[""{name=0, anchor=center, inner sep=0}, "g"{description,pos=0.3}, from=1-1, to=1-3]
	\arrow[""{name=1, anchor=center, inner sep=0}, "f", curve={height=-24pt}, from=1-1, to=1-3]
	\arrow[""{name=2, anchor=center, inner sep=0}, "h"', curve={height=24pt}, from=1-1, to=1-3]
	\arrow[""{name=3, anchor=center, inner sep=0}, "u", curve={height=-12pt}, from=1-3, to=1-5]
	\arrow[""{name=4, anchor=center, inner sep=0}, "v"', curve={height=12pt}, from=1-3, to=1-5]
	\arrow["\,\alpha", between={0.2}{0.8}, Rightarrow, from=1, to=0]
	\arrow["\,\beta", between={0.2}{0.8}, Rightarrow, from=0, to=2]
	\arrow["\,\gamma", between={0.2}{0.8}, Rightarrow, from=3, to=4]
\end{tikzcd}\] is depicted as \centerstringdiagram{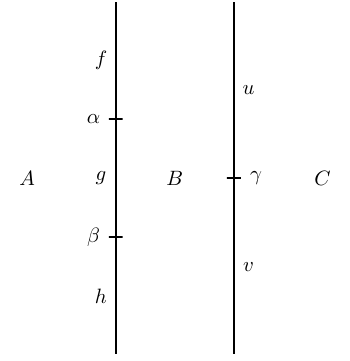} Composition inside
string diagrams is thus done left to right and top to bottom. Two string
diagrams that can be obtained from one another by stretching strings and
slidings nodes on strings represent the same composite $2$-cells.

In practice, we will omit to label objects, and will only label the source and
target arrows; we may also omit labels entirely when they are clear from the
context. \[ \stringdiagram{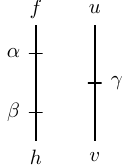} \]

The usual diagrammatic calculus of bicategories omits the unitors and
associators, and so does the string diagrammatic calculus. In practice this
amounts to assuming that the bicategory at hand is a $2$-category, which is made
rigorous by the strictification theorem for bicategories: every bicategory is
biequivalent to a $2$-category~\cite{powerCoherenceBicategoriesFinite1989}.

$2$-cells $f_1 \cdots f_n \Rightarrow g_1 \cdots g_m$ will more generally be depicted
as $(n+m)$-gones. A $2$-cell $f \Rightarrow gh$ will for instance be depicted
as a triangle, and a square $xu \Rightarrow yv$ as a square:
\[ \stringdiagram{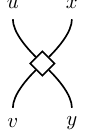} \]

Some of the composite $2$-cells we will consider will be invertible, and we
will use their inverses in our constructions. We depict these inverses by
flipping the string diagram vertically and enclosing it in a dashed box. For
instance:
\[ \left( \stringdiagram{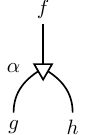}\right)^{-1} = \stringdiagram{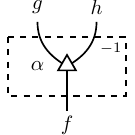} \]
We may omit the box if the $2$-cell being inverted is not a composite of
non-invertible $2$-cells. The usual string diagram manipulations can be done
both inside or outside such boxes without changing the composite $2$-cell that
is depicted, but $2$-cells can only be moved in and out of boxes if they are
themselves invertible. The usage of such boxes is made formal
in~\cite{melliesFunctorialBoxesString2006}.

Let $\K$ be a bicategory, and let $\Adj{\K}$ be the bicategory whose objects are
those of $\K$, arrows $A \to B$ are adjunctions $(f\colon A \to B, g\colon B \to
A, \eta\colon 1 \Rightarrow gf, \varepsilon\colon fg \Rightarrow 1)$ in $\K$ (so
that $g \varepsilon \cdot \eta g = \id_g$ and $\varepsilon f \cdot f \eta =
\id_f$), and $2$-cells $(f,g,\eta,\varepsilon) \Rightarrow
(f',g',\eta',\varepsilon')$ are \emph{mate pairs}, i.e., pairs $(\alpha\colon f
\Rightarrow f', \beta\colon g' \Rightarrow g)$ such that $\varepsilon' \cdot g'
\alpha = \varepsilon \cdot \beta f$ and $g \alpha \cdot \eta = \beta f' \cdot
\eta'$. Given an arrow $f\colon A \to B$ in $\Adj{\K}$, we write $f = (f_>, f^<,
\eta_f, \varepsilon_f)$, and similarly we write $2$-cells in $\Adj{\K}$ as
$\alpha = (\alpha_>, \alpha^<)$: this defines $2$-functors $-_>\colon \Adj{\K}
\to \K$ and $-^<\colon \op{\Adj{\K}} \to \co \K$.

In string diagrams, we use spikes to witness that arrows in $\K$ are in the
image of $-_>$ or $-^<$: we respectively write
\[ \stringdiagram{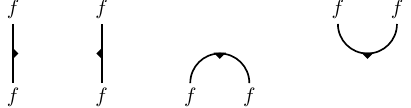} \] for the $2$-cells $\id_{f_>}$, $\id_{f^<}$,
$\eta_f$ and $\varepsilon_f$ in $\K$. Note that we label the source and targets
with $f$ (instead of $f_>$ or $f^<$) because the direction of the spike already
encodes which of $f_>$ or $f^<$ is being considered. Similarly, given a $2$-cell
$f \Rightarrow g$ in $\Adj{\K}$, we respectively write
\[ \stringdiagram{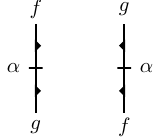}\] for the $2$-cells
$\alpha_>$ and $\alpha^<$ in $\K$, and the defining equations of mate pairs mean
we can slide these around bends:
\[ \stringdiagram{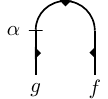} =
  \stringdiagram{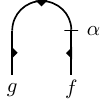} \qquad
  \stringdiagram{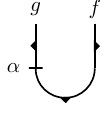} =
  \stringdiagram{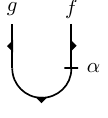} \] Note that the spikes are
only an annotation of the strings, meant to differentiate which arrow they
represent. In particular, they do not stand for any specific $2$-cell, and may
thus be duplicated or removed without changing the composite $2$-cell depicted
by the string diagram.

Finally, given a square $\gamma\colon x_>u \Rightarrow yv_>$, define
\[\gamma^\circlearrowright = x^< y \counit_v \cdot x^<
  \gamma v^< \cdot \eta_x u v^< \]
In string diagrams, we write this
\[ \stringdiagram{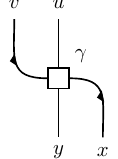} = \stringdiagram{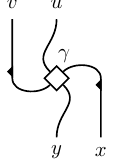} \]

\subsection{Two-sided discrete fibrations as a proarrow equipment}

Let $\K$ be a regular bicategory (cf.\ \Cref{thm:dfibs-bicategory}). Recall from
\Cref{sec:02} that $(-)_\diamond \colon \K \to \DFib{\K}$ is a proarrow
equipment. This means in particular that this pseudofunctor factors through the
$2$-functor $\Adj{\DFib{\K}} \to \DFib{\K}$: the right adjoint to $f_\diamond$
is $f^\diamond$, and given a square $\gamma\colon xu \Rightarrow yv$ in $\K$,
$\central \gamma$ (as defined in \Cref{ssec:exact-squares}) coincides with
$(\gamma_\diamond)^\circlearrowright$.

We now describe the units and counits of the adjunctions $f_\diamond \dashv
f^\diamond$. For this, first recall the following.

\begin{lemma}
  Given a cospan $A \xrightarrow{f} C \xleftarrow{g} B$, the two-sided discrete
  fibration $g^\diamond f_\diamond$ is a bicomma span of this cospan.
\end{lemma}

\begin{proof}
  Follows from \cite[Props.\ 4.26(a) and 1.7]{carboniModulatedBicategories1994}.
\end{proof}

By this lemma, $f^\diamond f_\diamond\colon A \pro A$ is the bicomma span of the
cospan at the bottom of the composite square
\[\begin{tikzcd}[column sep=tiny, row sep=scriptsize, cramped]
	&[5pt]&[-12pt]{A^\two \times_A A^\two} \\
	& {A^\two} & \cong &[-12pt]{A^\two} \\
	A && A &&[5pt] A \\
	& A & {=} & A \\
	&& B
	\arrow[from=1-3, to=2-2]
	\arrow[from=1-3, to=2-4]
	\arrow[from=2-2, to=3-1]
	\arrow[from=2-2, to=3-3]
	\arrow[from=2-4, to=3-3]
	\arrow[from=2-4, to=3-5]
	\arrow[equals, from=3-1, to=4-2]
	\arrow[between={0.4}{0.6}, Rightarrow, from=3-3, to=3-1]
	\arrow[equals, from=3-3, to=4-2]
	\arrow[equals, from=3-3, to=4-4]
	\arrow[between={0.4}{0.6}, Rightarrow, from=3-5, to=3-3]
	\arrow[equals, from=3-5, to=4-4]
	\arrow["f"', from=4-2, to=5-3]
	\arrow["f", from=4-4, to=5-3]
\end{tikzcd}\] while $\id_A \cong \id_A \id_A\colon A \pro A$ is the $\M$-part
in the $(\E,\M)$-factorization of the span at the top of this square. The unit
$\eta_f\colon \id_A \to f^\diamond f_\diamond$ is then induced by the universal
morphism of spans (from the top span to the bicomma span) that factors this square
into the bicomma square.

Similarly, the counit $\varepsilon_f\colon f_\diamond f^\diamond \Rightarrow
1_B$ is induced by the universal morphism of spans factoring the $2$-cell
\[\begin{tikzcd}[column sep=tiny, row sep=scriptsize, cramped]
    &&[-34pt] {\bicomma{f}{1_A} \times_A \bicomma{1_A}{f}} \\
    & {\bicomma f {1_A}} & \cong &[-34pt] {\bicomma {1_A} f} \\
    B && A && B \\
    & B & {=} & B \\
    && B
    \arrow[from=1-3, to=2-2]
    \arrow[from=1-3, to=2-4]
    \arrow[from=2-2, to=3-1]
    \arrow[from=2-2, to=3-3]
    \arrow[from=2-4, to=3-3]
    \arrow[from=2-4, to=3-5]
    \arrow[equals, from=3-1, to=4-2]
    \arrow[between={0.4}{0.6}, Rightarrow, from=3-3, to=3-1]
    \arrow["f"{description}, from=3-3, to=4-2]
    \arrow["f"{description}, from=3-3, to=4-4]
    \arrow[between={0.4}{0.6}, Rightarrow, from=3-5, to=3-3]
    \arrow[equals, from=3-5, to=4-4]
    \arrow[equals, from=4-2, to=5-3]
    \arrow[equals, from=4-4, to=5-3]
  \end{tikzcd}\] into the bicomma span $B \leftarrow B^\two \rightarrow B$.

More generally, in the above constructions, by replacing the squares $f 1_A = f
1_A$ and $1_B f = 1_B f$ by an arbitrary square $\gamma$ we retrieve $\central
\gamma$.

\begin{lemma}\label{lemma:dagger}
  Let $\gamma\colon xu \Rightarrow yv$ be a square in $\K$. In $\DFib{\K}$,
  $\central \gamma$ is the morphism of two-sided discrete fibrations induced by
  the universal morphism of spans that factors the square
  \[\begin{tikzcd}[column sep=tiny, row sep=scriptsize, cramped]
      &&[-32pt] {\bicomma{v}{1_C} \times_A \bicomma{1_B}{u}} \\
      & \bicomma{v}{1_C} & \cong &[-32pt] \bicomma{1_B}{u} \\
      C && A && B \\
      & C && B \\
      && D
      \arrow[from=1-3, to=2-2]
      \arrow[from=1-3, to=2-4]
      \arrow[from=2-2, to=3-1]
      \arrow[from=2-2, to=3-3]
      \arrow[from=2-4, to=3-3]
      \arrow[from=2-4, to=3-5]
      \arrow[from=3-1, to=4-2]
      \arrow[between={0.4}{0.6}, Rightarrow, from=3-3, to=3-1]
      \arrow["v"', from=3-3, to=4-2]
      \arrow["u", from=3-3, to=4-4]
      \arrow[between={0.4}{0.6}, Rightarrow, from=3-5, to=3-3]
      \arrow[from=3-5, to=4-4]
      \arrow["y"', from=4-2, to=5-3]
      \arrow["\gamma"', between={0.4}{0.6}, Rightarrow, from=4-4, to=4-2]
      \arrow["x", from=4-4, to=5-3]
    \end{tikzcd}\] into the bicomma span $x^\diamond y_\diamond$.

  In particular, $\gamma$ is an exact square if and only if the underlying
  $\K$-arrow of this universal morphism lies in $\E$.
\end{lemma}

\begin{proof}
  Temporarily write $\gamma \mapsto \gamma^*$ for the operation defined above.
  By definition, if $\alpha\colon f \Rightarrow g\colon A \to B$ is a $2$-cell
  in $\K$, then seeing $\alpha$ as a square $\alpha\colon 1_B f \Rightarrow g
  1_A$ we get that $\alpha^*$ is (up to the structural $2$-cells)
  $\alpha_\diamond\colon f_\diamond \Rightarrow g_\diamond$. Moreover $-^*$ is
  compositional: for a pasted square
  \[\begin{tikzcd}[column sep=scriptsize, cramped]
      & A \\
      C && B \\
      & D && E \\
      && F
      \arrow["v"', from=1-2, to=2-1]
      \arrow["u", from=1-2, to=2-3]
      \arrow["y"', from=2-1, to=3-2]
      \arrow["\gamma"', between={0.3}{0.7}, Rightarrow, from=2-3, to=2-1]
      \arrow["x"{description}, from=2-3, to=3-2]
      \arrow["t", from=2-3, to=3-4]
      \arrow["z"', from=3-2, to=4-3]
      \arrow["\beta"', between={0.3}{0.7}, Rightarrow, from=3-4, to=3-2]
      \arrow["w", from=3-4, to=4-3]
    \end{tikzcd}\] one shows by the universal property of bicomma squares that $(z
  \gamma \cdot \beta u)^* = \beta^* y^\diamond \circ t_\diamond \gamma^*$.
  Applying this to the composite square
  \[\begin{tikzcd}[column sep=scriptsize, cramped]
      & A \\
      A & {=} & A \\
      & C && B \\
      && D & {=} & D \\
      &&& D
      \arrow[equals, from=1-2, to=2-1]
      \arrow[equals, from=1-2, to=2-3]
      \arrow["v"', from=2-1, to=3-2]
      \arrow["v"{description}, from=2-3, to=3-2]
      \arrow["u", from=2-3, to=3-4]
      \arrow["y"', from=3-2, to=4-3]
      \arrow["\gamma"', between={0.3}{0.7}, Rightarrow, from=3-4, to=3-2]
      \arrow["x"{description}, from=3-4, to=4-3]
      \arrow["x", from=3-4, to=4-5]
      \arrow[equals, from=4-3, to=5-4]
      \arrow[equals, from=4-5, to=5-4]
    \end{tikzcd}\] we get that $\gamma_\diamond = \varepsilon_x y_\diamond
  v_\diamond \circ x_\diamond \gamma^* v_\diamond \circ x_\diamond u_\diamond
  \eta_v$ and hence, by postcomposing with $\varepsilon_v$ and precomposing with
  $\eta_x$ (and applying the triangle equalities involving the unit and counit)
  we get that $\central \gamma = \gamma^*$.
\end{proof}

\begin{lemma}[\protect{\Cref{lemma:left-class-exact}}]
      If $e$ is in $\E$, the square $1 e = 1 e$ is exact.
\end{lemma}
\begin{proof}
  By \Cref{lemma:dagger}, recall that the square $1e = 1e$ is exact if and only if the universal morphism of spans factoring the square
  \[\begin{tikzcd}[column sep=tiny, row sep=scriptsize, cramped]
      &&[-30pt] \bicomma{e}{1_B} \times_A \bicomma{1_B}{e} \\
      & \bicomma{e}{1_B} & \cong &[-30pt] \bicomma{1_B}{e} \\
      B && A && B \\
      & B & {=} & B \\
      && B
      \arrow[from=1-3, to=2-2]
      \arrow[from=1-3, to=2-4]
      \arrow[from=2-2, to=3-1]
      \arrow[from=2-2, to=3-3]
      \arrow[from=2-4, to=3-3]
      \arrow[from=2-4, to=3-5]
      \arrow[equals, from=3-1, to=4-2]
      \arrow[between={0.4}{0.6}, Rightarrow, from=3-3, to=3-1]
      \arrow["e"{description}, from=3-3, to=4-2]
      \arrow["e"{description}, from=3-3, to=4-4]
      \arrow[between={0.4}{0.6}, Rightarrow, from=3-5, to=3-3]
      \arrow[equals, from=3-5, to=4-4]
      \arrow[equals, from=4-2, to=5-3]
      \arrow[equals, from=4-4, to=5-3]
    \end{tikzcd}\]
	into the bicomma span $B \leftarrow B^\two \rightarrow B$ lies in $\E$.

  By \cite[Prop.\ 4.26(a)]{carboniModulatedBicategories1994}, it is in fact
  enough to only consider the two bottom right squares. By \cite[Prop.
  1.7]{carboniModulatedBicategories1994}, the span at the top of these two
  squares is the same as the one at the top of the diagram
  \[\begin{tikzcd}[column sep=tiny, row sep=scriptsize, cramped]
      &[10pt] &[-5pt] {B^\two \times_B A} \\
      & {B^\two} & \cong &[-5pt] A \\
      B && B & {=} &[10pt] B \\
      & B & & \\
      && B
      \arrow[from=1-3, to=2-2]
      \arrow[from=1-3, to=2-4]
      \arrow[from=2-2, to=3-1]
      \arrow[from=2-2, to=3-3]
      \arrow["e"{description}, from=2-4, to=3-3]
      \arrow["e"{description}, from=2-4, to=3-5]
      \arrow[equals, from=3-1, to=4-2]
      \arrow[between={0.4}{0.6}, Rightarrow, from=3-3, to=3-1]
      \arrow[equals, from=3-3, to=4-2]
      \arrow[equals, from=3-5, to=5-3]
      \arrow[equals, from=4-2, to=5-3]
    \end{tikzcd}\] \looseness=-1 The arrow $B^\two \times_B A \to B^\two$ is the
  one underlying the universal morphism of spans factoring this top span into $B
  \leftarrow B^\two \rightarrow B$. Since $\E$ is stable under bipullbacks,
  this arrow is also in $\E$.
\end{proof}

\begin{lemma}[\protect{\Cref{lemma:bicommas-exact}}]
  Bicomma squares are exact.
\end{lemma}
\begin{proof}
  In \cite[Prop. 4.25]{carboniModulatedBicategories1994}, if the two-sided
  discrete fibration is a bicomma span, the $\E$-arrow that is constructed is the
  universal morphism described in \Cref{lemma:dagger}.
\end{proof}

\begin{corollary}[\protect{\Cref{cor:composition-with-bicomma}}]
    Composition of TSDFs can be performed by means of bicomma squares instead of
  bipullbacks.
\end{corollary}

\begin{proof}
  Consider TSDFs \[ e = A \xleftarrow{q} E \xrightarrow {p} B \qquad f = B
    \xleftarrow{s} F \xrightarrow{r} C \]
  and write $E \xleftarrow{\cod} \bicomma{s}{p} \xrightarrow{\dom} F$ for the
  bicomma of $E \xrightarrow{p} B \xleftarrow{s} F$.

  Since $e \cong p_\diamond q^\diamond$ and $f \cong r_\diamond
  s^\diamond$~\cite[Prop. 4.26(a)]{carboniModulatedBicategories1994}, \[ fe
    \cong r_\diamond s^\diamond p_\diamond q^\diamond \cong r_\diamond
    (\cod)_\diamond (\dom)^\diamond q^\diamond .\] By~\cite[\textsection
  4.15]{carboniModulatedBicategories1994}, $A \leftarrow \bicomma{s}{p}
  \rightarrow C$ is a TSF. The construction of \cite[Prop.
  4.25]{carboniModulatedBicategories1994} also works for TSFs and shows that the
  $\M$-morphism in the $(\E,\M)$ factorization of $\bicomma{s}{p} \to A \times C$
  can be chosen to be $r_\diamond (d_1)_\diamond (d_0)^\diamond q^\diamond \cong
  fe$.
\end{proof}

\subsection{Proofs of the extension theorems}

We start by constructing the various extensions, and then we address uniqueness
and the converse implications. Fix from now on a regular bicategory $\K$ and a
bicategory $\K'$.

\subsubsection{Constructing extensions (pseudofunctors)}

Let $\ps F\colon \K \to \Adj{\K'}$ be a pseudofunctor such that,
\begin{itemize}
\item for every co-fully-faithful arrow $f: A \to B$ in $\K$, the $2$-cell
  $\varepsilon_{\ps F f}: (\ps F f)_> (\ps F f)^< \Rightarrow 1_B$ is invertible.
\end{itemize}
We extend $\ps F_> = (\ps F -)_>$ to an oplax functor $\u{\ps F}\colon \DFib{\K}
\to \K'$, as witnessed by an oplax natural transformation $\delta^{\ps F}:
\u{\ps F}(-)_\diamond \Rightarrow (\ps F -)_\diamond$.

On the way, we moreover show that $\u{\ps F}$ and $\delta^{\ps F}$ are
respectively a pseudofunctor and a pseudonatural transformation as soon as the
following condition holds:
\begin{description}
\item[BCC\hypertarget{hyp:BCC-functor}] for every exact square $\gamma\colon xu \Rightarrow yv$ in $\K$,
  $(\ps F \gamma)_>^\circlearrowright \colon (\ps F u)_> (\ps F v)^< \Rightarrow
  (\ps F x)^< (\ps F y)_>$ is invertible.
\end{description}

\noindent \emph{The family of functors.}
First, we define $\u {\ps F}$ on objects by setting $\underline{\ps F}A = \ps F A$,
and on arrows by setting:
\[ \u{\ps F} (A \xleftarrow{q} E \xrightarrow{p} B) = (\ps F p)_> (\ps F q)^< \]
Consider now a $2$-cell in $\DFib{\K}$:
\[\begin{tikzcd}[sep = 20pt]
	& E \\
	A && B \\
	& {E'}
	\arrow["q"', curve={height=12pt}, from=1-2, to=2-1]
	\arrow["p", curve={height=-12pt}, from=1-2, to=2-3]
	\arrow[""{name=0, anchor=center, inner sep=0}, "h"{description}, from=1-2, to=3-2]
	\arrow["{q'}", curve={height=-12pt}, from=3-2, to=2-1]
	\arrow["{p'}"', curve={height=12pt}, from=3-2, to=2-3]
	\arrow["\chi"', between={0.3}{0.7}, Rightarrow, from=0, to=2-1]
	\arrow["\zeta", between={0.3}{0.7}, Rightarrow, from=2-3, to=0]
\end{tikzcd}\] We set the image of this $2$-cell under $\underline{\ps F}$ to be
the $2$-cell in $\bicat M$ obtained by pasting:
\[ (\ps F p')_> \counit_{\ps F h} (\ps F q')^< \cdot (\ps F \zeta)_> (\ps F \chi)^< \]
As a string diagram, this is
\[ \stringdiagram{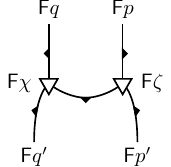} \]

We have implicitly omitted to depict the associator of $\ps F$ in the string
diagram above: for instance, by $\ps F \zeta$ we really mean the composite
\[ \begin{tikzcd}
	{\ps F p} & {\ps F (p'h)} &[10pt] {(\ps F p') (\ps F h)}
	\arrow["{\ps F \zeta}", Rightarrow, from=1-1, to=1-2]
	\arrow["{\phi_{p',h}^{-1}}", Rightarrow, from=1-2, to=1-3]
\end{tikzcd} \]
From now on we omit in this way all the structural $2$-cells
making $\K'$ a bicategory and $\ps F$ a pseudofunctor. For full rigor, these can
be considered to be merged with the adjacent $2$-cells, and can be added back in
a canonical way.

The above defines an action of $\ps F$ on $\DFib{\K}(A,B)$: given an isomorphism
between morphisms of spans
\[\begin{tikzcd}[sep = 20pt]
	&& E \\
	A &&&& B \\
	&& {E'}
	\arrow["q"', curve={height=18pt}, from=1-3, to=2-1]
	\arrow["p", curve={height=-18pt}, from=1-3, to=2-5]
	\arrow[""{name=0, anchor=center, inner sep=0}, "{h'}"{description}, curve={height=18pt}, from=1-3, to=3-3]
	\arrow[""{name=1, anchor=center, inner sep=0}, "h"{description}, curve={height=-18pt}, from=1-3, to=3-3]
	\arrow["{q'}", curve={height=-18pt}, from=3-3, to=2-1]
	\arrow["{p'}"', curve={height=18pt}, from=3-3, to=2-5]
	\arrow["\chi"', between={0.2}{0.8}, Rightarrow, from=0, to=2-1]
	\arrow["\nu"', between={0.2}{0.8}, Rightarrow, from=1, to=0]
	\arrow["\zeta"', between={0.2}{0.8}, Rightarrow, from=2-5, to=1]
\end{tikzcd}\]
the images of these two morphisms of spans coincide:
\[ \stringdiagram{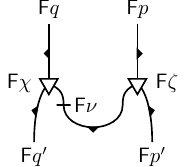} =
  \stringdiagram{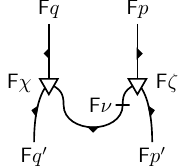} \] This action is moreover functorial:
because $\ps F$ is a pseudofunctor with codomain $\Adj{\K'}$, its unitor and
associator are compatible with the units and counits in $\K'$, and it follows
that $\ps F$ preserves identity $2$-cells and composition of $2$-cells. In
string diagrams, because we do not depict the structural $2$-cells involved in
the corresponding equations, this is trivial. As an example, if we were to
actually depict the unitors of $\K$ and $\ps F$, proving preservation of
identities would look like
\[ \stringdiagram{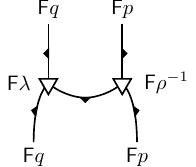} = \stringdiagram{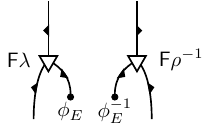} =
  \stringdiagram{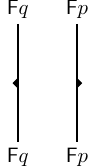} \] (recall the invertible $2$-cell
$\phi_E\colon 1_{\ps F E} \Rightarrow \ps F(1_E)$ and the coherence axioms it
satisfies with respect to the left and right unitors in $\K$).

\bigskip
\noindent \emph{The structural $2$-cells.} We now construct the structural
$2$-cells making $\u{\ps F}$ into a oplax functor. For an object $A$ of $\K$,
the $2$-cell $\underline{\ps F}(A \xleftarrow{\cod} A^\two \xrightarrow{\dom} A)
\Rightarrow 1_{\ps F A} $, witnessing preservation of identities, is constructed
as the composite
\[ (\ps F \dom)_> (\ps F \cod)^< \Rightarrow (\ps F 1_A)^< (\ps F 1_A)_> \cong
  1_{\ps F A} \]using the $2$-cell induced by the image of the bicomma square $1_A \dom
\Rightarrow 1_A \cod$ for the first step, and coherence in $\K'$ and the
pseudofunctoriality of $\ps F$ for the second step. Note that this $2$-cell is
invertible if \hyperlink{hyp:BCC-functor}{BCC} holds. As a string diagram, it is
depicted
\[ \stringdiagram{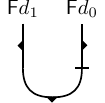} \]Consider next two TSDFs and their composite:
\[ e = A \xleftarrow{q} E \xrightarrow {p} B \quad f = B \xleftarrow{s} \ps F
  \xrightarrow{r} C \quad fe = A \xleftarrow{y} G \xrightarrow{x} C \]Let $E \xleftarrow{\cod} \bicomma{s}{p} \xrightarrow{\dom} F$ be the bicomma
span of the cospan $E \xrightarrow{p} B \xleftarrow{s} F$, and write $h\colon
\bicomma{s}{p} \to G$ for the arrow factoring $A \xleftarrow{q \cod}
\bicomma{s}{p} \xrightarrow{p \dom} C$ into $gf$, which by
\Cref{cor:composition-with-bicomma} lies in $\mathcal E$:
\[\begin{tikzcd}[row sep=1pt, column sep=large]
    && {\bicomma{s}{p}} \\
    & E && F \\
    A && G && C
    \arrow["\cod"', from=1-3, to=2-2]
    \arrow["\dom", from=1-3, to=2-4]
    \arrow[""{name=0, anchor=center, inner sep=0}, "h", dashed, from=1-3, to=3-3]
    \arrow["q"', from=2-2, to=3-1]
    \arrow["r", from=2-4, to=3-5]
    \arrow["y", from=3-3, to=3-1]
    \arrow["x"', from=3-3, to=3-5]
    \arrow["\cong"{description, pos=0.2}, draw=none, from=0, to=3-5]
    \arrow["\cong"{description, pos=0.8}, draw=none, from=3-1, to=0]
  \end{tikzcd}\] The $2$-cell $\underline{\ps G}(fe) \Rightarrow (\underline{\ps
  F}f)(\underline{\ps G}e)$, witnessing preservation of composition, is
constructed as the composite
\begin{align*}
  (\ps F x) (\ps F y)^r
  &\cong (\ps F x) (\ps F h) (\ps F h)^r (\ps F y)^r \\
  &\cong (\ps F r)(\ps F \dom)(\ps F \cod)^r(\ps F q)^r \\
  &\Rightarrow (\ps F r)(\ps F s)^r(\ps F p)(\ps F q)^r
\end{align*}
using preservation of the co-fully-faithful morphism $h$ for the first step,
pseudofunctoriality of $\ps F$ for the second step, and the $2$-cell induced by
the image of the bicomma square $s \dom \Rightarrow p \cod$ for the third step.
Note that this $2$-cell is invertible if \hyperlink{hyp:BCC-functor}{BCC} holds.
As a string diagram, it is depicted
\[ \stringdiagram{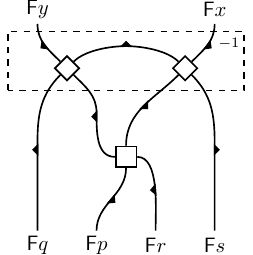} \]

\bigskip
\noindent\emph{The coherence axioms.}
For the identity coherence axioms, consider TSDFs
\[ 1_A = A \xleftarrow{\cod} A^\two \xrightarrow{\dom} A \quad e = A
  \xleftarrow{q} E \xrightarrow{p} B \quad e 1_A = A \xleftarrow{y} F
  \xrightarrow{x} B \]the bicomma square
\[\begin{tikzcd}[column sep=large, row sep=1pt]
    && {\bicomma{p}{\dom}} \\
    & {A^\two} && E \\
    A && A && B \arrow["s"', from=1-3, to=2-2]\arrow["r", from=1-3, to=2-4]\arrow["\cod"', from=2-2, to=3-1]\arrow["\dom"', from=2-2, to=3-3]\arrow["\gamma"', between={0.4}{0.6}, Rightarrow, from=2-4, to=2-2]\arrow["q", from=2-4, to=3-3]\arrow["p", from=2-4, to=3-5]\end{tikzcd}\]and the $(\E,\M)$-factorization
\[\begin{tikzcd}[row sep=1pt, column sep=large]
	&& {\bicomma{p}{d_0}} \\
	& {A^\two} && E \\
	A && F && B \arrow["s"', from=1-3, to=2-2]\arrow["r", from=1-3, to=2-4]\arrow["h'"{description}, from=1-3, to=3-3]\arrow["\cod"', from=2-2, to=3-1]\arrow["\zeta'"', shift right=1, between={0.3}{0.7}, Rightarrow, from=2-4, to=3-3]\arrow["p", from=2-4, to=3-5]\arrow["\chi'"', shift right=1, between={0.3}{0.7}, Rightarrow, from=3-3, to=2-2]\arrow["y", from=3-3, to=3-1]\arrow["x"', from=3-3, to=3-5]\end{tikzcd}\]For the unitor $e 1_A \cong e$ in $\DFib{K}$, recall that it is a morphism of
spans $F \to E$ whose precomposite with $h'\colon \bicomma{p}{\dom} \to F$ is a
morphism $h\colon \bicomma{p}{\dom} \to E$ that factors the square
\[\begin{tikzcd}[column sep=huge]
	A & {A^\two} & {\bicomma{p}{\dom}} \\
	A & A & E \arrow[from=1-1, to=2-1]\arrow["\cod"', from=1-2, to=1-1]\arrow["\dom"{description}, from=1-2, to=2-2]\arrow["s"', from=1-3, to=1-2]\arrow["r", from=1-3, to=2-3]\arrow["\alpha"', between={0.3}{0.7}, Rightarrow, from=2-2, to=1-1]\arrow[equals, from=2-2, to=2-1]\arrow["\gamma"', between={0.4}{0.6}, Rightarrow, from=2-3, to=1-2]\arrow["q", from=2-3, to=2-2]\end{tikzcd}\]as
\[\begin{tikzcd}[column sep=huge, column sep=large]
	A & {A^\two} & {\bicomma{p}{\dom}} \\
	A & A & E \arrow[from=1-1, to=2-1]\arrow["\cod"', from=1-2, to=1-1]\arrow["s"', from=1-3, to=1-2]\arrow[""{name=0, anchor=center, inner sep=0}, "r", curve={height=-24pt}, from=1-3, to=2-3]\arrow[""{name=1, anchor=center, inner sep=0}, "h"', curve={height=24pt}, from=1-3, to=2-3]\arrow[equals, from=2-2, to=2-1]\arrow["\chi"', between={0.4}{0.6}, Rightarrow, from=2-3, to=1-1]\arrow["q", from=2-3, to=2-2]\arrow["\,\zeta"', between={0.3}{0.7}, Rightarrow, from=0, to=1]\end{tikzcd}\]where $\chi$ and $p \zeta$ are invertible. The proof of the identity coherence
axiom corresponding to this unitor is then
\begin{align*}
  \stringdiagram{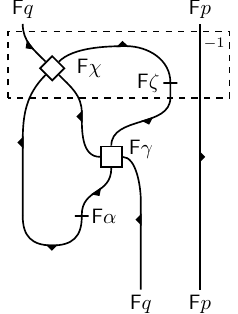}
  &= \stringdiagram{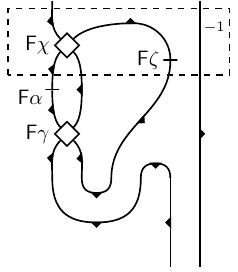} \\
  &= \stringdiagram{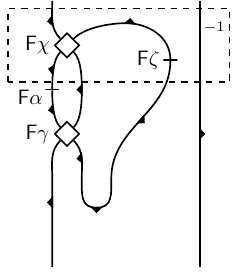} \\
  &= \stringdiagram{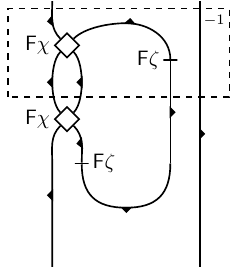} \\
  &= \stringdiagram{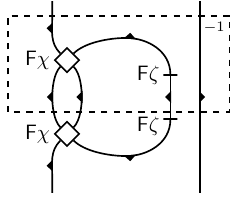} \\
  &= \stringdiagram{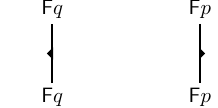} \\
\end{align*}

The proof for the other identity coherence axiom is symmetric to the above one.
Similarly, for the composition coherence axiom, consider three TSDFs and
associated squares as below.
\[\begin{tikzcd}[row sep=1pt]
	&&& {\bicomma{s}{p} \times_F \bicomma{t}{r}} \\
	&& {\bicomma{s}{p}} & \cong & {\bicomma{t}{r}} \\
	& E && F && G \\
	A && B && C && D \arrow[from=1-4, to=2-3]\arrow[from=1-4, to=2-5]\arrow[from=2-3, to=3-2]\arrow[from=2-3, to=3-4]\arrow[from=2-5, to=3-4]\arrow[from=2-5, to=3-6]\arrow["q"', from=3-2, to=4-1]\arrow["p"', from=3-2, to=4-3]\arrow[between={0.5}{0.7}, Rightarrow, from=3-4, to=3-2]\arrow["s", from=3-4, to=4-3]\arrow["r"', from=3-4, to=4-5]\arrow[between={0.3}{0.5}, Rightarrow, from=3-6, to=3-4]\arrow["u", from=3-6, to=4-5]\arrow["t", from=3-6, to=4-7]\end{tikzcd}\]The two composites of interest, on the left and right side of the square
defining the composition coherence axiom, can both be shown to factor through a
$2$-cell with string diagram of the shape
\[ \stringdiagram{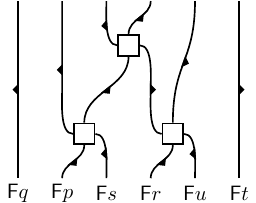} \](where the three squares correspond to the three squares in the previous
diagram), respectively pre-composed with images of morphisms of spans $\ps F h$
and $\ps F h'$ with codomain $\ps F (\bicomma{s}{p} \times_F \bicomma{t}{r})$.
The associator in $\DFib{\K}$, being constructed by the orthogonality of the
factorization system and the universal property of TSDFs, precisely sends $h$
and $h'$, and thus the two $2$-cells of interest are equal.

\bigskip
\noindent \emph{The natural transformation.} We finally show that $\u{\ps F}$
extends $\ps F_>$ by constructing an oplax natural transformation $\delta\colon
\u{\ps F}(-)_\diamond \Rightarrow \ps F_>$. We let $\delta_A$ be the identity
$1_{\ps FA}\colon \ps F A = \ps F A$ for every object $A$ of $\K$. Given an
arrow $f\colon A \to B$ in $\K$ and writing $f_\diamond = A \xleftarrow{\cod}
\bicomma{1_A}{f} \xrightarrow{\dom} B$, we let $\delta_f\colon \u {\ps
  F}(f_\diamond) \Rightarrow (\ps F f)_>$ be the $2$-cell induced by the image
of the bicomma square $\alpha \colon 1_B \dom \Rightarrow f \cod$. Note that it
is invertible if \hyperlink{hyp:BCC-functor}{BCC} holds. As a string diagram, it
is depicted
\[ \stringdiagram{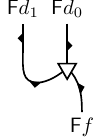} \]For the naturality axiom, recall that $\gamma_\diamond\colon \bicomma{1_B}{f}
\Rightarrow \bicomma{1_B}{g}$ is the essentially unique morphism of spans that
factors the square
\[\begin{tikzcd}[column sep=huge, row sep=large]
    A & {\bicomma{1_B}{f}} \\
    B & B \arrow[""{name=0, anchor=center, inner sep=0}, "g"{description},
    curve={height=18pt}, from=1-1, to=2-1]\arrow[""{name=1, anchor=center, inner sep=0}, "f"{description},
    curve={height=-18pt}, from=1-1, to=2-1]\arrow["\cod"', from=1-2, to=1-1]\arrow["\dom", from=1-2, to=2-2]\arrow["\alpha"', between={0.3}{0.7}, Rightarrow, from=2-2, to=1-1]\arrow[equals, from=2-2, to=2-1]\arrow["\gamma"', between={0.3}{0.7}, Rightarrow, from=1, to=0]\end{tikzcd}\]through the bicomma square $\beta$ of $g$:
\[\begin{tikzcd}[column sep=huge, row sep=large]
    && \bicomma{1_B}{g} \\
    A & {\bicomma{1_B}{f}} \\
    B & B \arrow[""{name=0, anchor=center, inner sep=0}, curve={height=12pt},
    from=1-3, to=2-1]\arrow["{\gamma_\diamond}"{description}, from=1-3, to=2-2]\arrow[""{name=1, anchor=center, inner sep=0}, curve={height=-18pt},
    from=1-3, to=3-2]\arrow["g"{description}, from=2-1, to=3-1]\arrow["\cod"{description}, from=2-2, to=2-1]\arrow["\dom"{description}, from=2-2, to=3-2]\arrow["\beta"', between={0.3}{0.7}, Rightarrow, from=3-2, to=2-1]\arrow[equals, from=3-2, to=3-1]\arrow["\zeta"'{pos=0.6}, between={0.2}{0.8}, Rightarrow, from=1, to=2-2]\arrow["\chi"'{pos=0.4}, between={0.2}{0.8}, Rightarrow, from=2-2, to=0]\end{tikzcd}\]Naturality follows:
\[ \stringdiagram{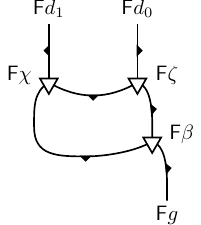} =
  \stringdiagram{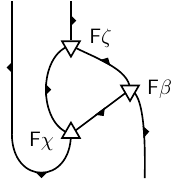} =
  \stringdiagram{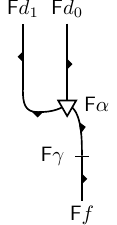} \]

Unitality is immediate: the unitor of $\u {\ps F}(-_\diamond)$ is equal to
$\delta_{1_A}$ (modulo the structural $2$-cells of $\K'$ and $\ps F$) by
construction.

For associativity, consider $f\colon A \to B$ and $g\colon B \to C$ in $\K$, the
bicomma squares
\[\begin{tikzcd}[cramped]
    {\bicomma{1_C}{gf}} & C \\
    A & C \arrow[from=1-1, to=1-2]\arrow[from=1-1, to=2-1]\arrow["\beta", between={0.3}{0.7}, Rightarrow, from=1-2, to=2-1]\arrow[equals, from=1-2, to=2-2]\arrow["gf"', from=2-1, to=2-2]\end{tikzcd} \qquad \begin{tikzcd}[column sep=tiny, row sep=scriptsize, cramped]
    && [-9pt] {\bicomma{\cod'}{\dom}} \\
    & {\bicomma{1_B}{f}} && [-9pt] {\bicomma{1_C}{g}} \\
    A && B && C \\
    & B & {=} & C \\
    && C \arrow[from=1-3, to=2-2]\arrow[from=1-3, to=2-4]\arrow["\cod"', from=2-2, to=3-1]\arrow["\dom", from=2-2, to=3-3]\arrow["{\,\alpha''}"', shift right=2, between={0.2}{0.8}, Rightarrow,
    from=2-4, to=2-2]\arrow["{\cod'}"', from=2-4, to=3-3]\arrow["{\dom'}", from=2-4, to=3-5]\arrow["f"', from=3-1, to=4-2]\arrow["\alpha"', between={0.3}{0.7}, Rightarrow, from=3-3, to=3-1]\arrow[equals, from=3-3, to=4-2]\arrow["g"{description}, from=3-3, to=4-4]\arrow["{\alpha'}"', between={0.3}{0.7}, Rightarrow, from=3-5, to=3-3]\arrow[equals, from=3-5, to=4-4]\arrow["g"', from=4-2, to=5-3]\arrow[equals, from=4-4, to=5-3]\end{tikzcd}\]and the essentially unique morphism of spans factoring this last square into the
bicomma $\bicomma{1_C}{gf}$:
\[\begin{tikzcd}[column sep=large, row sep=1pt]
    && {\bicomma{\cod'}{\dom}} \\
    & {\bicomma{1_B}{f}} && {\bicomma{1_C}{g}} \\
    A && {\bicomma{1_C}{gf}} && C \arrow[from=1-3, to=2-2]\arrow[from=1-3, to=2-4]\arrow["h"{description}, from=1-3, to=3-3]\arrow["\cod"', from=2-2, to=3-1]\arrow["\zeta"', between={0.3}{0.7}, Rightarrow, from=2-4, to=3-3]\arrow["{\dom'}", from=2-4, to=3-5]\arrow["\chi"', between={0.3}{0.7}, Rightarrow, from=3-3, to=2-2]\arrow["{\cod''}", from=3-3, to=3-1]\arrow["{\dom''}"', from=3-3, to=3-5]\end{tikzcd}\](with $\chi$ and $\zeta$ invertible).

The associator of $\u{\ps F}(-_\diamond)$ at $g$ and $f$ is
\[ \stringdiagram{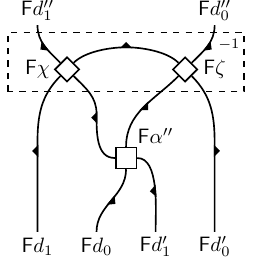} \]Associativity follows:
\begin{align*}
  \stringdiagram{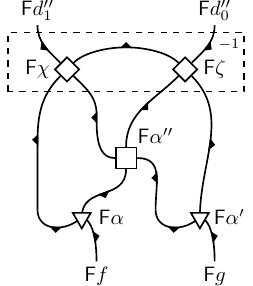}
  &= \stringdiagram{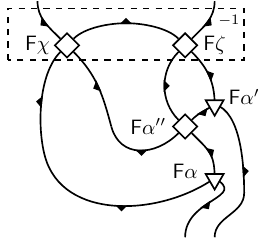} \\
  &= \stringdiagram{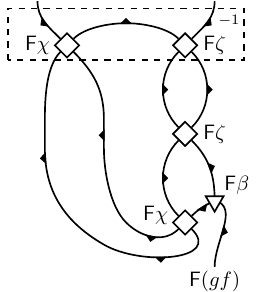} \\
  &= \stringdiagram{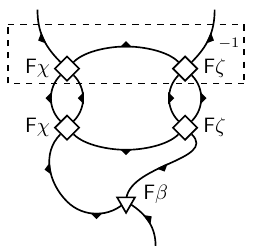} \\
  &= \stringdiagram{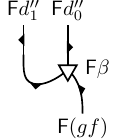}
\end{align*}

\subsubsection{Constructing extensions (transformations)}

Let $\ps F, \ps G\colon \K \rightrightarrows \Adj{\K'}$ be two pseudofunctors.
Suppose that $\ps F$ and $\ps G$ satisfy the condition of the previous
subsection, so that they have oplax extensions $\u{\ps F}, \u{\ps G}\colon
\DFib{\K}\to \K'$, realized by identity-on-objects oplax natural transformations
$\delta^{\ps F}\colon \underline{\ps F}(-)_\diamond \Rightarrow \ps F_>$ and
$\delta^{\ps G}\colon \underline{\ps G}(-)_\diamond \Rightarrow \ps G_>$.

Let now $\gamma\colon \ps F_> \Rightarrow \ps G_>$ be a pseudonatural
transformation. We construct an oplax natural transformation $\underline \gamma
\colon \underline{\ps F} \Rightarrow \underline{\ps G}$, such that $\gamma \cdot
\delta^{\ps F} = \delta^{\ps G} \cdot \underline \gamma (-)_\diamond$. On the
way, we also show that, when \hyperlink{hyp:BCC-functor}{BCC} holds for both
$\ps F$ and $\ps G$ (so that $\u{\ps F}$ and $\u{\ps G}$ are pseudofunctors),
$\u\gamma$ is pseudonatural as soon as
\begin{description}
\item[BCC'\hypertarget{hyp:BCC-transformation}] for every $f\colon A \to B$ in
  $\K$, $\gamma_f^\circlearrowright\colon \gamma_A (\ps F f)^< \Rightarrow (\ps
  G g)^< \gamma_B$ is invertible.
\end{description}

\bigskip \noindent \emph{Construction.} For an object $A$ of $\K$, we set the
component $\underline{\gamma}_A\colon \underline{\ps F}A \rightarrow
\underline{\ps G} A$ to simply be $\gamma_A\colon \ps F A \to \ps G A$. Then,
for $f = A \xleftarrow{q} E \xrightarrow{p} B$ in $\DFib{\K}$, we set the
$2$-cell $\u\gamma_f\colon \underline{\gamma}_B \circ \underline{\ps F} f
\Rightarrow \underline{\ps G} f \circ \underline{\gamma}_A$ to be the composite
\[ \gamma_B (\ps F p)_> (\ps F q)^< \cong (\ps G p)_> \gamma_E (\ps F q)^<
  \Rightarrow (\ps G p)_>(\ps G q)^< \gamma_A \]of the inverse of $\gamma_p$ with $\gamma_q^\circlearrowright$. Note that it is
invertible if \hyperlink{hyp:BCC-transformation}{BCC'} holds. In string
diagrams, we depict $\ps F$ in blue and $\ps G$ in red, so that $\u \gamma_f$ is
\[ \stringdiagram{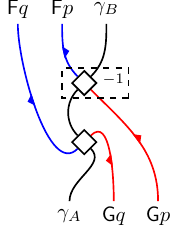} \]
Note that we do not label the $2$-cells in the string diagrams anymore as they
can be inferred from the context and the previous definitions.

Throughout the proofs below, we implicitly use the general fact that if $xy =
y'x'$, for any appropriately typed $2$-cells $x$, $y$, $x'$ and $y'$, then
\[ (y')^{-1}x = (y')^{-1}xyy^{-1} = (y')^{-1}y'x'y^{-1} = x'y^{-1} \] (when $y$
and $y'$ are invertible) and similarly
\[ y(x')^{-1} = x^{-1}y' \] (when $x$ and $x'$ are invertible). Hence when an
equation of the former shape holds we freely allow ourselves to also use the
corresponding equations of the latter shapes. We for instance do this in the
first step below in the proof of naturality of $\u\gamma$.
\[ \stringdiagram{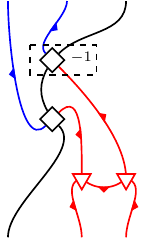} =
  \stringdiagram{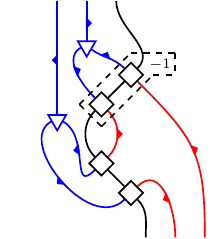} =
  \stringdiagram{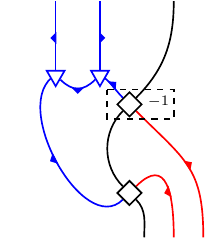} \]

\bigskip \noindent \emph{Coherence.}
Unitality follows from the naturality of $\gamma$:
\[ \stringdiagram{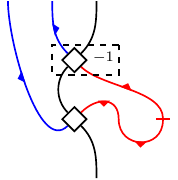} =
  \stringdiagram{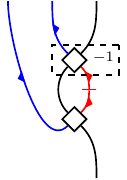} =
  \stringdiagram{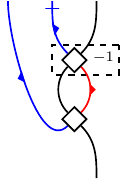} =
  \stringdiagram{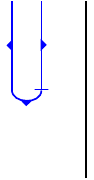} \]

For associativity, first notice that the definitions of $\ps F$ and $\ps G$ on
each $\DFib{\K}(A,B)$ do not use that the source and target spans are TSDFs, and
can hence be extended to $\Span{\K}(A,B)$. The proofs of naturality of
$\u\gamma$ above also works in this context, so that $\u\gamma$ commutes with
images of morphisms of spans. This fact gives us the first step in the proof of
associativity below; the second step is given by the lemma that follows.
\begin{align*}
  \stringdiagram{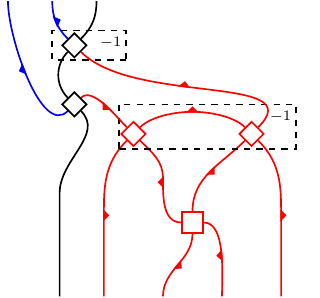}
  &= \stringdiagram{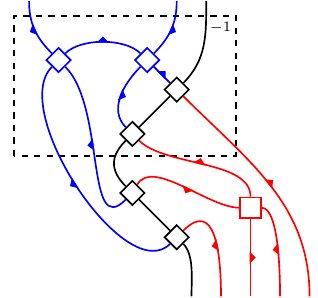} \\
  &= \stringdiagram{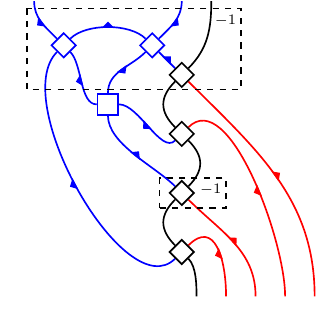}
\end{align*}

\begin{lemma}
  \[ \stringdiagram{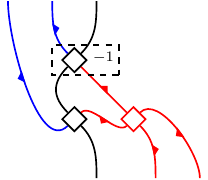} =
    \stringdiagram{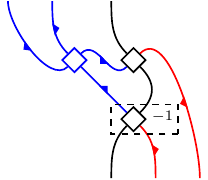} \]
\end{lemma}
\begin{proof}
  Because the naturality squares are invertible,
  \begin{align*}
    \stringdiagram{transformation-natural-rotated-1}
    &= \stringdiagram{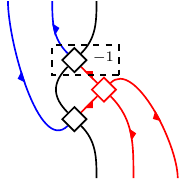} \\
    &= \stringdiagram{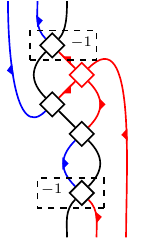} \\
    &= \stringdiagram{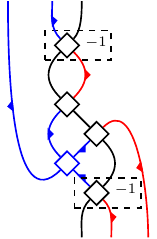} \\
    &= \stringdiagram{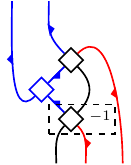} \\
    &= \stringdiagram{transformation-natural-rotated-6}
  \end{align*}
\end{proof}

\bigskip \noindent \emph{Extension.} We show that $\gamma \cdot \delta^{\ps F}
= \delta^{\ps G} \cdot \u \gamma (-)_\diamond$. This holds along objects because
$\delta^{\ps F}$ and $\delta^{\ps G}$ are the identities thereon, and $\u
\gamma_A = \gamma_A$. The naturality $2$-cells also coincide:
\[ \stringdiagram{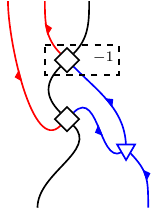} =
  \stringdiagram{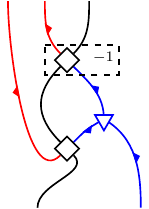} =
  \stringdiagram{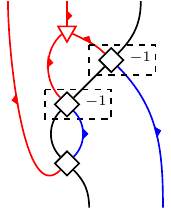} =
  \stringdiagram{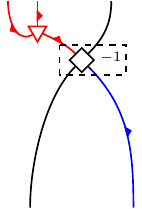} \]

\subsubsection{Constructing extensions (modifications)}
Consider now pseudofunctors $\ps F, \ps G\colon \K \to \Adj{\K'}$ and
pseudonatural transformations $\gamma, \gamma'\colon \ps F_> \Rightarrow \ps
G_>$ as above and their oplax extensions $\u\gamma,\u\gamma' \colon \u{\ps F}
\Rightarrow \u{\ps G}$. Let $\frak m\colon \gamma \Rrightarrow \gamma'$ be a
modification. We show that $\frak m_A \colon \gamma_A \Rightarrow \gamma_A'$
also defines a modification $\u{\frak m}\colon \u \gamma \Rrightarrow \u
\gamma'$.

Indeed, we have that
\[ \stringdiagram{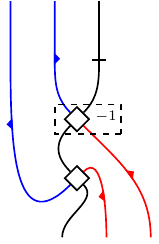} = \stringdiagram{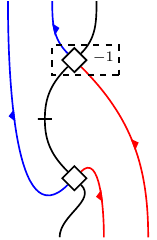} = \stringdiagram{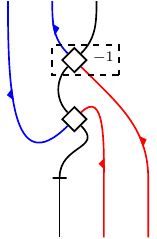}\]

\subsubsection{Uniqueness and converse implications (pseudofunctors)}
Consider $\u{\ps F}$ and $\delta_{\ps F}$ as in
\Cref{thm:extension-pseudofunctors}:
\[\begin{tikzcd}[column sep = 14pt]
    {\DFib{\K}} & {\DFib{\K'}} \\
    \K & {\K'}
    \arrow["{\u{\ps F}}", from=1-1, to=1-2]
    \arrow["{(-)_\diamond}", from=2-1, to=1-1]
    \arrow["{\ps F}"', from=2-1, to=2-2]
    \arrow["{\delta_{\ps F}}", between={0.3}{0.8}, Rightarrow, from=2-2, to=1-1]
    \arrow["{(-)_\diamond}"', from=2-2, to=1-2]
  \end{tikzcd} \]

In particular, this means that $\u{\ps F} f_\diamond \cong (\ps F f)_\diamond \cong \u {\ps F}' f_\diamond$ for each arrow $f$ in ${\K}$.
Because $\u{\ps  F}, \u{\ps F}'$ are pseudofunctors, they preserves adjunctions, and hence $\u{\ps F}
f^\diamond \cong (\ps F f)^\diamond \cong \u{\ps F}' f^\diamond$. Given now an arbitrary TSDF $ e = A
  \xleftarrow{q} E \xrightarrow{p} B $, by \cite[Prop. 4.25]{carboniModulatedBicategories1994} we know that $e \cong
p_\diamond q^\diamond$, and
hence $\u{\ps F}e \cong (\ps F p)_\diamond (\ps F q)^\diamond \cong \u{\ps F}'e$. This determines $\u{\ps F}$ uniquely, up to isomorphism, on arrows.

Because $(-)_\diamond$ is locally fully-faithul, the action of $\u{\ps F}$ is
determined on $2$-cells $f_\diamond \Rightarrow g_\diamond$ in $\DFib{\K}$, and
because $\u{\ps F}$ preserves adjunctions, it is also determined on
$2$-cells $f^\diamond \Rightarrow g^\diamond$. Note then that an arbitrary $2$-cell
in $\DFib{\K}$
\[\begin{tikzcd}[sep = 13pt]
	& E \\
	A && B \\
	& {E'}
	\arrow["q"', curve={height=12pt}, from=1-2, to=2-1]
	\arrow["p", curve={height=-12pt}, from=1-2, to=2-3]
	\arrow[""{name=0, anchor=center, inner sep=0}, "f"{description}, from=1-2, to=3-2]
	\arrow["{q'}", curve={height=-12pt}, from=3-2, to=2-1]
	\arrow["{p'}"', curve={height=12pt}, from=3-2, to=2-3]
	\arrow["\chi"', between={0.3}{0.7}, Rightarrow, from=0, to=2-1]
	\arrow["\zeta", between={0.3}{0.7}, Rightarrow, from=2-3, to=0]
\end{tikzcd}\]
can be rewritten as the composite:
\[\begin{tikzcd}[column sep=large]
	{p_\diamond q^\diamond} & {p'_\diamond f_\diamond f^\diamond q'^\diamond} & {p'_\diamond q'^\diamond.}
	\arrow["{\zeta_\diamond \chi^\diamond}", Rightarrow, from=1-1, to=1-2]
	\arrow["{ p'_\diamond \varepsilon_f q'^\diamond}", Rightarrow, from=1-2, to=1-3]
\end{tikzcd}\]
Thus, the action of $\u{\ps F}$ on this $2$-cell is also
determined, as $\u{\ps F}$ preserves (counits of) adjunctions, and hence $\ps F$ is entirely determined as a pseudofunctor.

Finally, we show that $\ps F$ must preserve exact squares. Given $\gamma\colon
xu \Rightarrow yv$, since by definition
\[\gamma^\dagger = x^\diamond y_\diamond \counit_v \cdot x^\diamond
  \gamma_\diamond v^\diamond \cdot \eta_x u_\diamond v^\diamond\]
then $\u{\ps F}\gamma^\dagger \cong (\ps F \gamma)^\dagger$
by the action of $\u{\ps F}$ on $2$-cells determined above. In particular, if
$\gamma^\dagger$ is invertible, then so is $(\ps F \gamma)^\dagger$, which concludes the proof.

\subsubsection{Uniqueness and converse implications (transformations)}

Consider now pseudonatural transformations as in
\Cref{thm:extension-psnat-trans}:
\[\begin{tikzcd}[cramped, column sep=9pt, row sep=10pt]
    {\DFib{\K}} && {\DFib{\K'}} \\
    \\
    \K && {\K'}
    \arrow["{\u{ \ps F}'}", curve={height=-18pt}, from=1-1, to=1-3]
    \arrow["{(-)_\diamond}", from=3-1, to=1-1]
    \arrow[""{name=0, anchor=center, inner sep=0}, "{\ps F'}"{description}, curve={height=-18pt}, from=3-1, to=3-3]
    \arrow[""{name=1, anchor=center, inner sep=0}, "{\ps F}"', curve={height=18pt}, from=3-1, to=3-3]
    \arrow[""{name=2, anchor=center, inner sep=0}, "{(-)_\diamond}"', from=3-3, to=1-3]
    \arrow["\gamma\,", between={0.2}{0.8}, Rightarrow, from=1, to=0]
    \arrow["{\delta_{\ps F'}}"'{pos=0.9}, between={0.5}{1}, Rightarrow, from=2, to=1-1]
  \end{tikzcd}
  =
  \begin{tikzcd}[cramped, column sep = 9pt, row sep=10pt]
    {\DFib{\K}} && {\DFib{\K'}} \\
    \\
    \K && {\K'}
    \arrow[""{name=0, anchor=center, inner sep=0}, "{\u{ \ps F}'}", curve={height=-18pt}, from=1-1, to=1-3]
    \arrow[""{name=1, anchor=center, inner sep=0}, "{\u{\ps F}}"{description}, curve={height=18pt}, from=1-1, to=1-3]
    \arrow[""{name=2, anchor=center, inner sep=0}, "{(-)_\diamond}", from=3-1, to=1-1]
    \arrow["{\ps F}"', curve={height=18pt}, from=3-1, to=3-3]
    \arrow["{(-)_\diamond}"', from=3-3, to=1-3]
    \arrow["{\u{\gamma}}\,", between={0.2}{0.8}, Rightarrow, from=1, to=0]
    \arrow["{\delta^{\ps F}}"{pos=0.4}, shift left=1, between={0.3}{0.6}, Rightarrow, from=3-3, to=2]
  \end{tikzcd}
\]

Because $\delta^{\ps F}$ is the identity on objects, we immediately get that $\u
\gamma_A = (\gamma_A)_\diamond$, and that $\u \gamma_{f_\diamond}$ is uniquely
determined. Moreover, using the coherence data of pseudonatural transformations,
$\u \gamma$ can be shown to appropriately commute with images by $\ps F$ and
$\ps F'$ of units and counits in $\K$.

As $(-)_\diamond$ is locally-fully-faithul, the correspondence $\chi
\mapsto \chi^\dagger$ is in fact a bijection between $2$-cells in $\K$ and
$2$-cells in $\DFib{\K}$ of the appropriate types: this is an instance of the
folklore \emph{central lemma} in a proarrow equipment --- see, e.g.~\cite[Lemma
1.2.3]{myersStringDiagramsDouble2018}. The two functions in this bijection can
be fully expressed in terms of units and counits in $\DFib{\K}$, as we have already
seen above for the mapping $\chi \mapsto \chi^\dagger$.

It follows from elementary computations using the calculus of adjunctions that,
up to invertible modification $\u \gamma_{f^\diamond}^{-1} = (\gamma_f)^\dagger$. This entails that
the commuting squares of $\gamma$ are indeed exact; moreover, $\gamma$ is
uniquely determined because, by the coherence data of pseudonatural
transformations, $\u \gamma_{p_\diamond q^{\diamond}}$ is, up to invertible modification,
the pasting of $\u \gamma_{p_\diamond}$ and $\u \gamma_{q^\diamond}$.

 \section{Deferred proofs from \texorpdfstring{\Cref{sec:04}}{Section IV}}\label{app:A4}

In this appendix we conclude the proof of \Cref{thm:loke-extend} and we provide more details on \Cref{cor:extension-on-cat}.

We begin by giving a complete definitions of the pseudomonads on $\CAT$ obtained by \emph{left oplax Kan extension} of relative (2-)monads, in the case of interest for this paper. Fix a monad $\monad T$ on $\Set$.
We define a 2-functor $\ps L\colon \CAT \to \CAT$ as follows. For a category $C$, we set $\ps LC$ to be the category having:
\begin{itemize}
    \item as objects, triples $(X,h,\nu)$ of a set $X$, a functor $h\colon X \to C$, and an element $\nu \in T (X)$;
    \item as morphisms $(X,h,\nu) \to (X',h',\nu')$, pairs of a function $g \colon X' \to X$ such that $Tg(\nu') = \nu$ and a natural transformation $\alpha \colon h \circ g \Rightarrow h'$.
\end{itemize}
For a functor $f \colon C \to D$, we set $\ps L f \colon \ps L C \to \ps L D$ to be the functor defined by:
\[\begin{tikzcd}
	{(X,h,\nu)} & {(X,fh,\nu)} \\
	{(X',h',\nu')} & {(X',fh',\nu')}
	\arrow[maps to, from=1-1, to=1-2]
	\arrow["{(g,\alpha)}"', from=1-1, to=2-1]
	\arrow["{(g, f*\alpha)}", from=1-2, to=2-2]
	\arrow[maps to, from=2-1, to=2-2]
\end{tikzcd}\]
For a natural transformation $\sigma \colon f \Rightarrow f'$ with $f,f' \colon C \to D$, we set $\ps L \sigma \colon \ps L f \Rightarrow \ps L f'$ to be the natural transformation defined at a triple $(X,h,\nu)$ in $\ps L C$ by the $\ps L D$-arrow:
\[ (\id, \sigma*h) \colon  (X,f h,\nu ) \to (X,f'h,\nu)\]

As shown in \cite{tarantinoUltracategoriesKanExtensions2025}, the 2-functor $\ps L$ carries the structure of a pseudomonad, which we now describe. Its unit $\eta^{\ps L} \colon \id \Rightarrow \ps L$ is given by the transformation defined at a category $C$ by the functor:
\[\begin{tikzcd}
	c & {(1, c, \eta^T_1(*))} \\
	{c'} & {(1, c', \eta^T_1(*))}
	\arrow[maps to, from=1-1, to=1-2]
	\arrow["t"', from=1-1, to=2-1]
	\arrow["{(\id_1, t)}", from=1-2, to=2-2]
	\arrow[maps to, from=2-1, to=2-2]
\end{tikzcd}\]
where we identify an object $c$ in $C$ with a functor $c \colon 1 \to C$, and an arrow $t \colon c \to c'$ with a natural transformation $t \colon c \Rightarrow c'$. Its multiplication $\mu^{\ps L}\colon \ps L^2 \Rightarrow \ps L$ is given by the transformation defined at a category $C$ by the following functor.
\begin{itemize}
    \item Consider an object $(X,H,\nu)$ in $\ps L^2 C$. For $x\in X$, denote $H(x) = (X_x, h_x, \nu_x)$, and let $S$ be the direct sum $\sum_{x\in X}X_x$ and $s \colon S \to C$ the functor $\sum_{x\in X}h_x$. Denote by $i_{x}$ each inclusion $X_x \hookrightarrow S$; let $q\colon X \to T ( S )$ be the function defined by $q(x) \coloneqq T i_x ( \nu_x )$, and let $Q \colon T (X) \to T ( S )$ be the function $\mu^T_{S} \circ T q$.  Then, the functor $\mu^{\ps L}_C$ maps $(X,H,\nu)$ to the triple $(S, s, Q(\nu))$.
    \item Consider an arrow $(g,\alpha) \colon (X,H,\nu)\to (X',H',\nu')$ in $\ps L^2 C$. For $x\in X'$, denote $\alpha_x \colon (X_{g(x)}, h_{g(x)}, \nu_{g(x)}) \to (X_x', h'_x, \nu'_x)$ by $(a_x, \gamma_x)$. The functions $\set{a_x}$ determine a function $a \colon S' \to S$ which can be showed to satisfy $Ta (Q'(\nu')) = Q(\nu)$, and the transformations $\set{ \gamma_x }$ determine a transformation $\gamma \colon s \circ a \Rightarrow s'$: then, the functor $\mu^{\ps L}_C$ maps $(g,\alpha)$ to the pair $(a, \gamma) \colon (S,s,Q(\nu)) \to (S',s',Q'(\nu'))$.
\end{itemize}
Note that, up to fixing canonical choices for coproducts in $\Set$, both $\eta^{\ps L}$ and $\mu^{\ps L}$ can be assumed to be $2$-natural --- that is, such that that their naturality squares commute on the nose.

To complete the proof of \Cref{thm:loke-extend}, we now show that the unit and
multiplication of $\ps L$ satisfy the Beck-Chevalley condition.

\begin{proposition}\label{prop:lokes-allow-for-lax-algebras-2}
	The unit $\eta^{{\ps L}} \colon 1_{\CAT} \Rightarrow {\ps L}$ satisfies the Beck-Chevalley condition.
\end{proposition}
\begin{proof}

	Fix a functor $f \colon C \to D$ and consider the square:
	\[\begin{tikzcd}
		C & {\ps L C} \\
		D & {\ps L D}
		\arrow["{\eta_C^{\ps L}}", from=1-1, to=1-2]
		\arrow["f"', from=1-1, to=2-1]
		\arrow["{\ps L f}", from=1-2, to=2-2]
		\arrow["{\eta_D^{\ps L}}"', from=2-1, to=2-2]
	\end{tikzcd}\]
To show that its $\ps L$-image is exact we need to show that, for each $d \in D$ and each $(X,h,\nu) \in \ps L C$, the set $\ps L D (  (X, fh, \nu), \eta^{\ps L}_D(d) )$ is the colimit of the diagram:
\[(*) \quad \begin{tikzcd}
	{ \op {\left( \int D(f(-) , d)\right)} } &[-10pt] C &[60pt] \Set
	\arrow[from=1-1, to=1-2]
	\arrow["{\ps L C ( (X,h,\nu), \eta_C^{\ps L}(-))}", from=1-2, to=1-3]
\end{tikzcd}\]
where $\int D(f(-) , d)$ is the category of elements of the presheaf $ D(f(-) , d) \colon \op C \to \Set$, together with its projection functor to $\op C$.

Note first that a morphism $(X, h, \nu) \to \eta_C^{\ps L}(c)$ in $\ps L C$ is simply given by a pair $(x, \psi)$ of an element $x \in X$ such that $\nu = \eta_X^T(x)$, and a $C$-arrow $\psi \colon h(x) \to c$. Therefore, the set  $\ps L D ( (X,fh,\nu), \eta_D^{\ps L}(d))$ is a cocone on the diagram $(*)$ where, for each $(c,\phi)$ in $\int D(f(-) , d)$, the coprojection
\[     \iota_{(c,\phi)} \colon \ps L C ( (X,h,\nu), \eta_C^{\ps L}(c)) \to \ps L D ( (X,fh,\nu), \eta_D^{\ps L}(d)) \]
is defined by $(x, \psi) \mapsto (x, \phi \circ f(\psi))$.

Suppose $S$ is another cocone on the diagram $(*)$, determined by a family of maps $\set{ \sigma_{(c,\phi)}\colon \ps L C ( (X,h,\nu), \eta_C^{\ps L}(c)) \to S}$. It is then immediate to see that the map
\[ \Omega \colon \ps L C ( (X,fh,\nu), \eta_D^{\ps L}(d)) \to S\]
defined by
\[ (x \in X, \chi \colon fh(x) \to d) \mapsto \sigma_{(h(x), \chi)}(x, \id_{h(x)})\]
is the unique such making all relevant triangles commute since $(x,\chi) = \iota_{(h(x),\chi)}(x,\id_{h(x)})$, thus exhibiting $\ps L D ((X, fh, \nu), \eta_D^{\ps L}(d))$ as the colimit of $(*)$.
\end{proof}

\begin{proposition}\label{prop:lokes-allow-for-lax-algebras-3}
	The multiplication $\mu^{\ps L}\colon \ps L^2 \Rightarrow\ps L$ satisfies the Beck-Chevalley condition.
\end{proposition}
\begin{proof}

Fix a functor $f \colon C \to D$ and consider the square:
\[\begin{tikzcd}
	{\ps L^2C} & {\ps L C} \\
	{\ps L^2 D} & {\ps L D}
	\arrow["{\mu_C^{\ps L}}", from=1-1, to=1-2]
	\arrow["{\ps L^2 f}"', from=1-1, to=2-1]
	\arrow["{\ps L f}", from=1-2, to=2-2]
	\arrow["{\mu_D^{\ps L}}"', from=2-1, to=2-2]
\end{tikzcd}\]
	As above, to show that its $\ps L$-image is exact we need to show that, for each $(Y,K,\xi)$ in $\ps L^2 D$ and each $(Z,l,\chi)$ in $\ps L C$, the set $\ps L D ( (Z, f l,\chi), \mu_D^{\ps L} (Y, K , \xi))$ is the colimit of the diagram:
	{\[ (*) \, \begin{tikzcd}
	{ \op {\left( \int \ps L^2 D(\ps L^2 f (-) , (Y, K, \xi))\right)} } &[-8pt] \ps L^2 C &[60pt] \Set
	\arrow[from=1-1, to=1-2]
	\arrow["{\ps L C ( (Z, l, \chi), \mu_C^{\ps L}(-))}", from=1-2, to=1-3]
	\end{tikzcd}\]}

    Throughout, we write $H(x) = (X_x, h_x, \nu_x)$ and $\mu_C^{\ps L}(X,H,\nu) = (\sum_{x\in X}X_x, \sum_{x\in X}h_x, \bar\nu)$, and we denote by $i^H_{x_0}$ the inclusion $X_{x_0} \to \sum_{x\in X}X_x$ for some $x_0 \in X$; we use similar notations for $(Y,K,\xi)$.
    Note that $\ps L D ( (Z, f l,\chi), \mu_D^{\ps L} (Y, K , \xi))$ is a cocone on the diagram $(*)$ where, for each $(X, H , \nu)$ in $\ps L^2 C$ and each $(p, \alpha) \colon (X, \ps L f \circ H, \nu) \to (Y, K ,\xi)$ in $\ps L^2 D$, the corresponding coprojection
	\[ \iota \colon \ps L C ( (Z, l, \chi), \mu_C^{\ps L}(X, H, \nu) ) \to \ps L D ( (Z, f l,\chi), \mu_D^{\ps L} (Y, K, \xi)) \]
    is defined by $(q,\gamma) \mapsto \mu_D^{\ps L} (p,\alpha)\circ \ps L f (q,\gamma)$. Graphically, this means that $\iota_{(X,h,\nu),(p,\alpha)}$ maps $(q,\gamma)$ to the morphism described by the pasting:
    \[ \begin{tikzcd}[column sep=24pt, row sep = 12pt]
	Z & {\sum_{x}X_x} &[20pt] {\sum_y Y_y} \\
	C \\
	& D
	\arrow["l"', from=1-1, to=2-1]
	\arrow["q"', from=1-2, to=1-1]
	\arrow[""{name=0, anchor=center, inner sep=0}, "{\sum_x h_x}", from=1-2, to=2-1]
	\arrow["{\sum_y i^H_{p(y)} a_y}"', from=1-3, to=1-2]
	\arrow[""{name=1, anchor=center, inner sep=0}, "{{\sum_yk_y}}", curve={height=-12pt}, from=1-3, to=3-2]
	\arrow["f"', curve={height=6pt}, from=2-1, to=3-2]
	\arrow["\gamma", between={0.2}{0.8}, Rightarrow, from=1-1, to=0]
	\arrow["{\sum_y\tah\alpha_y}", between={0.4}{0.8}, Rightarrow, from=0, to=1]
\end{tikzcd}\]
where we denote by $(a_y, \tah\alpha_y) \colon (X_{p(y)}, f h_{p(y)}, \nu_{p(y)} ) \to (Y_y,k_y,\xi_y)$ the component of $\alpha \colon \ps L f \circ H \circ p \Rightarrow K$ at $y\in Y$. 

Suppose $S$ is another cocone on the diagram $(*)$, determined by a family of maps $\set{ \sigma_{(X,H,\nu),(p,\alpha)} \colon \ps L C ( (Z, l, \chi), \mu_C^{\ps L}(X, H, \nu) )  \to S}$. We define a map
\[\Omega \colon \ps L D ( (Z, f l,\chi), \mu_D^{\ps L} (Y, K , \xi)) \to S\]
as follows: let $(r, \rho) \colon (Z, fl, \chi) \to \mu_D^{\ps L}(Y,K,\xi)$ in $\ps L D$. Note that:
	\[\begin{tikzcd}[column sep=6pt, row sep = 12pt]
	Z && {\sum_{y}Y_y} \\
	C \\
	& D
	\arrow[""{name=0, anchor=center, inner sep=0}, "l"', from=1-1, to=2-1]
	\arrow["r"', from=1-3, to=1-1]
	\arrow[""{name=1, anchor=center, inner sep=0}, "{\sum_y k_y}", curve={height=-12pt}, from=1-3, to=3-2]
	\arrow["f"', curve={height=6pt}, from=2-1, to=3-2]
	\arrow["\rho", between={0.2}{0.8}, Rightarrow, from=0, to=1]
\end{tikzcd} = \begin{tikzcd}[column sep=22pt, row sep = 12pt]
	Z &[11pt] {\sum_{y} Z} &[15pt] {\sum_{y}Y_y} \\
	C \\
	& D
	\arrow["l"', from=1-1, to=2-1]
	\arrow["{\sum_{y}\id_Z}"', from=1-2, to=1-1]
	\arrow[""{name=0, anchor=center, inner sep=0}, "{\sum_{y}l}", from=1-2, to=2-1]
	\arrow["{ \sum_y i^L_y r i_y^k }"', from=1-3, to=1-2]
	\arrow[""{name=1, anchor=center, inner sep=0}, "{\sum_{y}k_y}", curve={height=-12pt}, from=1-3, to=3-2]
	\arrow["f"', curve={height=6pt}, from=2-1, to=3-2]
	\arrow["{=}"{marking, allow upside down}, draw=none, from=1-1, to=0]
	\arrow["{\sum_{y}\rho*i_y^k}", between={0.35}{0.8}, Rightarrow, from=0, to=1]
	\end{tikzcd}\]
    where $L\colon Y \to \ps L C$ is given by $L(y) \coloneqq (Z, l, \chi)$, so that $i^L_y$ is the inclusion $Z \to \sum_{y\in Y}Z$. Consider the object $(Y,L,\xi)$ in $\ps L^2 C$:
    \begin{itemize}
    \item the pair $(\id_Y, \hat \rho)$, where $\hat\rho \colon \ps L f \circ L \Rightarrow K$ is the transformation defined at $y\in Y$ by the $\ps L D$-arrow $(ri_y^k, \rho*i_y^k) \colon (Z, fl, \chi) \to (Y_y, k_y, \xi_y)$, defines an $\ps L ^2 D$-arrow $(Y, \ps L f\circ L, \xi) \to (Y, K, \xi)$;
    \item the pair $(\sum_y \id_Z, \id)$ defines an $\ps L C$-arrow $(Z,l,\chi) \to \mu_C^{\ps L}(Y,L,\xi)$.
    \end{itemize}
Thus, by the above description of the maps $\set{\iota}$, the previous diagram expresses that $(r,\rho)$ coincides with $\iota_{(Y, L,\xi),(\id_Y, \hat \rho)}(\sum_y \id_Z, \id)$, so that we can define:
	\[\textstyle \Omega ( r, \rho ) \coloneqq \sigma_{(Y, L,\xi),(\id_Y, \hat \rho)}(\sum_y \id_Z, \id)\]
It is then immediate to see that $\Omega$ is the unique map making all relevant triangles commute, thus exhibiting $\ps L D ( (Z, f l,\chi), \mu_D^{\ps L} (Y, K , \xi))$ as the colimit of $(*)$.
\end{proof}

\begin{scholium}
	Both proofs above carry over verbatim to the quotient pseudomonad $\bb T$ for $T \in \set{ \mathcal F, \beta, \mathcal P,\mathcal D}$, hence proving also \Cref{cor:quotients-extend} since clearly these quotients also preserve fully-faithfulness.
\end{scholium}

Moving on to \Cref{cor:extension-on-cat}, we first recall the definition of the \emph{Kleisli bicategory} of a pseudomonad $\monad {\ps T}$ on a bicategory $\K$.

\begin{definition}
    The \emph{Kleisli bicategory} of $\ps T$ is the bicategory $\Kl {\ps T}$ defined as follows:
    \begin{itemize}
        \item objects of $\Kl{\ps T}$ are objects of $\K$;
        \item the hom-category $\Kl{\ps T}(A,B)$ is given by $\K(A, \ps T B)$;
        \item the identity arrow at an object $A$ is given by the unit $\eta^{\ps T}_A \colon A \to \ps T A$;
        \item the composite of $f \colon A \to \ps T B$ and $g \colon B \to \ps T C$ in $\Kl{\ps T}$ is given by the composite
\small
\[\begin{tikzcd}
	A & {\ps T B} & {\ps T ^2C} & {\ps T C}
	\arrow["f", from=1-1, to=1-2]
	\arrow["{\ps T g}", from=1-2, to=1-3]
	\arrow["{\mu^{\ps T}_C}", from=1-3, to=1-4]
\end{tikzcd}\]
\normalsize
in $\K$, which extends to a functor $\K(B,\ps TC) \times \K(A,\ps TB)\to \K(A, \ps T C)$.
\end{itemize}

\looseness=-1
Note that $\Kl{\ps T}$ is generally not a $2$-category even if $\K$ is one, as it inherits its structural $2$-cells from the pseudomonad structure of $\ps T$.

Note also that $\Kl{\ps T}$ comes equipped with a canonical pseudofunctor $\ps J_{\ps T} \colon \K \to \Kl{\ps T}$ acting as the identity on objects and composing with $\eta^{\ps T}$ on the hom-categories.
\end{definition}

\looseness=-1 Towards addressing \Cref{cor:extension-on-cat}, consider a
pseudomonad $\monad {\ps T}$ on $\CAT$. Identifying $\coop{\DFib{\coop\CAT}}$
with $\PROF$, by \Cref{cor:extension-pseudomonads} we know
that $\ps T$ extends to a pseudomonad on $\PROF$ if and only if $\ps T$, $\eta$
and $\mu$ all satisfy the appropriate Beck-Chevalley condition. However, through
the biequivalence $\PROF \simeq \Kl\psh$ --- along which $\ps J_{\psh} \colon
\CAT \to \PROF$ can be identified with $(-)_\diamond$ --- we can see this
extension from the perspective of the extensions to Kleisli bicategories
discussed in \cite{chengPseudodistributiveLaws2003}.

Intuitively, the idea is that while $\ps T$, $\eta$ and $\mu$ can be extended \emph{separately} to constitute a pseudomonad, there should be some further compatibility with respect to the monad structure of $\monad{\ps T}$. More formally, for a pseudomonad $\monad{\ps S}$ on a bicategory $\K$, we say that a pseudomonad $\monad{\ps T}$ on $\K$ \emph{extends to $\Kl{\ps S}$} if there is a pseudomonad $\monad{\u{\ps T}}$ on $\Kl{\ps S}$ together with a pseudonatural equivalence $\zeta \colon  \ps J_{\ps S} \ps T \cong \u{\ps T} \ps J_{\ps S}$ and two invertible modifications
\[\begin{tikzcd}[column sep = 15pt, row sep = 20pt,cramped]
	{\ps J_{\ps S}} && {\u{\ps T} \ps J_{\ps S}} \\
	& {\ps J_{\ps S} \ps T}
	\arrow[""{name=0, anchor=center, inner sep=0}, "{\eta^{\u{\ps T}}*\ps J_{\ps S}}", from=1-1, to=1-3]
	\arrow["{\ps J_{\ps S} * \eta^{\ps T}}"', from=1-1, to=2-2]
	\arrow["\zeta"', from=2-2, to=1-3]
	\arrow["{{{\mathfrak p}}}", between={0.2}{0.8}, Rightarrow, from=0, to=2-2]
\end{tikzcd} \quad \begin{tikzcd}[row sep = 24pt,cramped]
	{\ps J_{\ps S}{\ps {T}}^2} & {\ps J_{\ps S} \ps T} \\
	{{\u{\ps T}}\ps J_{\ps S}\ps T} \\
	{{\u{\ps {T}}}^2 \ps J_{\ps S}} & {\u{\ps T} \ps J_{\ps S}}
	\arrow["{\ps J_{\ps S} * \mu^{\ps T}}", from=1-1, to=1-2]
	\arrow["{{{\zeta *\ps T}}}"', from=1-1, to=2-1]
	\arrow["{\mathfrak q}"', between={0.2}{0.8}, Rightarrow, from=1-2, to=3-1]
	\arrow["\zeta", from=1-2, to=3-2]
	\arrow["{\u{\ps T} *\zeta }"', from=2-1, to=3-1]
	\arrow["{\mu^{\u{\ps T}}*\ps J_{\ps S}}"', from=3-1, to=3-2]
\end{tikzcd}\]
satisfying coherence axioms analogous to those of a pseudoalgebra for a pseudomonad.

To prove \Cref{cor:extension-on-cat}, we now see that these modifications trivialize if $\monad{\ps T}$ extends to $\PROF$ as in \Cref{cor:extension-pseudomonads}.

\begin{corollary}
  For a pseudomonad $\monad{\ps T}$ on $\CAT$, the following are equivalent:
  \begin{enumerate}
  \item $\ps T$, $\eta^{\ps T}$ and $\mu^{\ps T}$ satisfy the Beck-Chevalley condition;
  \item $\monad{\ps T}$ extends to a pseudomonad on $\PROF$;
  \item $\monad{\ps T}$ pseudodistributes over $\psh$.
\end{enumerate}
\begin{proof}
The equivalence between (2) and (3) is proved in \cite[Thm.\ 4.3]{chengPseudodistributiveLaws2003}. That (2) implies (1) follows by \Cref{thm:extension-pseudofunctors}.(3), \Cref{thm:extension-psnat-trans}.(2) and \Cref{thm:extension-modifications}. To show that (1) implies (2) and conclude the proof, we then need to show that the pseudomonad $\monad{\u{\ps T}}$ on $\PROF$ extending $\monad{\ps T}$ according to \Cref{cor:extension-pseudomonads} satisfies the further requirements of an extension to the Kleisli bicategory $\Kl{\psh}$ described above.

First, set $\zeta \colon (\ps T -)_\diamond \cong \ps T(-)_\diamond$ to be the canonical pseudonatural equivalence $\delta^{\ps T}$ witnessing the extension in the sense of \Cref{thm:extension-pseudofunctors}. Recall also that, by construction in \Cref{cor:extension-pseudomonads}, the unit $\eta^{\u{\ps T}}$ and multiplication $\mu^{\u{\ps T}}$ are given by $\u{\eta^{\ps T}} \circ \omega^1$ and $\u{\mu^{\ps T}}\circ \omega^{\ps T}$ where $\u{\eta^{\ps T}}$ and $\u{\mu^{\ps T}}$ are the extensions of $\eta^{\ps T}$ and $\mu^{\ps T}$ given by \Cref{thm:extension-psnat-trans} while $\omega^1$ and $\omega^{\ps T}$ have identity components. Unraveling definitions, we see that the component of the modification $\mathfrak p$ at a category $C$ is a natural transformation of profunctors
\[\begin{tikzcd}[sep = small, cramped]
	C && {\ps T C} \\
	& {\ps F C}
	\arrow[""{name=0, anchor=center, inner sep=0}, "{{{(\eta^{\ps T}_C)_\diamond}}}"{inner sep=.8ex}, "\shortmid"{marking}, from=1-1, to=1-3]
	\arrow["{{{(\eta^{\ps T}_C)_\diamond}}}"'{inner sep=.8ex}, "\shortmid"{marking}, from=1-1, to=2-2]
	\arrow[equals, from=2-2, to=1-3]
	\arrow[between={0.2}{1}, Rightarrow, from=0, to=2-2]
\end{tikzcd}\]
which we can take to be the identity. Similarly, the component of $\mathfrak q$ at $C$ is a natural transformation of profunctors
\[\begin{tikzcd}[row sep = 12pt, column sep = 20pt, cramped]
	{\ps T^2 C} & {\ps T C} \\
	{\ps T^2 C} \\
	{\ps T^2 C} & {\ps T C}
	\arrow["{{(\mu^{\ps T}_C)_\diamond}}"{inner sep=.8ex}, "\shortmid"{marking}, from=1-1, to=1-2]
	\arrow[equals, from=1-1, to=2-1]
	\arrow[between={0.2}{0.8}, Rightarrow, from=1-2, to=3-1]
	\arrow[equals, from=1-2, to=3-2]
	\arrow[equals, from=2-1, to=3-1]
	\arrow["{{(\mu^{\ps T}_C)_\diamond}}"'{inner sep=.8ex}, "\shortmid"{marking}, from=3-1, to=3-2]
\end{tikzcd}\]
which we can also take to be the identity. This concludes the proof, as coherence is then trivial.

\end{proof}
\end{corollary}
 \section{Deferred proofs from \texorpdfstring{\Cref{sec:05}}{Section V}}\label{app:A5}

We here provide more details on the definitions of \Cref{sec:05} and we give a proof of \Cref{prop:lokes-allow-for-lax-algebras}. 

First, we begin by justifying \Cref{def:lax-algebras-allowed} by showing that, for a left skew monad allowing for lax algebras, \emph{free algebras} are indeed lax algebras.

\begin{proposition}
	Let $\monad{\u{\ps T}}$ be a left skew monad on $\PROF$ allowing for lax algebras. Then, for any category $C$, the category $\ps T C$ carries the structure of a lax algebra defined by the profunctor $(\mu^{\ps T}_C)_\diamond \colon \ps T ^2 C \pro \ps T C$.
\end{proposition}
\begin{proof}
	First recall by \Cref{rem:lax-algebras-allowed-identifying-delta} that we can identify, for a functor $f \colon C \to D$:
	\begin{itemize}
	\item $\u{\ps T}(f_\diamond)$ with $(\ps T f)_\diamond$, and
	\item the components of $\eta^{\u {\ps T}}$ and $\mu^{\u{\ps T}}$ at $f_\diamond$ with $(\eta^{\ps T}_f)_\diamond$ and $(\mu^{\ps T}_f)_\diamond$, respectively.
	\end{itemize}

	Consider now the category $\ps T C$ with the profunctor $(\mu^{\ps T}_C)_\diamond \colon \ps T ^2 C \pro \ps T C$.
	As a unitor $1_{\ps T C} \Rightarrow (\mu^{\ps T}_C)_\diamond \circ (\eta^{\ps T}_{\ps T C})_\diamond$ we take $(\mathfrak l^{\u {\ps T}}_C)^{-1}$, i.e.\ the natural transformation $(\mathfrak l^{\ps T}_C)^{-1}_\diamond$. As a multiplicator $(\mu^{\ps T}_C)_\diamond \circ {\ps T} (\mu^{\ps T}_C)_\diamond \Rightarrow (\mu^{\ps T}_C)_\diamond \circ (\mu^{\ps T}_{\ps T C})_\diamond$ we take $(\mathfrak{m}^{\u {\ps T}}_C)^{-1}$, i.e.\ the natural transformation $(\mathfrak{m}^{{\ps T}}_C)_\diamond^{-1}$. With these definitions, it is straightforward to see that the three axioms of \Cref{def:lax-algebras-allowed} follow respectively from axiom (3), axiom (2), and axiom (1) in \Cref{def:skew-monad}, thus making the tuple a lax $\u{\ps T}$-algebra.

\end{proof}

We now proceed to justify \Cref{def:colax-morphisms-lax-algebras-allowed} by showing that lax algebras determine a $2$-category whose arrows are representable colax morphisms.

\begin{proposition}\label{prop:2cat-lax-algebras-allowed}
	Let $\monad{\u{\ps T}}$ be a left skew monad on $\PROF$ allowing for lax algebras. There is a $2$-category of lax $\u{\ps T}$-algebras, representable colax morphisms, and algebra $2$-cells.
\end{proposition}
\begin{proof}
	We here describe how representable colax morphisms are composed; that this yields a $2$-category is then tedious but straightforward.

	Consider two representable colax morphisms $\braket{f, \Theta} \colon \braket{A_0, A, \Gamma,\Delta} \to \braket{A_0', A', \Gamma',\Delta'}$ and $\braket{f',\Theta'} \colon \braket{A_0', A', \Gamma',\Delta'} \braket{A_0'', A'', \Gamma'',\Delta''}$. We define their composite as the pair of the functor $f' f \colon A_0 \to A_0''$ and the natural transformation obtained as the pasting:
\[\begin{tikzcd}
	{\ps T A_0} & {\ps TA_0'} & {\ps TA_0''} \\
	{A_0} & {A_0'} & {A_0''}
	\arrow["{\u{\ps T} f_\diamond}"', from=1-1, to=1-2]
	\arrow[""{name=0, anchor=center, inner sep=0}, "{\u{\ps T}( f'f)_\diamond}",bend left = 45, from=1-1, to=1-3]
	\arrow["A"', from=1-1, to=2-1]
	\arrow["{\u{\ps T} f_\diamond'}"', from=1-2, to=1-3]
	\arrow["{A'}"', from=1-2, to=2-2]
	\arrow["{A''}", from=1-3, to=2-3]
	\arrow["\Theta"', between={0.3}{0.7}, Rightarrow, from=2-1, to=1-2]
	\arrow["{f_\diamond}"', from=2-1, to=2-2]
	\arrow["{\Theta'}"', between={0.3}{0.7}, Rightarrow, from=2-2, to=1-3]
	\arrow["{f'_\diamond}"', from=2-2, to=2-3]
	\arrow["{\psi_{f'_\diamond, f_\diamond}}", between={0}{0.8}, Rightarrow, from=1-2, to=0, pos=0.4]
\end{tikzcd}\]
where we identify $f'_\diamond \circ f_\diamond \cong (f'f)_\diamond$ by pseudofunctoriality of $(-)_\diamond$. The fact that the composite satisfies the axioms of \Cref{def:colax-morphisms-lax-algebras-allowed} can be verified by diagram chasing. For instance, axiom (2) corresponds to the commutativity of the outer border in the diagram:
\[\hspace{-4em}\begin{tikzcd}[row sep = small, column sep = 0pt]
	{(f'f)_\diamond \ A \ \u{\ps T}A } &&&& {(f'f)_\diamond \ A \ (\mu^{\ps T}_{A_0})_\diamond} \\
	{f'_\diamond \ A' \ \u{\ps T}f_\diamond \ \u{\ps T}A} & {f'_\diamond \ A' \ \u{\ps T}(f_\diamond \ A)} & {f_\diamond' \ A' \ \u{\ps T}(A' \ \u{\ps T}f_\diamond )} \\
	{A'' \ \u{\ps T}f_\diamond' \ \u{\ps T}f_\diamond \ \u{\ps T}A} & {A'' \ \u{\ps T}f_\diamond' \ \u{\ps T}(f_\diamond \ A)} && {f_\diamond' \ A' \ (\mu^{\ps T}_{A_0'})_\diamond \ \u{\ps T}^2f_\diamond } \\
	{A'' \ \u{\ps T}(f'f)_\diamond \ \u{\ps T}A} && {f_\diamond' \ A' \ \u{\ps T}A' \ \u{\ps T}^2f_\diamond } && {f'_\diamond \ A' \ \u{\ps T}f_\diamond \ (\mu^{\ps T}_{A_0})_\diamond} \\
	{A'' \ \u{\ps T}(f'_\diamond \ f_\diamond \ A)} & {A'' \ \u{\ps T}f_\diamond' \ \u{\ps T}(A' \ \u{\ps T}f_\diamond )} & {A'' \ \u{\ps T}f_\diamond' \ \u{\ps T}A' \ \u{\ps T}^2f_\diamond )} \\
	&& {A'' \ \u{\ps T}(f'_\diamond \  A') \ \u{\ps T}^2 f_\diamond)} & {A'' \ \u{\ps T}f'_\diamond \ (\mu^{\ps T}_{A_0'})_\diamond \ \u{\ps T}^2f_\diamond } \\
	{A'' \ \u{\ps T}(f'_\diamond \  A' \ \u{\ps T} f_\diamond)} && {A'' \ \u{\ps T}(A'' \ \u{\ps T}f'_\diamond) \ \u{\ps T}^2 f_\diamond)} && {A'' \ \u{\ps T}f'_\diamond \  \u{\ps T}f_\diamond \ (\mu^{\ps T}_{A_0})_\diamond} \\
	{A'' \ \u{\ps T}(A'' \ \u{\ps T}f'_\diamond \ \u{\ps T} f_\diamond)} && {A'' \ \u{\ps T}A'' \ \u{\ps T}^2f'_\diamond \ \u{\ps T}^2 f_\diamond)} \\
	{A'' \ \u{\ps T}(A'' \ \u{\ps T}(f'f)_\diamond )} & {A'' \ \u{\ps T}A'' \ \u{\ps T}(\u{\ps T}f'_\diamond \ \u{\ps T} f_\diamond)} && {A'' \ (\mu^{\ps T}_{A_0''})_\diamond \ \u{\ps T}^2f'_\diamond \ \u{\ps T}^2 f_\diamond)} \\
	&& {A'' \ (\mu^{\ps T}_{A_0''})_\diamond \ \u{\ps T}(\u{\ps T}f'_\diamond \ \u{\ps T} f_\diamond)} \\
	{A'' \ \u{\ps T}A'' \ \u{\ps T}^2(f'f)_\diamond )} && {A'' \ (\mu^{\ps T}_{A_0''})_\diamond \ \u{\ps T}^2(f' f)_\diamond} && {A'' \ \u{\ps T}(f'f)_\diamond \ (\mu^{\ps T}_{A_0})_\diamond}
	\arrow["\Delta", from=1-1, to=1-5]
	\arrow["\Theta"', from=1-1, to=2-1]
	\arrow["\Theta", from=1-5, to=4-5]
	\arrow["\psi", from=2-1, to=2-2]
	\arrow["{{\Theta'}}"', from=2-1, to=3-1]
	\arrow["{{\u{\ps T}\Theta}}", from=2-2, to=2-3]
	\arrow["{{\Theta'}}"', from=2-2, to=3-2]
	\arrow["\phi", from=2-3, to=4-3]
	\arrow["{\Theta'}", from=2-3, to=5-2]
	\arrow["\psi", from=3-1, to=3-2]
	\arrow["\psi"', from=3-1, to=4-1]
	\arrow["\psi", from=3-2, to=5-1]
	\arrow["{{\u{\ps T}\Theta}}"', from=3-2, to=5-2]
	\arrow["{{\Theta'}}"', from=3-4, to=6-4]
	\arrow["\psi"', from=4-1, to=5-1]
	\arrow["{{\Delta'}}"', from=4-3, to=3-4]
	\arrow["{{\Theta'}}", from=4-3, to=5-3]
	\arrow["\mu", from=4-5, to=3-4]
	\arrow["{{\Theta'}}", from=4-5, to=7-5]
	\arrow["{{\u{\ps T}\Theta}}"', from=5-1, to=7-1]
	\arrow["\phi", from=5-2, to=5-3]
	\arrow["\psi"', from=5-2, to=7-1]
	\arrow["\psi", from=5-3, to=6-3]
	\arrow["{{\u{\ps T}\Theta'}}", from=6-3, to=7-3]
	\arrow["\mu", from=6-4, to=9-4]
	\arrow["\phi", from=7-1, to=6-3]
	\arrow["{{\u{\ps T}\Theta'}}"', from=7-1, to=8-1]
	\arrow["\phi", from=7-3, to=8-3]
	\arrow["\mu", from=7-5, to=6-4]
	\arrow["\psi", from=7-5, to=11-5]
	\arrow["\phi", from=8-1, to=7-3]
	\arrow["{{\u{\ps T}\psi}}"', from=8-1, to=9-1]
	\arrow["\phi", from=8-1, to=9-2]
	\arrow[""{name=0, anchor=center, inner sep=0}, "{{\Delta''}}", from=8-3, to=9-4]
	\arrow["\phi"', from=9-1, to=11-1]
	\arrow["\phi", from=9-2, to=8-3]
	\arrow[""{name=1, anchor=center, inner sep=0}, "{{\Delta''}}", from=9-2, to=10-3]
	\arrow["{{\u{\ps T}\psi}}", from=9-2, to=11-1]
	\arrow["\psi", from=9-4, to=10-3]
	\arrow["{{(**)}}"{description}, draw=none, from=9-4, to=11-5]
	\arrow["{{\u{\ps T}\psi}}", from=10-3, to=11-3]
	\arrow["{{\Delta''}}"', from=11-1, to=11-3]
	\arrow["\mu", from=11-5, to=11-3]
	\arrow["{{(*)}}"{description}, draw=none, from=1, to=0]
\end{tikzcd}\]
where each subdiagram commutes either trivially, or by the axioms making $\braket{f, \Theta}$ and $\braket{f',\Theta'}$ colax morphisms, or by the axioms of \Cref{def:allowing-lax-algebras}. In particular, subdiagrams $(*)$ and $(**)$ commute since $\psi_{\u{\ps T} f'_\diamond , \u{\ps T} f_\diamond} =  \phi^{-1}_{\ps T f'_\diamond , \u{\ps T} f_\diamond} $, where as usual we identify $\u{\ps T} f_\diamond$ with the representable $(\ps T f)_\diamond$ via the isomorphism $\delta^{\ps T}_f$ (and similarly for $f'$), together with the fact that $\mu^{\u{\ps T}}$ is a pseudonatural transformation between pseudofunctors when restricted to representables.
\end{proof}

Finally, we give a proof of \Cref{prop:lokes-allow-for-lax-algebras}: that is,
we show that the skew monads obtained by \Cref{thm:loke-extend} and
\Cref{cor:quotients-extend} allow for lax algebras. Fix for the rest of this
section a monad $\monad{T}$ on $\Set$, and let $\monad{\ps L}$ be the
pseudomonad on $\CAT$ obtained from it by left oplax Kan extension. Consider the
left skew monad $\monad{\u{\ps L}}$ on $\PROF$. We split the proof in
\Cref{prop:lokes-allow-for-lax-algebras-0,prop:lokes-allow-for-lax-algebras-1,prop:lokes-allow-for-lax-algebras-2,prop:lokes-allow-for-lax-algebras-3}
--- the latter two we have stated and proved in \Cref{app:A4} --- (also together
with \Cref{sch:quotients}).

\begin{proposition}\label{prop:lokes-allow-for-lax-algebras-0}
	The transformation $\delta^{\ps L} \colon \u{\ps L}(-)\Rightarrow (\ps L -)_\diamond$ is pseudonatural.
\end{proposition}
\begin{proof}
	Let $f \colon C \to D$ be a functor and consider the canonical natural transformation $\delta^{\ps L}_f \colon \u{\ps L}f_\diamond \Rightarrow (\ps L f)_\diamond$. Let $D \xrightarrow{d_0} R \xleftarrow{d_1} C$ be the representation of $f_\diamond$ as a two-sided codiscrete cofibration. Fix $(X, h, \nu)$ in $\ps L D$ and $(Y, k, \xi)$ in $\ps L C$.

	An element of $\u{\ps L}f_\diamond ( (X,h,\nu), (Y,k,\xi))$ can be identified with a suitable equivalence class of morphisms $(p, \gamma) \colon (X, d_0 h, \nu ) \to (Y, d_1k, \xi)$ in $\ps L R$. Note that the component of $\gamma \colon d_0 h p \Rightarrow d_1 k$ at each $y\in Y$ is given by an $R$-arrow $d_0hp(y) \to d_1k(y)$, which we can identify with a $D$-arrow $hp(y) \to fk(y)$. Since we have that
	\[(\ps L f)_\diamond ((X,h,\nu), (Y,k,\xi)) \cong \ps L D((X,h,\nu), (Y, fk, \xi)),\]
	the map $\delta^{\ps L}_f$ is defined simply by seeing the pair $(p,\gamma)$ as an $\ps L D$-arrow $(X,h,\nu) \to (Y,fk,\xi)$, and thus it is clearly a bijection.
\end{proof}

\begin{proposition}\label{prop:lokes-allow-for-lax-algebras-1}
	The transformation $\omega^{\ps L} \colon \u{\ps L}^2 \Rightarrow \u{(\ps L ^2)}$ is pseudonatural.
\end{proposition}
\begin{proof}
	Fix a profunctor $F \colon C \pro D$, identified with a two-sided codiscrete cofibration $D \xrightarrow{i_0} R_F \xleftarrow{i_1} C$. By construction, $\u{\ps L} F$ is the profunctor $(\ps L d_0)^\diamond (\ps L d_1)_\diamond$, which we can also represent as a two-sided codiscrete cofibration $\ps L D \xrightarrow{j_0} R_{\u{\ps L}F} \xleftarrow{j_1} \ps L C$. Thus, $\u{\ps L}^2 F$ is the profunctor $(\ps L j_0)^\diamond (\ps L j_1)_\diamond$, while $\u{(\ps L^2)}F$ is the profunctor $(\ps L^2 i_0)^\diamond (\ps L^2i_1)_\diamond$.

	Fix $(X,H, \nu)$ in $\ps L^2 D$ and $(Y,K,\xi)$ in $\ps L^2 C$. An element of $\u{\ps L}^2F((X,H,\nu), (Y,K, \xi))$ is given by a suitable equivalence class of $\ps L R_{\u{\ps L}F}$-arrows $(X, j_0 H , \nu) \to (Y, j_1 K , \xi)$, while an element of $\u{(\ps L^2)}F((X,H,\nu), (Y,K, \xi))$ is given by a suitable equivalence class of $\ps L^2 R_{F}$-arrows $(X, \ps L i_0 H, \nu) \to (Y, \ps L i_1 K , \xi)$. Fix then any function $p \colon Y \to X$. A natural transformation $\gamma \colon j_0 H p \Rightarrow j_1 K$ is defined by a family of $ R_{\u{\ps L}F}$-arrows $\gamma_y \colon j_0Hp(y) \to j_1K(y)$, which we can identify with elements in $\u{\ps L}F(Hp(y), K(y))$. By direct inspection, we see that there is a bijection
	\[ \u{\ps L}F(Hp(y), K(y)) \cong \ps L R_{F}(\ps L i_0 H p(y) , \ps L i_1 K(y))\]
	through which we can identify each $\gamma_y$ with a $\ps L R_F$-arrow $\ps L i_0 Hp(y) \to \ps L i_1K(y)$, and hence $\gamma$ with a natural transformation $\ps L i_0 Hp \Rightarrow \ps L i_1K$. This bijective correspondence between transformations $j_0 H p \Rightarrow j_1 K$ and $\ps L i_0 Hp \Rightarrow \ps L i_1K$ lifts to a bijection between $\u{\ps L}^2F((X,H,\nu), (Y,K, \xi))$ and $\u{(\ps L^2)}F((X,H,\nu), (Y,K, \xi))$ which is precisely the map $\omega^{\ps L}_F$; thus, the transformation $\omega^{\ps L}$ is pseudonatural.
\end{proof}

Before concluding the proof of \Cref{prop:lokes-allow-for-lax-algebras}, we here describe the lax associator $\psi_{G,f_\diamond} \colon \u{\ps L} G \circ \u{\ps L} f_\diamond\Rightarrow \u{\ps L}(G \circ f_\diamond)$ for $\u{\ps L}$ explicitly, for a functor $f \colon A \to B$ and a profunctor $G \colon B \pro C$, so as to introduce notations that we will use in the next proposition. Throughout, we identify $f_\diamond$ and $G$ with the two-sided codiscrete cofibrations $B \xrightarrow{i_0} R_f \xleftarrow{i_1} A$ and $C \xrightarrow{j_0} R_G \xleftarrow{j_1} B$ respectively, and similarly $G\circ f_\diamond$ with $C \xrightarrow{d_0} R_{Gf_\diamond} \xleftarrow{d_1} A$. Fix $(X, h, \nu)$ in $\ps L C$ and $(Y, k, \xi)$ in $\ps L A$.

An element of $(\u{\ps L}G\circ \u{\ps L}f_\diamond) ( ( X, h, \nu), (Y, k, \xi))$ is given by an element of the coend
\[ \int^{(Z, l, \chi) \in \ps L B} \u{\ps L}G ((X,h,\nu), (Z,l,\chi)) \times \u{\ps L}f_\diamond ((Z,l,\chi), (Y,k,\xi)) \]
and hence by the equivalence class of a triple of:
\begin{itemize}
	\item an object $(Z,l,\chi)$ in $\ps L B$,
	\item a $\ps L R_G$-arrow $(p,\gamma) \colon (X, j_0h, \nu) \to (Z,j_1l,\chi)$, and
	\item a $\ps L R_{f_\diamond}$-arrow $(p',\gamma') \colon (Z,i_0l,\chi) \to (Y,i_1k, \nu)$,
\end{itemize}
identified up to coherent zig-zags in $\ps L B$, $\ps L R_G$, and $\ps L R_{f_\diamond}$. Note that, for each $z \in Z$, the component $\gamma_z \colon j_0 h p(z) \to j_1 l(z)$ in $R_G$ can be identified with an element of $G(hp(z), l(z))$; similarly, for each $y \in Y$, the component $\gamma'_y \colon i_0 l p'(y)\to i_1k(y)$ in $R_{f_\diamond}$ can be identified with a $B$-arrow $lp'(y) \to fk(y)$. Thus, the transformation $\psi_{G,f_\diamond}$ maps an equivalence class as above to the element of $\u{\ps L}(Gf_\diamond)((X,h,\nu), (Y,k,\xi))$ represented by the $\ps L R_{Gf_\diamond}$-arrow $(X,d_0h,\nu) \to (Y, d_1k, \xi)$ defined by:
\begin{itemize}
\item the function $pp' \colon Y \to X$, and
\item for each $y\in Y$, the $R_{Gf_\diamond}$-arrow $d_0h pp'(y) \to d_1k(y)$ corresponding to the element of $G(hpp'(y),f k(y))$ obtained by pushing $\gamma_{p'(y)}$ forward along $\gamma'_y$.
\end{itemize}

\begin{proposition}\label{prop:lokes-allow-for-lax-algebras-4}
There exists a natural family of natural transformations $\phi_{G,f_\diamond} \colon \u{\ps L}(G \circ f_\diamond) \Rightarrow \u{\ps L} G \circ \u{\ps L} f_\diamond$, for any functor $f \colon A \to B$ and any profunctor $G \colon B \pro C$, such that
    \begin{enumerate}
      \item the diagram
      \[\begin{tikzcd}
	{\u{\ps L}(G \circ (f'f)_\diamond)} & {\u{\ps L}G \circ \u{\ps L}(f' f)_\diamond} & {\u{\ps L}G \circ \u{\ps L}(f'_\diamond \circ  f_\diamond)} \\
	{\u{\ps L}(G \circ f'_\diamond\circ f_\diamond)} & {\u{\ps L}(G \circ f'_\diamond) \circ \u{\ps L}f_\diamond} & {\u{\ps L}G \circ \u {\ps L}  f'_\diamond \circ \u{\ps L}f_\diamond}
	\arrow["{{\phi_{G, (f'f)_\diamond}}}", from=1-1, to=1-2]
	\arrow["\cong"', from=1-1, to=2-1]
	\arrow["\cong", from=1-2, to=1-3]
	\arrow["{{\u{\ps L}G* \phi_{f'_\diamond, f_\diamond}}}", from=1-3, to=2-3]
	\arrow["{{\phi_{G\circ f'_\diamond, f_\diamond}}}"', from=2-1, to=2-2]
	\arrow["{{\phi_{G,f'_\diamond}*\u{\ps L }f_\diamond}}"', from=2-2, to=2-3]
    \end{tikzcd}\]
      commutes, and
      \item $\phi_{f'_\diamond,f_\diamond} = \psi_{f'_\diamond,f_\diamond}^{-1}$, where $\psi_{f'_\diamond,f_\diamond} \colon \u{\ps L} f'_\diamond \circ \u{\ps L} f_\diamond\Rightarrow \u{\ps L}(f'_\diamond \circ f_\diamond)$ is the lax associator for $\u{\ps L}$
    \end{enumerate}
    for each pair of functors $f\colon A \to B, f'\colon B \to C$ and each profunctor $G \colon C \pro D$.
\end{proposition}
\begin{proof}
With the above notations and conventions, fix an element of $\u{\ps L}(Gf_\diamond)((X,h,\nu), (Y,k,\xi))$, represented up to coherent zig-zags by an $\ps L R_{Gf_\diamond}$-arrow $(t, \zeta) \colon (X,d_0h,\nu) \to (Y, d_1k, \xi)$. We define the value of $\phi_{G, f_\diamond}$ on $(t,\zeta)$ as the element of $(\u{\ps L}G\circ \u{\ps L}f_\diamond) ( ( X, h, \nu), (Y, k, \xi))$ represented by the triple:
\begin{itemize}
	\item the object $(Y, fk)$ in $\ps L B$,
	\item the $\ps L R_G$-arrow $(t, \zeta) \colon (X, j_0h, \nu) \to (Y, j_1fk, \xi)$, where indeed for each $y \in Y$ we can identify the $R_{Gf_\diamond}$-arrow $\zeta_y \colon d_0ht(y) \to d_1k(y)$ with a $R_{G}$-arrow $j_0ht(y) \to j_1fk(y)$;
	\item the $\ps L R_{f_\diamond}$-arrow $(\id_Y,\id) \colon (Y, i_0 fk, \nu) \to (Y, i_1k,\nu)$ defined by choosing, for each $y\in Y$, the identity $fk(y) = fk(y)$ in $B$ as an $R_{f_\diamond}$-arrow $i_0fk(y) \to i_1k(y)$.
\end{itemize}

By construction, it is evident that $\psi_{G,f_\diamond} \circ \phi_{G,f_\diamond} = 1$. Suppose now that $G = f'_\diamond$ for some functor $f' \colon B \to C$. By lax naturality of $\delta^{\ps L}$, the diagram
\[\begin{tikzcd}[column sep = 40pt]
	{(\ps L f')_\diamond \circ (\ps L f)_\diamond} & {\u{\ps L}f'_\diamond \circ (\ps L f)_\diamond} & {\u{\ps L}f'_\diamond \circ \u{\ps L}f_\diamond} \\
	{(\ps L (f' f))_\diamond} & {\u{\ps L}(f' f)_\diamond} & {\u{\ps L}(f'_\diamond \circ f_\diamond)}
	\arrow["{\delta^{\ps L}_{f'} *(\ps L f)_\diamond}", from=1-1, to=1-2]
	\arrow["\cong"', from=1-1, to=2-1]
	\arrow["{\u{\ps L}f'_\diamond *\delta^{\ps L}_f}", from=1-2, to=1-3]
	\arrow["{\psi_{f'_\diamond, f_\diamond}}", from=1-3, to=2-3]
	\arrow["{\delta_{f' f}^{\ps L}}"', from=2-1, to=2-2]
	\arrow["\cong"', from=2-2, to=2-3]
\end{tikzcd}\]
  commutes: this means that $\psi_{f'_\diamond, f_\diamond}$ is invertible, so that necessarily $\phi_{f'_\diamond, f_\diamond} = \psi_{f'_\diamond,f_\diamond}^{-1}$. The fact that the diagram in (1) commutes is immediate.
\end{proof}

\begin{scholium}\label{sch:quotients}
    The same proofs show how the analogous results also hold for the four quotients $\bb F$, $\bbbeta$, $\bb P$ and $\bb D$ introduced in \Cref{sec:04}, hence concluding the proof of \Cref{prop:lokes-allow-for-lax-algebras}. Indeed, in all four of the previous propositions, all arguments remain true if the arrows of $\ps L C$, for any category $C$, are identified up to almost-everywhere equality.
\end{scholium}

\end{document}